\documentclass{lms}
\usepackage[utf8]{inputenc}
\usepackage[english]{babel}
\usepackage[T1]{fontenc}
\usepackage{amsmath}

\usepackage{amsthm}
\usepackage{amsfonts}
\usepackage{leftidx}
\usepackage{amssymb}
\usepackage{color}
\usepackage{amsthm}
\usepackage{graphicx}
\usepackage[pdftex]{hyperref}
\usepackage[left=2cm,right=2cm,top=2cm,bottom=2cm]{geometry}
  \usepackage[all,color,cmtip]{xy}

\newcounter{counter_thm}
\newcounter{counter_conj-thm}

\newtheorem{thm}{Theorem}[subsection]
\newtheorem{thm_app}{Theorem}[section]
\newtheorem{thm_int}[counter_thm]{Theorem}
\newtheorem{lem}[thm]{Lemma}
\newtheorem{lem_app}[thm_app]{Lemma}

\newcounter{counter_conj-lem}
\newtheorem{conj-lem}[thm]{Conjecture-Lemma}
\newtheorem{prop}[thm]{Proposition}
\newtheorem{prop_app}[thm_app]{Proposition}

\newcounter{counter_conj-prop}
\newtheorem{conj-prop}[counter_conj-prop]{Conjecture-Proposition}

\newtheorem{conj-thm}[counter_conj-thm]{Conjecture-Theorem}
\newtheorem{cor}[thm]{Corollary}
\newtheorem{cor_app}[thm_app]{Corollary}

\newtheorem{defi}[thm]{Definition}
\theoremstyle{remark}
\newtheorem{rem}[thm]{Remark}
\newtheorem{rem_app}[thm_app]{Remark}

\newtheorem{ex}[thm]{Example}

\newcommand{\black}{\color{black}}
\newcommand{\red}{\color{red}}
\DeclareMathOperator{\C}{\mathbb{C}}

\DeclareMathOperator{\BPF}{BPF}

\DeclareMathOperator{\divi}{div}
\DeclareMathOperator{\ord}{ord}

\DeclareMathOperator{\res}{Res}

\DeclareMathOperator{\Id}{Id}
\DeclareMathOperator{\Ker}{Ker}

\DeclareMathOperator{\Spec}{Spec}

\DeclareMathOperator{\Q}{\mathbb{Q}}

\DeclareMathOperator{\psef}{Psef}

\DeclareMathOperator{\PL}{PL}

\DeclareMathOperator{\Alg}{Alg}
\DeclareMathOperator{\PGL}{PGL}
\DeclareMathOperator{\Pic}{Pic}

\DeclareMathOperator{\Pg}{\mathbb{P}}
\DeclareMathOperator{\ch}{ch}

\DeclareMathOperator{\num}{N}
\DeclareMathOperator{\numd}{N}

\DeclareMathOperator{\rat}{ { \sim} }
\DeclareMathOperator{\Hom}{ Hom }

\DeclareMathOperator{\Ch}{ch}

\DeclareMathOperator{\Cc}{BPF}

\DeclareMathOperator{\nef}{Nef}

\DeclareMathOperator{\reldeg}{reldeg}

\DeclareMathOperator{\nefd}{Nef}
\DeclareMathOperator{\IC}{CI}

\DeclareMathOperator{\Int}{Int}

\newcommand{\relmor}[2]{\underset{#2}{\overset{#1}{\to}}}
\newcommand{\relrat}[2]{\underset{#2}{\overset{#1}{\dashrightarrow}}}

\DeclareMathOperator{\actd}{\llcorner}

\DeclareMathOperator{\psefd}{Psef}

\classno{14E05,37A35, 14C25.
}

\title[Degrees of iterates of rational maps]{Degrees of iterates of rational maps on normal projective varieties}

\author{Nguyen-Bac Dang}

\date{\today}

\begin{document}
\maketitle
\begin{abstract} Let $X$ be a normal projective variety defined over an algebraically closed field of arbitrary characteristic. 
We study the sequence of intermediate degrees of the iterates of a dominant rational selfmap of $X$, recovering former results by Dinh, Sibony~\cite{dinh_sibony_une_borne_sup}, and by Truong~\cite{tuyen2}. 
Precisely, we give a new proof of the submultiplicativity properties of these degrees and of their birational invariance.
Our approach exploits intensively positivity properties in the space of numerical cycles of arbitrary codimension. 
In particular, we prove an algebraic version of an inequality first obtained by Xiao~\cite{xiao} and Popovici~\cite{popovici}, which generalizes Siu's inequality (see \cite{trapani}) to algebraic cycles of arbitrary codimension.  This allows us to show that 
the degree of a map is controlled up to a uniform constant by the norm of its action by pull-back
on the space of numerical classes in $X$.
\end{abstract}

\section*{Introduction}

Let $f : X \dashrightarrow X$ be any dominant rational self-map of a normal projective variety $X$ of dimension $n$ defined over an algebraically closed field $\C$ of arbitrary characteristic. 
If $X$ is not normal then one can always consider its normalization. Moreover, if the field is not algebraically closed, then we shall take its algebraic closure.
 
 Given any big and nef (e.g ample) Cartier divisor $H_X$ on $X$, and any integer $ 0 \leqslant i \leqslant n$, one defines the \textit{$i$-th degree of $f$} as the integer:
\[ \deg_{i, H_X}(f) = (\pi_1^* H_X^{n-i} \cdot \pi_2^* H_X^i), \]
where $\pi_1$ and $\pi_2$ are the projections from the normalization of the graph of $f$ in $X \times X$ onto the first and the second factor respectively and where $( \cdot )$ denotes the intersection product on this graph.

Our main theorem can be stated as follows.

\begin{thm_int} \label{thm_int_A} Let $X$ be a normal projective variety of dimension $n$ and let $H_X$ be a big and nef Cartier divisor on $X$. 
\begin{enumerate}
\item[(i)] There is a positive constant $C>0$ 
such that for any dominant rational self-maps $f, g$ on $X$, one has:
\begin{equation*}
\deg_{i,H_X}(f \circ g ) \leqslant C \deg_{i,H_X} (f) \deg_{i,H_X}(g).
\end{equation*}
\item[(ii)] For any big nef Cartier divisor $H_X'$ on $X$, there exists a constant $C>0$ such that for any rational self-map $f$ on $X$, one has:
\begin{equation*}
\dfrac{1}{C} \leqslant \dfrac{\deg_{i,H_X}(f)}{\deg_{i,H_X'} (f)} \leqslant C.
\end{equation*}
\end{enumerate}
\end{thm_int}

Observe that Theorem \ref{thm_int_A}.(ii) implies that the degree growth of $f$ is a birational invariant, in the sense that there is a positive constant $C$ such that for any birational map $g: X' \dashrightarrow X$ with $X'$ projective, and any big nef Cartier divisor $H_{X'}$ on $X'$, one has
\[ \dfrac{1}{C} \leqslant \dfrac{\deg_{i, H_X} (f^p)}{ \deg_{i,H_{X'}} (g^{-1} \circ f^p \circ g)} \leq C,\]
for any $p \in \mathbb{N}$. 
Indeed, by applying Theorem \ref{thm_int_A}.(ii) for the induced action by $f$ on the normalization of the graph of $g$, one deduces that the growth of the degrees on the graph of $g$ and on $X$ and $X'$ are controlled by a strictly positive constant. 
Fekete's lemma and Theorem \ref{thm_int_A}.(i) also imply the existence of the dynamical degree (first introduced in \cite{russakovskii_shiffman} for rational maps of the projective space) as the following quantity:
\[\lambda_i(f):=\lim_{p \rightarrow + \infty} \deg_{i,H_X}(f^p)^{1/p}~.\] 
The independence of $\lambda_i(f)$ under the choice of $H_X$, and its birational invariance are the consequence of Theorem \ref{thm_int_A}.(ii) .
\medskip

When $\C = \mathbb{C}$, Theorem \ref{thm_int_A} was proved by Dinh and Sibony in \cite{dinh_sibony_une_borne_sup}, and further generalized to compact K\"ahler manifolds in \cite{dinh_sibony_reg_currents_entropy}. 
 The core of their argument relied on a procedure of regularization for closed positive currents of any bidegree (\cite[Theorem 1.1]{dinh_sibony_reg_currents_entropy}) and was therefore transcendental in nature.
When $\C$ is a field of characteristic zero, there exists an inclusion of the field $\C$ in $\mathbb{C}$ by Lefschetz principle (\cite{lefschetz_algebraic_geometry}) and Dinh and Sibony's argument  proves that the $i$-th dynamical degree of any rational dominant map is well-defined. 
 Recently, Truong \cite{tuyen1} managed to get around this problem and proved Theorem \ref{thm_int_A} for arbitrary smooth varieties using an appropriate Chow-type moving lemma. 
 He went further in \cite{tuyen2} and obtained Theorem \ref{thm_int_A} for any normal variety in all characteristic by applying de Jong's alteration theorem (\cite{de_jong}). 
 Note however that he had to deal with correspondences since a rational self-map can only be lifted as a correspondence through a general alteration map.
 Our approach avoids this technical difficulty.     

\bigskip

To illustrate our method, let us explain the proof of Theorem \ref{thm_int_A}, when $X$ is smooth, $i=1$ and $f$, $g$ are regular following the method initiated in \cite[Proposition 3.1]{boucksom_favre_jonsson_deggrowth}. 
Recall that a divisor $\alpha$ on $X$ is \textit{pseudo-effective} and one writes $\alpha \geqslant 0$ if for any ample Cartier divisor $H$ on $X$, and any rational $\epsilon >0$, a suitable multiple of the $\mathbb{Q}$-divisor $\alpha + \epsilon H$ is linearly equivalent to an effective one. 

Recall also the fundamental Siu inequality\footnote{this inequality is also referred to as the weak transcendantal holomorphic Morse inequality in \cite{lehmann_xiao}} (\cite{trapani}, \cite[Theorem 2.2.13]{lazarsfeld_positivity_1}, \cite{cutkosky_teissier_pb}) which states:
\begin{equation}\label{int_siu_ineq}
\alpha \leqslant n \dfrac{(\alpha \cdot  \beta^{n-1})}{(\beta^n)} \beta, 
\end{equation}
for any nef divisor $\alpha$, and any big and nef divisor $\beta$. 
 
Since the pullback by a dominant morphism of a big nef divisor remains big and nef, 
we may apply \eqref{int_siu_ineq} to the big nef divisors $\alpha = g^* f^* H_X$ and $\beta = f^* H_X$, and we get
\begin{equation*}
 g^*f^*  H_X \leqslant n \dfrac{\deg_{1,H_X}(f)}{(H_X^n)}  g^* H_X~.
\end{equation*}
Intersecting with the cycle $H_X^{n-1}$ yields the submultiplicativity of the degrees with the constant $C = n/(H_X^n)$. 

We observe that the previous inequality \eqref{int_siu_ineq} can be easily extended to complete intersections by cutting out by suitable ample sections. 
In particular, we get a positive constant $C$ such that for any big nef divisors $\alpha$ and $\beta$, one has:
\begin{equation} \label{int_ineq_siu_gen}
\alpha^i \leqslant C \dfrac{(\alpha^i \cdot \beta^{n-i})}{(\beta^n)} \beta^i.
\end{equation}
Such inequalities have been obtained by Xiao (\cite{xiao}) and Popovici (\cite{popovici}) in the case $\C = \mathbb{C}$. 
Their proof uses the resolution of complex Monge-Ampère equations and yields a constant $C ={ n \choose i}$. 
On the other hand, our proof applies in arbitrary characteristic and in fact to more general classes than complete intersection ones. 
We refer to Theorem \ref{thm_int_Siu} below and the discussion preceding it for more details.
Note however that we only obtain $C = (n-i+1)^i$, far from the expected optimal constant $C = { n \choose i}$ of Popovici.
Once \eqref{int_ineq_siu_gen} is proved, Theorem \ref{thm_int_A} follows by a similar argument as in the case $i = 1$.

\bigskip
Going back to the case where $X$ is a complex smooth projective variety, recall that the degree of $f$ is controlled up to a uniform constant by the norm of the linear operator $f^{\bullet,i}$, induced by pullback on the de Rham cohomology space $H^{2i}_{dR} (X)_\mathbb{R}$ (\cite[Lemma 4]{dinh_sibony_une_borne_sup}).
One way to construct $f^{\bullet,i}$ is to use the Poincar\'e duality isomorphisms $\psi_X: H^{2i}_{dR}(X,\mathbb{R}) \to {H_{2n-2i}}(X,\mathbb{R})$, $\psi_{\Gamma_f} : H^{2i}_{dR}(\Gamma_f,\mathbb{R}) \to {H_{2n-2i}}(\Gamma_f,\mathbb{R})$ where $H_i(X,\mathbb{R})$ denotes the $i$-th simplicial homology group of $X$.
 The operator $f^{\bullet,i}$ is then defined following the commutative diagram below:
\[
\xymatrix{ 
H_{dR}^{2i}(\Gamma_f, \mathbb{R})\ar[r]^-{\psi_{\Gamma_f}}  &  {H_{2n-2i}}(\Gamma_f,\mathbb{R}) \ar[r]^{{\pi_1}_*} & {H_{2n-2i}}(X,\mathbb{R}) \ar[d]^{\psi_{X}^{-1}} \\
H_{dR}^{2i}(X, \mathbb{R}) \ar[rr]_{f^{\bullet,i}} \ar[u]^{\pi_2^*}& &   H_{dR}^{2i}(X,\mathbb{R}),}\]
where $\Gamma_f$ is a desingularization of the graph of $f$ in $X \times X$, and $\pi_1, \pi_2$ are the projections from $\Gamma_f$ onto the first and second factor respectively. 

\medskip
In order to state an analogous result in our setting, we need to find a replacement for the de Rham cohomology group $H^{2i}_{dR}(X)_\mathbb{R}$ and define suitable pullback operators. 
When $X$ is smooth, one natural way to proceed  is to consider the spaces $\numd^i(X)_\mathbb{R}$ of algebraic $\mathbb{R}$-cycles of codimension $i$ modulo numerical equivalence. 
The operator $f^{\bullet, i}$ is then simply given by the composition ${\pi_1}_* \circ \pi_2^* : \numd^i(X)_\mathbb{R} \to \numd^i(X)_\mathbb{R}$. 

When $X$ is singular, then the situation is more subtle because one cannot intersect arbitrary cycle classes in general \footnote{an arbitrary curve can only be intersected with a Cartier divisor, not with a general Weil divisor.}. 
One can consider two natural spaces of numerical cycles $\numd^i(X)_\mathbb{R}$ and $\num_i(X)_\mathbb{R}$ on which pullback operations and pushforward operations by proper morphisms are defined respectively. 
More specifically, the \textit{space of numerical $i$-cycles} $\num_i(X)_\mathbb{R}$ is defined as the group of $\mathbb{R}$-cycles of dimension $i$ modulo the relation $z \equiv 0 $ if and only if $(p^* z \cdot D_1 \cdot \ldots \cdot D_{e+i}) = 0$ for any proper flat surjective map $p : X' \to X$ of relative dimension $e$ and any Cartier divisors $D_j$ on $X'$.
One can prove that $\num_i(X)_\mathbb{R}$ is a finite dimensional vector space and one defines $\numd^i(X)_\mathbb{R}$ as its dual $\Hom(\num_i(X)_\mathbb{R}, \mathbb{R})$.

Note that our presentation differs slightly from Fulton's definition (see Appendix \ref{appendix} for a comparison), but we also recover the main properties of the numerical groups. This approach is more suitable to compare cycles using positivity estimates on complete intersections.

\bigskip
As in the complex case, we are able to construct Poincar\'e duality  maps $\psi_{X} : \numd^i(X)_\mathbb{R} \to \num_{n-i}(X)_\mathbb{R}$ and $\psi_{\Gamma_f} :  \numd^i(\Gamma_f)_\mathbb{R} \to \num_{n-i}(\Gamma_f)_\mathbb{R}$, but they are not necessarily isomorphisms due to the presence of singularities. As a consequence, we are only able to define a linear map $f^{\bullet,i}$ as $  f^{\bullet,i} := {\pi_1}_* \circ \psi_{\Gamma_f} \circ \pi_2^* : \numd^i(X)_\mathbb{R} \to \num_{n-i}(X)_\mathbb{R}$ between two distinct vector spaces. 
Despite this limitation, we prove a result analogous to one of Dinh and Sibony. The next theorem was obtained by Truong for smooth varieties (\cite[Theorem 1.1.(5)]{tuyen2}).

\medskip
\begin{thm_int}\label{thm_int_B} Let $X$ be a normal projective variety of dimension $n$. Fix any norms on $\numd^i(X)_\mathbb{R}$ and $\num_{n-i}(X)_\mathbb{R}$, and denote by $\|\cdot\|$ the induced operator norm on linear maps from $\numd^i(X)_\mathbb{R}$ to $\num_{n-i}(X)_\mathbb{R}$. Then there is a constant $C>0$ such that for any rational selfmap $f: X \dashrightarrow X$, one has:
\begin{equation}
  \dfrac{1}{C} \leqslant \dfrac{|| (f)^{\bullet,i}||}{\deg_{i,H_X}(f)} \leqslant C.
\end{equation} 
\end{thm_int}

\smallskip

Our proof of Theorem \ref{thm_int_B} exploits a natural notion of positive classes in $\numd^i(X)_\mathbb{R}$ combined with a strengthening of \eqref{int_ineq_siu_gen} to these classes that we state below (see Theorem \ref{thm_int_Siu}).

To simplify our exposition, let us suppose again that $X$ is smooth. 
As in codimension $1$, one can define the \textit{pseudo-effective cone} $\psefd^i(X)$ as the closure in $\numd^i(X)_\mathbb{R}$ of the cone generated by effective cycles of codimension $i$. Its dual with respect to the intersection product is the \textit{nef cone} $\nefd^{n-i}(X)$, which however does not behave well when $i \geqslant 2$ (see \cite{debarre_ein_lazarsfeld_voisin}). 
Some alternative notions of positive cycles have been introduced by Fulger and Lehmann in \cite{fulger_lehmann}, among which the notion of basepoint free classes emerges. Basepoint free classes have many good properties such as being both pseudo-effective and nef, being invariant by pull-backs by morphisms and by intersection products, and forming a salient convex cone with non-empty interior. The terminology comes from the fact that the basepoint free classes always have a cycle representing them with intersects any subvariety with the expected dimension. 
 Denote by $\BPF^i(X)$ the cone of basepoint free classes. It is defined as the closure in $\numd^i(X)_\mathbb{R}$ of the cone generated by $\mathbb{R}$-cycles of the form $p_* (D_1 \cdot \ldots \cdot D_{e+i})$ where $D_j$ are ample Cartier $\mathbb{R}$-divisors and $p : X' \to X$ is a flat surjective proper morphism of relative dimension $e$. 

For basepoint free classes, we are able to prove the following generalization of \eqref{int_ineq_siu_gen}. 

\begin{thm_int} \label{thm_int_Siu} Let $X$ be a normal projective variety of dimension $n$. Then there exists a constant $C>0$ such that for any basepoint free class $\alpha \in \BPF^i(X)$, for any big nef divisor $\beta$, one has in $\numd^i(X)_\mathbb{R}$:
  \begin{equation} \label{int_ineq_siu_pliante}
  \alpha \leqslant C \dfrac{(\alpha \cdot \beta^{n-i})}{(\beta^n)} \times \beta^i. 
\end{equation}  
 
 \end{thm_int}
Theorem \ref{thm_int_B} follows from~\eqref{int_ineq_siu_pliante} by observing that  $f^{\bullet,i} \BPF^i(X) \subset \psefd^i(X)$, so that the operator norm $||f^{\bullet,i}||$ can be computed by evaluating $f^{\bullet,i}$ only on basepoint free classes.

In the singular case, the proof of Theorem \ref{thm_int_B} is completely similar but the spaces  $\numd^i(X)_\mathbb{R}$ and $\num_{n-i}(X)_\mathbb{R}$ are not necessarily isomorphic in general. 
As a consequence, several dual notions of positivity appear in $\numd^i(X)_\mathbb{R}$ and $\num_i(X)_\mathbb{R}$
that make the arguments more technical.

\bigskip
Finally, using the techniques developed in this paper, we give a new proof of the product formula of Dinh, Nguyen, Truong (\cite[Theorem 1.1]{dinh_nguyen_2011},\cite[Theorem 1.1]{dinh_nguyen_truong}) which they proved when $\C = \mathbb{C}$ and which was later generalized by Truong (\cite[Theorem 1.1.(4)]{tuyen2}) to normal projective varieties over any field.

The setup is as follows. 
Let $q : X \to Y$ be any proper surjective  morphism between normal projective varieties, and
fix two big and nef divisors  $H_X$, $H_Y$  on $X$ and $Y$ respectively.   
Consider  two dominant rational self-maps $f: X \dashrightarrow X$, $g : Y \dashrightarrow Y $, which are semi-conjugated by 
$q$, i.e. which satisfy $q \circ f = g \circ q$. To simplify notation we shall write
$X/_qY \relrat{f}{g} X/_q Y$ when these assumptions hold true.

Recall that the \textit{$i$-th relative degree} of $X/_q Y \relrat{f}{g} X/_q Y$ is given by the intersection product
\begin{equation*}
\reldeg_{i} (f) := (\pi_1^*( H_X^{\dim X -\dim Y - i}  \cdot q^* H_Y^{\dim Y}) \cdot \pi_2^* H_X^i),
\end{equation*}
where $\pi_1$ and $\pi_2$ are the projections from the graph of $f$ in $X \times X$ onto the first and the second component respectively. 
One can show a relative version of Theorem \ref{thm_int_A}  (see Theorem \ref{thm_sub_multipicativity}), and define as in the absolute case, the \textit{$i$-th relative dynamical degree} $\lambda_i(f, X/Y)$ 
as the limit $\lim_{p\rightarrow +\infty} \reldeg_i(f^p)^{1/p}$. 
It is also a birational invariant in the sense that if $\varphi : X' \dashrightarrow X$, $\psi: Y' \dashrightarrow Y$ such that $q' = \psi^{-1}\circ q\circ \varphi$ is regular, then $\lambda_i( \varphi^{-1} \circ f \circ  \varphi , X'/Y') = \lambda_i(f, X/Y)$, and does not depend on the choices of $H_X$ and $H_Y$. 
When $q : X\dashrightarrow Y$ is merely rational and dominant, then we define (see Section \ref{section_semi_conj_rational}) the $i$-th relative degree of $f$ by replacing $X$ with the normalization of graph of $q $. 
We prove the following theorem.
\begin{thm_int} \label{thm_int_C}
Let $X,Y$ be normal projective varieties.
For any dominant 
rational self-maps $f: X \dashrightarrow X$, $g : Y \dashrightarrow Y $ which are  semi-conjugated by 
a dominant rational map $q: X \dashrightarrow Y$, we have
\begin{equation}
\lambda_i(f)  = \max_{\max(0, i- l ) \leqslant j \leqslant \min(i,e)} ( \lambda_{i-j}(g) \lambda_{j} (f, X/Y) )~.
\end{equation}
\end{thm_int}
%

Our proof follows closely Dinh and Nguyen's method from \cite{dinh_nguyen_2011} and relies on a fundamental inequality (see Corollary \ref{cor_ku_pullback} below) which follows from K\"unneth formula at least when $\C = \mathbb{C}$.
To state it precisely, consider $\pi : X' \to X$ a surjective generically finite morphism and $q : X \to Y$ a surjective morphism where $X'$, $X$ and $Y$ are normal projective varieties such that   $n= \dim X=\dim X'$ and such that $l = \dim Y$. 
We prove that for any basepoint free classes $\alpha \in \Cc^i(X')$ and $\beta\in \Cc^{n-i}(X')$, one has:
\begin{equation} \label{ineq_int_ku}
(\beta \cdot \alpha) \leqslant C \sum_{ \max(0 , i-l) \leq j \leq \min(i,e)} U_j(\alpha) \times ( \beta \cdot \pi^* (q^* H_Y^{i-j} \cdot H_X^j)),
\end{equation}  
where $H_Y$ and $H_X $ are big and nef divisors on $Y$ and $X$ respectively, 
and $U_j(\alpha)$ is the intersection product given by $ U_j(\alpha) = (\pi^*(q^*H_Y^{l-i+j} \cdot H_X^{e-j}) \cdot \alpha)$.

In the singular case, Truong has obtained this inequality using Chow's moving intersection lemma. 
We replace this argument by a suitable use of Siu's inequality and Theorem \ref{thm_int_Siu} in order to prove a positivity property for a class given by the difference between a basepoint free class in $X' \times X'$ and the fundamental class of the diagonal of $X'$ in $X' \times X'$ (see Theorem \ref{thm_ku_pullback}).
Inequality \eqref{ineq_int_ku} is a weaker version of \cite[Proposition 2.3]{dinh_nguyen_2011} proved by Dinh-Nguyen when $Y$ is a complex projective variety, and was extended to a field arbitrary characteristic by Truong when $Y$ is smooth (\cite[Lemma 4.1]{tuyen2}).

\bigskip

\subsection*{Organization of the paper}

In the first Sections \ref{section_chow} and \ref{section_num}, we review the background on the Chow groups and recall the definitions of the spaces of numerical cycles and provide their basic properties. 
In \S\ref{section_positivity}, we discuss the various notions of positivity of cycles and prove Theorem \ref{thm_int_Siu}. 
In \S\ref{section_relative}, we define relative numerical cycles and canonical morphisms which are the analogous to the Poincaré morphisms $\psi_X$ in a relative setting. 
In \S\ref{section_applications}, we prove Theorem \ref{thm_int_A}, Theorem \ref{thm_int_B} and Theorem \ref{thm_int_C}. 
Finally we give an alternate proof of Dinh-Sibony's theorem in the K\"ahler case (\cite[Proposition 6]{dinh_sibony_une_borne_sup}) in \S\ref{section_kahler} using Popovici \cite{popovici} and Xiao's inequality \cite{xiao}. Note that these inequalities allow us to avoid regularization techniques of closed positive currents but rely on a deep theorem of Yau. 
In Section \ref{appendix}, we prove that our presentation and Fulton's definition of numerical cycles are equivalent, hence proving that any numerical cycles can be pulled back by a flat morphism. 
\bigskip

\begin{acknowledgements}
Firstly, I would like to thank my advisor C. Favre for his patience and our countless discussions on this subject. I thank also S. Boucksom for some helpful discussions and for pointing out the right argument for the appendix, S. Cantat, L. Fantini, M. Fulger, T. Truong,  B. Lehmann, R. Mboro and J. Xie for their precious comments on my previous drafts and for providing me some references.
The author is supported by the ERC-starting grant project "Nonarcomp" no.307856, and is supported by ANR project ``Lambda'' ANR-13-BS01-0002
\end{acknowledgements}
%

\bigskip

\section{Chow group} \label{section_chow}

\subsection{General facts}

Let $X$ be a normal projective variety of dimension $n$ defined over an algebraically closed field $\C$ of arbitrary characteristic.

The space of cycles $Z_i(X)$ is the free abelian group generated by irreducible subvarieties of $X$ of dimension $i$,   
 and $Z_i(X)_\mathbb{Q}$, $Z_i(X)_\mathbb{R}$ will denote the tensor products $Z_i(X) \otimes_\mathbb{Z} \mathbb{Q} $ and $Z_i(X) \otimes_\mathbb{Z} \mathbb{R}$.

\smallskip

Let $q: X \to Y$ be a morphism where $Y$ is a normal projective variety. Since $X$ and $Y$ are respectively projective, the map $q$ is proper. Following \cite{fulton}, we define the proper pushforward of the  cycle $[V] \in Z_i(X)$ as the element of $Z_i(Y)$ given by: 
\begin{equation*}
q_* [V] = \left \lbrace \begin{array}{lll}
0   & \text{if} &\dim (q(V)) < \dim V \\
 \left [ \C(\eta) : \C(q(\eta)) \right ] \times [q(V)] & \text{if}& \dim V = \dim(q(V)), 
\end{array} \right .
\end{equation*} 
where $V$ is an irreducible subvariety of $X$ of dimension $i$, $\eta$ is the generic point of $V$ and $\C(\eta)$, $\C(q(\eta))$ are the residue fields of the local rings $\mathcal{O}_\eta$ and $\mathcal{O}_{q(\eta)}$ respectively.
We extend this map by linearity and obtain a morphism of abelian groups $q_* : Z_i(X) \to Z_i(Y)$.

\smallskip

Let $C$ be any closed subscheme of $X$ of dimension $i$ and denote by $C_1, \ldots, C_r$ its $i$-dimensional irreducible components. Then $C$ defines a fondamental class $[C]\in Z_i(X)$ by the following formula: 
\begin{equation*}
[C] := \sum_{j=1}^r l_{\mathcal{O}_{C_j, C}}(\mathcal{O}_{C_j,C})[C_j],
\end{equation*}
where $l_{A}(M)$ denotes the length of an $A$-module $M$ (\cite[section 2.4]{eisenbud_commutative}).
\medskip

For any flat morphism $q: X \to Y$ of relative dimension $e$ between normal projective varieties, we can define a flat pullback of cycles $q^*: Z_i(Y) \to Z_{i+e}(X)$ (see \cite[section 1.7]{fulton}). If $C$ is any subscheme of $Y$ of dimension $i$, the cycle $q^* [C]$ is by definition the fundamental class of the scheme-theoretic inverse by $q$:
 \begin{equation*}
 q^* [C] := [q^{-1}(C)] \in Z_{i+e}(X).
\end{equation*}  

\medskip

Let $W$ be a subvariety of $X$ of dimension $i+1$ and $\varphi $ be a rational map on $W$. Then we define a cycle on $X$ by:
\begin{equation*}
[\divi(\varphi)]:= \sum \ord_V(\varphi) [V], 
\end{equation*}
where the sum is taken over all irreducible subvarieties $V$ of dimension $i$ of $W \subset X$. 
A cycle $\alpha$ defined this way is rationally equivalent to $0$ and in that case we shall write $\alpha \rat 0$. 
\smallskip

The $i$-th Chow group $A_i(X)$ of $X$ is the quotient of the abelian group $Z_i(X)$ by the free group generated by the cycles that are rationally equivalent to zero. We denote by $A_\bullet(X)$ the abelian group $\oplus A_i(X)$.

We recall now the functorial operations on the Chow group, which result from the intersection theory developped in \cite{fulton}.

\begin{thm} Let $q: X \to Y$ be a morphism between normal projective varieties. Then we have:
\begin{enumerate}
\item[(i)] 
The morphism of abelian groups $q_* : Z_i(X) \to Z_i(Y)$ induces a morphism of abelian groups $q_* : A_i(X) \to A_i(Y)$. 
\item[(ii)] 
If the morphism $q$ is flat of relative dimension $e$, then the morphism $q^* : Z_i(Y) \to Z_{i+e}(X)$ induces a morphism of abelian groups $q^* : A_i(Y) \to A_{i+e}(X)$. 
\end{enumerate}
\end{thm}

Assertion $(i)$ is proved in \cite[Theorem 1.4]{fulton} and assertion $(ii)$ is given in \cite[Theorem 1.7]{fulton}.

\begin{rem} Let $q: X \to Y$ is a flat morphism of normal projective varieties. Suppose $\alpha \in A_{i}(Y)$ is represented by an effective cycle $\alpha \rat \sum n_j [V_j]$ where the  $n_j$ are positive integers. Then $q^*\alpha$ is also represented by an effective cycle. 
\end{rem}

Any cycle $\alpha \in Z_{0}(X)_\mathbb{Z}$ is of the form $\sum n_j [p_j]$ with $p_j \in X(\C)$ and $n_j \in \mathbb{Z}$. We define the degree of $\alpha$ to be $\deg(\alpha) := \sum n_j $ and we shall write:
\begin{equation*}
(\alpha ):= \deg(\alpha) = \sum n_j.
\end{equation*}

The morphism of abelian groups $\deg : Z_{0}(X)_\mathbb{Z} \to \mathbb{Z}$ induces a morphism of abelian groups $\deg: A_0(X) \to \mathbb{Z}$.

\subsection{Intersection with Cartier divisors}

Let $X$ be a normal projective variety and $D$ be a Cartier divisor on $X$. Let $V$ be a subvariety of of dimension $i$ in $X$ and denote by $j : V \hookrightarrow X$ the inclusion of $V$ in $X$. We define the intersection of $D$ with $[V]$ as the class:
\begin{equation*}
D \cdot [V] := j_* [D'] \in A_{i-1}(X),
\end{equation*}
where $D'$ is a Cartier divisor on $V$ such that the line bundles $j^* \mathcal{O}_X(D) $ and $ \mathcal{O}_V(D')$ are isomorphic.
Observe that $D'$ exists since the exact sequence $\xymatrix{ 0 \ar[r] & \mathcal{O}^*_V \ar[r] & \mathcal{M}_V^* \ar[r] & \mathcal{M}_V^* / \mathcal{O}_V^*\ar[r] & 0}$ induces a  surjective map from the divisor subgroups $H^0(V, \mathcal{M}_V^* / \mathcal{O}_V^*) $ of $V$ onto the Picard group $\Pic(V) = H^1(V, \mathcal{O}_V^*)$ where $\mathcal{M}_V^*$ is the sheaf of non-zero rational functions on $V$.

We extend this map by linearity into a morphism of abelian groups $D \cdot :Z_i(X) \to A_{i-1}(X)$.  

\begin{thm} \label{thm_prop_intersection_diviseurs} Let $X$ be a normal projective variety and $D$ be a Cartier divisor on $X$. 
The map $D \cdot : Z_i(X) \to A_{i-1}(X)$ induces a morphism of abelian groups $D \cdot : A_i(X) \to A_{i-1}(X)$. 
Moreover, the following properties are satisfied:
\begin{enumerate} 
\item For all Cartier divisors $D$ and $D'$ on $X$, for all class $\alpha \in A_i(X)$, we have: 
\begin{equation*}
(D' + D) \cdot \alpha = D' \cdot \alpha + D \cdot \alpha.
\end{equation*}
\item (Projection formula) Let $q: X \to Y$ be a morphism between normal projective varieties. Then for all class $\beta \in A_i(X)$ and all Cartier divisor $D$ on $Y$, we have in $A_{i-1}(Y)$:
\begin{equation*}
q_* (q^*D \cdot \beta) = D \cdot q_*(\beta). 
\end{equation*}
\end{enumerate}
\end{thm}

\subsection{Characteristic classes} \label{section_characteristic}

\begin{defi} Let $X$ be a normal projective variety of dimension $n$ and $L$ be a line bundle on $X$. There exists a Cartier divisor $D$ on $X$ such that the line bundles $L$ and $\mathcal{O}_X(D)$ are isomorphic. We define the first Chern class of $L$ as:
\begin{equation*}
c_1(L) := [D] \in A_{n-1}(X).
\end{equation*}
\end{defi}

\begin{defi} For all normal projective varieties $X$, the group $\IC^i(X)$ is the free group generated by elements of the form $D_1 \cdot \ldots \cdot D_i$ where $D_1$, $\ldots$, $D_i$ are Cartier divisors on $X$.
\end{defi}

\begin{defi} Let $X$ be a normal projective variety and $E$ be a vector bundle of rank $e+1$ on $X$. 
Given any vector bundle $E$ on $X$, we shall denote by $\Pg(E)$ the projective bundle of hyperplanes in $E$ following the convention of Grothendieck.
Let $p $ be the projection from $\Pg(E^*)$ to $X$ and $\xi = c_1(\mathcal{O}_{\Pg(E^*)} (1))$. 
We define the $i$-th Segré class $s_i (E) $ as the morphism $s_i(E) \actd \cdot: A_\bullet(X) \to A_{\bullet-i}(X)$ given by:
\begin{equation} \label{def_segre_class}
s_i(E) \actd \alpha := p_* (\xi^{e+i} \cdot p^*\alpha). 
\end{equation}
\end{defi}

\begin{rem}\label{rem_segre_image_diviseur} When $X$ is smooth of dimension $n$, we can define an intersection product on the Chow groups $ A_{i}(X) \times A_{l}(X) \to A_{n-i-l}(X)$ (see \cite[Definition 8.1.1]{fulton}) which is compatible with the intersection with Cartier divisors and satisfies the projection formula (see \cite[Example 8.1.7]{fulton}).
Applying the projection formula to \eqref{def_segre_class}, we get 
\begin{equation*}
 \  s_i(E) \actd \alpha = p_* (\xi^{e+i})\cdot \alpha,
\end{equation*}
so that  $s_i(E)$ is represented by an element in $A_{n-i}(X)$. To simplify we shall also denote $s_i(E)$ this element.

\end{rem}

As Segré classes of vector bundles are operators on the Chow groups $A_\bullet(X)$, the composition of such operators defines a product. 

 \begin{thm} \label{thm_segre}(cf \cite[Proposition 3.1]{fulton}) Let $q: X \to Y$ be a morphism between normal projective varieties. 
 For any vector bundle $E$ and $F$ on $Y$, the following properties hold. 
\begin{enumerate}
\item[(i)] For all $\alpha \in A_i(Y)$ and all $j< 0$, we have $s_j(E) \actd \alpha = 0$. 
\item[(ii)] For all $\alpha \in A_i(Y)$, we have $s_0(E) \actd \alpha = \alpha$. 
\item[(iii)] For all integers $j,m$, we have $s_j(E) \actd( s_m(F) \actd \alpha) = s_m(F) \actd (s_j(E) \actd \alpha)$.
\item[(iv)] (Projection formula) For all $\beta \in A_i(X)$ and any integer $j$, we have $q_*(s_j(q^* E) \actd \beta) = s_j(E) \actd q_* \beta$. 
 \item[(v)] If the morphism $q: X \to Y$ is flat, then for all $\alpha \in A_i(Y)$ and any integer $j$, we have $s_j(q^*E) \actd q^* \alpha = q^* (s_j(E) \actd \alpha))$. 
\end{enumerate}
\end{thm}

The $j$-th Chern class $c_j(E)$ of a vector bundle $E$ on $X$ is an operator $c_j(E) : A_{\bullet}(X) \to A_{\bullet -j}$ defined formally as the coefficients in the inverse power series:
\[ (1+  s_1(E) t + s_2(E)t^2 + \ldots   )^{-1} = 1 + c_1(E) t + \ldots + c_{r+1}(E)t^{r+1}. \]
A direct computation yields for example
 $c_1(E) = - s_1(E)$, $c_2(E)  =  (s_1(E)^2 - s_2(E))$.

\begin{defi} \label{defi_operational_chow}Let $X$ be a normal projective variety. 
The abelian group $A^i(X)$ is the subgroup of $\Hom(A_{\bullet}(X), A_{\bullet-i}(X))$ generated by product of Chern classes $c_{i_1} (E_1) \cdot \ldots \cdot c_{i_p}(E_p)$ where $i_1$, $\ldots$, $i_p$ are integers satisfying $i_1 + \ldots + i_p = i$ and where $E_1, \ldots, E_p$ are vector bundles over $X$. We denote by $A^\bullet(X)$ the group $\oplus A^i(X)$.
\end{defi}
Observe that by definition, $A^i(X)$ contains the image of $\IC^i(X)$.

\bigskip

Recall that the Grothendieck group $K^0(X)$ is the free group generated by vector bundles on $X$ quotiented by the subgroup generated by relations of the form $[E_1] + [E_3] - [E_2]$ where there is an exact sequence of vector bundles:
\begin{equation*}
\xymatrix{0 \ar[r]& E_1 \ar[r] & E_2 \ar[r] & E_3 \ar[r] & 0  }.
\end{equation*}
Moreover, the group $K^0(X)$ has a structure of rings given by the tensor product of vector bundles.

\medskip

Recall also that the Chern character is the unique morphism of rings $\ch : (K^0(X), +, \otimes) \to (A^\bullet(X), + , \cdot )$ satisfying the following properties (see \cite[Example 3.2.3]{fulton}).
\begin{enumerate}
\item If $L$ is a line bundle on $X$, then one has:
$$\ch(L)= \sum_{i \geqslant 0} \dfrac{c_1(L)^i}{i!}.$$
\item For any morphism $q: X' \to X$ and any vector bundle $E$ on $X$, we have $q^* \ch(E) = \ch(q^*E)$.
\end{enumerate} 

\medskip
For any vector bundle $E$ on $X$, we will denote by $\ch_i(E)$ the term in $A^i(X)$ of $\ch(E)$.  

\medskip

We recall Grothendieck-Riemann-Roch's theorem for smooth varieties.
\begin{thm} \label{thm_grr} (see \cite[Corollary 18.3.2]{fulton}) Let $X$ be a smooth variety. Then the Chern character induces an isomorphism:
\begin{equation*}
\ch \actd [X] : E \in K^0(X)\otimes \mathbb{Q} \rightarrow \ch(E) \actd [X] \in A_\bullet(X)\otimes \mathbb{Q}.
\end{equation*}
\end{thm}

\medskip

We also recall the definition of Schur polynomials.

\begin{defi} \label{defi_schur} Consider a vector bundle $E$ of rank e on $X$.
Fix two integers $e,i$ and a decreasing partition $\lambda =(\lambda_1, \ldots, \lambda_i)$ of $i$ with terms lower or equal than $e$.
The Schur class $s_\lambda(E)$ is the class given by:
\begin{equation*}
s_\lambda(E) = \left | \begin{array}{llll}
c_{\lambda_1}(E) & c_{\lambda_1 + 1}(E) & \ldots & c_{\lambda_1 + i -1}(E) \\
c_{\lambda_2 -1}(E) & c_{\lambda_2} (E)&\ldots  &  c_{\lambda_2 +i - 2}(E)\\
 & & \ldots \\
c_{\lambda_i -i +1}(E) & c_{\lambda_i - i +2}(E) & \ldots & c_{\lambda_i}(E)
\end{array}  \right | .
\end{equation*}
If $E$ is a vector bundle of rank $e$ on $X$, then the Schur class $s_\lambda(E) \in A^i(X)$ is the Schur polynomial in the variables given by the Chern classes $c_1(E), \ldots , c_e(E)$. 
\end{defi}

When the vector bundle $E$ is globally generated, then the Schur classes can be interpreted as degeneracy loci (see \cite[Example 8.3.6]{lazarsfeld_positivity_1}).


\section{Space of numerical cycles} \label{section_num}

\subsection{Definitions} \label{def_num_classes}

In all this section, $X, Y,X_1,X_2,X_3 $ and $X'$ are  normal projective varieties and $X$ is of dimension $n$.
 Two cycles $\alpha$ and $\beta$ in $Z_i(X)$ are said to be numerically equivalent and we will denote by $\alpha \equiv \beta$ if for all flat morphisms $p_1 : X_1 \to X$ of relative dimension $e$ and all Cartier divisors $D_1 , \ldots , D_{e+i}$ in $X_1$, we have:
\begin{equation*}
 ( D_1 \cdot \ldots \cdot D_{e+i} \cdot q^*\alpha ) = ( D_1 \cdot \ldots \cdot D_{e+i} \cdot q^*\beta) .
\end{equation*}
\black

\begin{defi}
The group of numerical classes of dimension $i$ is the quotient $\num_i(X) = Z_i(X) / \equiv$. 
\end{defi}

By construction, the group $\num_i(X)$ is torsion free and there is a canonical surjective morphism $A_i(X) \to \num_i(X)$ for any integer $i$. 

\begin{rem}
Observe also that for $i = 0$, two cycles are numerically equivalent if and only if they have the same degree. Since smooth points are dense in $X$ (see \cite[Theorem 5.3]{hartshorne}) and are of degree $1$, this proves that the degree realizes the isomorphism $\num_0(X) \simeq \mathbb{Z}$.
\end{rem}

We set $\num_i(X)_\mathbb{Q}$ and $\num_i(X)_\mathbb{R}$ the two vector spaces obtained by tensoring by $\mathbb{Q}$ and $\mathbb{R}$ respectively.

\begin{rem}  This definition allows us to pullback numerical classes by any flat morphism $q : X \to Y$ of relative dimension $e$. 
Our presentation is slightly different from the classical one given in \cite[Section 19.1]{fulton}. We refer to Appendix \ref{appendix} for a proof of the equivalence of these two approaches.
\end{rem}

\begin{prop} Let $q: X \to Y$ a morphism. Then the morphism of groups $q_* : Z_{i}(X) \to Z_i(Y)$ induces a morphism of abelian groups $q_*:\num_{i}(X) \to \num_{i}(Y) $. 
\end{prop}

\begin{proof}  Let $n$ be the dimension of $X$ and $l$ be the dimension of $Y$, and
let $\alpha$ be a cycle in $Z_i(X)$ such that $\alpha$ is numerically trivial. 
We need to prove that $q_* \alpha$ is also numerically trivial.

Take $p_1 : Y_1 \to Y$ a flat morphism of relative dimension $e_1$.
Let $X_1$ be the fibred product $X \times_{Y} Y_1$ and let $p_1'$ and $q'$ be the natural projections from $X_1$ to $X$ and $Y_1$ respectively. 
\[\xymatrix{X_1 \ar[d]^{q'} \ar[r]^{p_1'} & X \ar[d]^q \\
Y_1 \ar[r]^{p_1}& Y} \]
Since flatness is preserved by base change (\cite[Proposition 9.2.(b)]{hartshorne}), the morphism $p_1'$ is flat and $q'$ is proper. 
Pick any cycle $\gamma$ whose class is in $\IC^{e_1 +i}(Y_1)$. 
We want to prove that $ ( \gamma  \cdot {p_1}^* q_*\alpha) = 0$.
By \cite[Proposition 1.7]{fulton}, we have that $p_1^*q_* \alpha = q'_*{p_1'}^* \alpha$ in $Z_{e_1+i}(Y_1)$. Applying the projection formula, we get: 
 \begin{equation*}
 \gamma \cdot p_1^* q_* \alpha = \gamma \cdot q'_* p_1'^* \alpha \rat q'_*(q'^* \gamma \cdot p_1'^* \alpha). 
\end{equation*} 
Because $p_1'$ is flat and $q'^*\gamma \in \IC^{e_1 + i }(X_1)$, we have $(q'^* \gamma  \cdot p_1'^* \alpha ) = 0$ so that $(\gamma \cdot p_1^* q_* \alpha )= 0$ as required. 
\end{proof}

The numerical classes defined above are hard to manipulate, we want to define a pullback of numerical classes by any proper morphism. 
We proceed and define dual classes.
\medskip

We denote by $Z^i(X) = \Hom_\mathbb{Z}( Z_i(X) , \mathbb{Z})$ the space of cocycles. If $p_1 : X_1 \to X$ is a flat morphism of relative dimension $e_1$, then any element $\gamma \in \IC^{e_1+i}(X_1)$ induces an element $[\gamma]$ in  $Z^i(X)$ by the following formula: 
\begin{equation} \label{eq_morphisme_classe_induite}
 [\gamma] : \alpha \in Z_i(X) \rightarrow (\gamma \cdot  p_1^* \alpha) \in \mathbb{Z}. 
\end{equation}

\begin{defi} \label{defi_numd}
The abelian group $\numd^i(X)$ is the subgroup of $Z^i(X)$ generated by elements of the form $[\gamma]$ where $ \gamma \in \IC^{e_1+i}(X_1)$ and $X_1$  is flat over $X$ of relative dimension $e_1 $.
\end{defi}

\begin{rem} \label{rem_deg} By definition, the map $\deg: Z_0(X) \to \mathbb{Z}$ is naturally an element of $Z^0(X)$.
 Moreover, one has using Theorem \ref{thm_segre}.(ii) that:
\begin{equation*}
z \in Z_{0}(X) \rightarrow (s_0(E) \actd z) = \deg(z) \in \mathbb{Z},
\end{equation*}
for any vector bundle $E$ on $X$.
Hence, $\deg $ defines an element of $\numd^0(X)$ by definition of Segré classes (Definition \ref{def_segre_class}).
\end{rem}

\bigskip

\begin{prop}\label{prop_numd_dual_num} By definition of the numerical equivalence relation, any element of $\numd^i(X)$ induces an element of the dual $\Hom_\mathbb{Z}(\num_i(X), \mathbb{Z})$. Hence, we can define a natural pairing between $\numd^i(X)$ and $\num_i(X)$.  For any normal projective variety, the pairing $\numd^i(X) \times \num_i(X) \to \mathbb{Z}$ is non degenerate (i.e the canonical morphism from $\numd^i(X)$ to $\Hom_\mathbb{Z}(\numd_i(X), \mathbb{Z})$ is injective). 
\end{prop}

\begin{proof} It follows directly from the definition of $\numd^i(X)$ and $\num_i(X)$. 

\end{proof}

\medskip

A priori, an element of $\numd^i(X)$ is a combination of elements $ [\gamma_1] + [\gamma_2] + \ldots + [\gamma_j]$. 
The following proposition proves one can always take $j=1$ at least if we tensor all spaces by $\mathbb{Q}$.

\begin{prop} \label{prop_representabilite_numd}
 Any element of $\numd^i(X)$ is induced by $\gamma \in \IC^{e_1+i}(X_1)_\mathbb{Q}$ where $p_1: X_1 \to X$ is a flat morphism of relative dimension $e_1$.
\end{prop}

\begin{proof} 
By an immediate induction argument, we are reduced to prove the assertion for the sum of two elements $[\gamma_1] +[\gamma_2]$ where  $\gamma_j \in \IC^{e_j + i}(X_i)_\mathbb{Q}$ and $p_j : X_j \to X$ are flat morphisms of relative dimension $e_1$ and $e_2$ respectively.  
\smallskip

Let us consider $X'$ the fibre product $X_1 \times X_2$ over $X$ and $p_j'$ the flat projections from $X'$ to $X_j$ for $j=1,2$. 
By linearity , we only need to show that there exists an element $\gamma_1' \in \IC^{e_1+e_2 + i} (X')$ such that $[\gamma_1'] = [\gamma_1]$ in $\numd^i(X)$. 
\[\xymatrix{ & X_1 \times X_2  \ar[ld]^{p_2'} \ar[rd]^{p_1'}& \\
 X_1  \ar[rd]^{p_1}& & X_2  \ar[ld]^{p_2}\\
& X & }\] 
Take an ample Cartier divisor $H_{X_2}$ on $X_2$ and $\lambda_2$ an integer such that ${p_2}_* H_{X_2}^{e_2} \rat \lambda_2 [X]$. 
Setting $ \gamma_1' = \dfrac{1}{\lambda_2} p_1'^* H_{X_2}^{e_2} \cdot p_2'^* \gamma_1 $, we need to prove that for any $\alpha \in Z_i(X)$, one has $(\gamma_1 \cdot p_1^* \alpha )= (\gamma_1' \cdot p_2'^* p_1^* \alpha)$.
By \cite[Proposition 1.7]{fulton}, we have the equality $ {p_2'}_*p_1'^* H_{X_2}^{e_2} = p_1^*{p_2}_* H_{X_2}^{e_2}$ in $Z^{e_2}(X_2)$, hence:
\begin{equation*}
 {p_2'}_*p_1'^* H_{X_2}^{e_2} = \lambda_2 p_1^* [X].  
\end{equation*}
Since $X_1$ is reduced and $p_1^* [X]$ is a cycle of codimension $0$ in $X_1$, we have $p_1^* [X] = [X_1]$. 
Hence by the projection formula, we have:
\begin{equation*}
\begin{array}{lll}
\dfrac{1}{\lambda_2} {p_2'}_* ( p_2'^* (\gamma_1 \cdot p_1^* \alpha) \cdot p_1'^* H_{X_2}^{e_2})& = & \dfrac{1}{\lambda_2} (p_1^* \alpha \cdot \gamma_1) \cdot {p_2}_* p_1'^* H_{X_2}^{e_2}\\
 & =  & \dfrac{1}{\lambda_2} (p_1^* \alpha \cdot \gamma_1) \cdot \lambda_2[X_1]\\
 & = & p_1^* \alpha \cdot \gamma_1. 
\end{array}
\end{equation*}
In particular, the degrees are equal and $[\gamma_1]= [\gamma_1'] \in \numd^i(X)$ as required. 
\smallskip 

By the same argument, there exists a class $\gamma_2' \in \IC^{e_1 + e_2 + i} (X_1 \times X_2)$ such that $[\gamma_2] = [\gamma_2'] \in \numd^i(X)$, hence $[\gamma_1] + [\gamma_2] = [\gamma_1'] + [\gamma_2'] = [\gamma_1' + \gamma_2'] \in \numd^i(X)$ as required.
\end{proof}

\begin{defi}
We define $\num_\bullet(X)$ (resp. $\numd^\bullet(X)$) by $\oplus_i \num_i(X)$ (resp. $ \oplus_i \numd^i(X)$). 
\end{defi}
\subsection{Algebra structure on the space of numerical cycles}

We now define a structure of algebra on $\numd^\bullet(X)$, and prove that $\num_\bullet(X)$ has a structure of $\numd^\bullet(X)$ module.

\medskip

Pick $\gamma\in \IC^{e_1+i}(X_1)_\mathbb{Q}$ where $p_1: X_1 \to X$ is a flat morphism of relative dimension $e_1$. The element $\gamma$ induces a morphism in the Chow group:
\begin{equation} \label{eq_ind_mor}
\gamma \actd \cdot :\alpha \in A_l(X) \rightarrow {p_1}_* (\gamma \cdot p_1^* \alpha) \in A_{l-i}(X). 
\end{equation}
The morphism $ \gamma \actd \cdot :A_l(X) \to A_{l-i}(X)$ induces a morphism of abelian groups from $\num_l(X)$ to $ \num_{l-i}(X)$.

\begin{prop} 
Any element $\alpha\in \numd^i(X)$ 
induces a morphism $\alpha \actd \cdot : N_\bullet(X) \to \num_{\bullet-i}(X)$ such that the following conditions are satisfied.
\begin{enumerate}
\item[(i)] If $\alpha $ is induced by $\gamma \in \IC^{e_1 + i}(X_1)_\mathbb{Q}$ where $p_1: X_1 \to X$ is a flat morphism of relative dimension $e_1$, 
then for any integer $l$ and any $z \in \num_l(X)$, one has in $\num_{l-i}(X)$:
\begin{equation*}
\alpha \actd z = \gamma \actd z .
\end{equation*}

\item[(ii)]  For any $\alpha, \beta \in \numd^i(X)$ and any $z \in \num_{l}(X)$, we have:
\begin{equation*}
(\alpha + \beta) \actd z = \alpha \actd z + \beta \actd z.
\end{equation*}
\end{enumerate}

\end{prop}

\begin{proof} 
Let us consider $\alpha \in \numd^i(X)$ and suppose it is induced by $\gamma_1\in \IC^{e_1 +i}(X_1)_\mathbb{Q}$ where $p_1 : X_1 \to X$ is a flat morphism of relative dimension $e_1$. 
We define the map $\alpha \actd \cdot $ as :
\begin{equation*}
\alpha \actd z = \gamma_1 \actd z, 
\end{equation*} 
for any $z \in \num_l(X)$. 
\medskip
We show that the morphism does not depend on the choice of the class $\gamma_1$ and $(i)$ is follows from Proposition \ref{prop_representabilite_numd}. Assertion $(ii)$ follows from the linearity of the intersection product whose proof follows closely the proof of  Proposition \ref{prop_representabilite_numd}.
\bigskip

 Suppose that $[\gamma_1] = [\gamma_2] \in \numd^i(X)$ where $\gamma_2 \in \IC^{e_2 + i}(X_2)_\mathbb{Q}$ and $p_2: X_2 \to X$ is a flat morphism of relative dimension  $e_2$,  then we need to prove that:
\begin{equation*}
{p_1}_* ( \gamma_1 \cdot  p_1^* z) \equiv {p_2}_* (\gamma_2 \cdot p_2^*z),
\end{equation*}
for any fixed $z \in Z_{l}(X)$. 
Take $\beta \in \IC^{e_3 + l-i}(X_3)$ where $p_3: X_3 \to X$ is flat morphism of relative dimension $e_3$, we only need to show that:
\begin{equation*}
(\beta \cdot p_3^* {p_1}_* (\gamma_1 \cdot p_1^* z) ) = (\beta \cdot p_3^* {p_2}_* (\gamma_2 \cdot p_2^* z)).
\end{equation*} 
Let $X_1'$ and $X_2'$ the fibre products $X_1 \times X_3$ and $X_2 \times X_3$, and $p_1': X_1' \to X_3$, $p_3': X_1' \to X_1$, $q_2 : X_2' \to X_3$, $q_3: X_2' \to X_2$ be the corresponding flat projection morphisms such that we obtain the following commutative diagrams:
\begin{equation*} \label{diag_geant}
\xymatrix { & X_1' \ar[rd]^{p_1'} \ar[ld]^{p_3'} &  & & X_2' \ar[rd]^{q_2} \ar[ld]^{q_3} & \\
X_1 \ar[rd]^{p_1} & & X_3 \ar[ld]^{p_3} & X_2 \ar[rd]^{p_2} & & X_3  \ar[ld]^{p_3}\\
& X & && X .&}
\end{equation*} 
As above, we have $p_3^*{p_1}_* = {p_1'}_* p_3'^*$, hence:

\begin{equation*}
\begin{array}{lll}
(\beta \cdot p_3^* {p_1}_* (\gamma_1 \cdot p_1^* z) )& =& (\beta \cdot {p_1'}_* p_3'^* (\gamma_1 \cdot p_1^* z)) \\
 & =  & (p_1'^* \beta \cdot p_3'^* (\gamma_1 \cdot p_1^* z))\\
 & = & (p_3'^* \gamma_1 \cdot p_1'^* p_3^* z \cdot p_1'^* \beta ) \\
& = & ( \gamma_1 \cdot {p_3'}_* p_1'^* ( p_3^* z \cdot \beta ) ) \\
& = & (\gamma_1 \cdot p_1^*{p_3}_* (p_3^* z \cdot \beta ))\\
& =& (\gamma_2 \cdot p_2^* {p_3}_* (p_3^* z \cdot \beta )).
\end{array}
\end{equation*}
By a similar argument, we show that $(\beta \cdot p_3^* {p_2}_* (\gamma_2 \cdot p_2^* z)) = (\gamma_2 \cdot p_2^* {p_3}_* (p_3^* z \cdot \beta))$ which implies the desired equality:
\begin{equation*}
(\beta  \cdot p_3^* {p_1}_* (\gamma_1 \cdot p_1^* z) ) = (\beta \cdot p_3^* {p_2}_* (\gamma_2 \cdot p_2^* z)).
\end{equation*}

\end{proof}

\begin{prop} \label{prop_num_numd_module}
 There exists a unique structure of commutative graded ring with unit ($\deg$) on $\numd^\bullet(X)$ given by $ (\alpha,\beta) \in  \numd^\bullet(X) \times \numd^\bullet(X) \mapsto \alpha \cdot \beta \in  \numd^\bullet(X)$  which satisfies the following properties:
 \begin{enumerate}
 \item[(i)] For any $\alpha,\beta \in \numd^\bullet(X)$ and any $z \in \numd_\bullet(X)$, one has:
 \begin{equation*}
( \alpha \cdot \beta) \actd z = ( \alpha \actd ( \beta \actd z) ) = ( \beta \actd ( \alpha \actd z) ).
 \end{equation*}
 \item[(ii)] For any $z\in \numd_\bullet(X)$, we have $(\deg) \actd z = z$.
 \item[(iii)] The morphism of abelian groups given by $$ (\alpha, z) \in \numd^\bullet(X) \times \numd_\bullet(X) \mapsto \alpha \actd z \in  \numd_\bullet(X)$$ is bilinear.
 \end{enumerate}
  Hence, the abelian group $\num_\bullet(X)$ has the structure of a graded $\numd^\bullet(X)$-module.
\end{prop}
\begin{proof} 
Take $\alpha_1 \in \numd^i(X)$ and $\alpha_2 \in \numd^l(X)$ and define $\varphi \in Z^{i+l}(X)$ by the formula:
\begin{equation*}
\varphi : z \in Z_{i+l}(X) \rightarrow  (\alpha_1 \actd ( \alpha_2\actd z)). 
\end{equation*}
We prove that $\varphi$ is an element of $\numd^{i+l}(X)$. 
\smallskip

By linearity, we can suppose that $\alpha_i$ is induced by $\gamma_i \in \IC^{i+e_j}(X_j)$ where $p_j :X_j \to X$ is a flat morphism of relative dimension $e_j$ for $j=1,2$.
 Let $X' = X_1 \times_X X_2$ be the fibre product, let $p_1'$ and $p_2'$ be the projections from $X'$ to $X_1$ and $X_2$ respectively such that we have the commutative diagram:
 \[\xymatrix{ & X' \ar[rd]^{p_1'} \ar[ld]^{p_2'}& \\
 X_1 \ar[rd]^{p_1} & & X_2  \ar[ld]^{p_2}\\
& X.& }\]
By the projection formula, we obtain for all $z \in Z_{i+l}(X)$:
\begin{equation} \label{eq_prod_numd}
\varphi(z) = ( p_1'^* \gamma_2 \cdot  p_2'^* \gamma_1 \cdot  p_2'^* p_1^*z). 
\end{equation} 
In particular, we have shown that $\varphi$ is induced by $p_1'^* \gamma_2 \cdot p_2'^* \gamma_1\in \IC^{e_1 + e_2 +i+l}(X')$, hence $\varphi$ is an element of $\numd^{i+l}(X)$. 
Moreover, the commutativity of the intersection product in \eqref{eq_prod_numd} proves that $(\alpha_2 \actd (\alpha_1 \actd z))= (\alpha_1 \actd (\alpha_2 \actd z))$ for any $z \in \num_{i+l}(X)$, hence $\alpha_1 \cdot \alpha_2 = \alpha_2 \cdot \alpha_1$.

\bigskip

Pick a vector bundle $E$ on $X$. 
As the element $\deg \in \numd^0(X)$ is equal to $ z \rightarrow (s_0(E) \actd z)$ in $\numd^0(X)$ (see Remark \ref{rem_deg}), we get using Theorem \ref{thm_segre}.(ii) that:
\begin{equation*}
 (\alpha \actd z) =(\alpha \actd (s_0(E) \actd z)) = (s_0(E) \actd (\alpha \actd z)) = ((\alpha \cdot \deg) \actd z) = (\deg \cdot \alpha) \actd z,
\end{equation*}
for any $z \in \num_l(X)$ and any $\alpha \in \numd^l(X)$.
Hence, $\deg$ is a unit of $\numd^\bullet(X)$.

\end{proof}

 \red 

\black 
\subsection{Pullback on dual numerical classes}

\bigskip 
Let us consider $q: X \to Y$ a proper morphism. 
We define for any integer $i$ the pullback $q^* : \numd^i(Y) \to \Hom_\mathbb{Z}(\numd_i(X) , \mathbb{Z}) $ as the dual of the pushforward operation $q_* : \num_i(X ) \to \num_i(Y)$ with respect to the pairing $\numd^i(X) \times \num_i(X) \to \mathbb{Z}$ defined in Proposition \ref{prop_numd_dual_num}.

\begin{prop} Let $q: X \to Y$ be a proper morphism.
The morphism $q^*$ induces a morphism of graded rings $q^* :\numd^\bullet(Y) \to \numd^\bullet (X)$ which satisfies the projection formula:
\begin{equation*}
\forall \alpha \in \numd^i(Y), \forall z \in \num_l(X), q_* (q^* \alpha \actd z ) = \alpha \actd q_* z.
 \end{equation*}
\end{prop}

\begin{proof}

We only need to prove that the image $q^* (\numd^i(Y)) $ is contained in $\numd^i(X)$ and that the projection formula is satisfied as it directly implies that $q^*: \numd^\bullet (Y) \to \numd^\bullet(X)$ is a morphism of rings since:
\begin{equation*}
 (\alpha \cdot \beta) \actd q_*z =  q^* (\alpha \cdot \beta) \actd z = \alpha \actd q_* (q^* \beta \actd z) = (q^* \alpha \cdot q^*\beta ) \actd z,
\end{equation*} 
for any $\alpha \in \numd^i(Y)$, $\beta\in \numd^l(Y)$ and any $z \in \num_{i+l}(X)$.

\bigskip

Consider a class $\alpha \in \numd^i(Y)$ which is induced by $\gamma \in \IC^{e_1 + i}(Y_1)$ where $p_1 : Y_1  \to Y$ is a flat proper morphism of relative dimension $e_1$. 
Setting $X_1$ to be the fibre product $Y_1 \times X$ and $p_1' $, $q'$ the projections from $X_1$ to $X$ and $X_1$ respectively, one remarks using the equality $q'_*p_1'^* = p_1^*q_*$ (\cite[Proposition 1.7]{fulton}) that $q^* \alpha$ is induced by $q'^* \gamma$, hence $q^* \alpha \in \numd^i(X)$ as required. 
And the projection formula follows easily from the projection formula on divisors (Theorem \ref{thm_prop_intersection_diviseurs}.(ii)).

\end{proof}

Let us sum up all the properties of numerical classes proven so far :
\begin{thm} \label{thm_num_sum_up} Let $q: X \to Y$ be a proper morphism. For any integer $0 \leqslant i\leqslant \dim X$ and $0 \leqslant l \leqslant \dim Y$: 
\begin{enumerate}
\item[(i)] The pushforward morphism $q_* : Z_i(X) \to Z_i(Y)$ induces a morphism of abelian groups $q_* : \num_{i}(X) \to \num_{i}(Y)$.
\item[(ii)] The dual morphism $q^* : Z^l(Y) \to Z^l(X) $ maps $\numd^l(Y) $ into $ \numd^l(X)$. 
\item[(iii)] The induced morphism $q^* : \numd^\bullet(Y) \to \numd^\bullet (X)$ preserves the structure of graded rings.
\item[(iv)] (Projection formula)For all $\alpha \in \numd^l(Y)$ and all $z \in \num_i(X)$, we have $q_*( q^*\alpha \actd z) \equiv \alpha \actd q_* z $ in $\num_{i-l}(Y)$. 
\end{enumerate}
\end{thm}

\subsection{Canonical morphism}

\begin{thm} \label{thm_identification} 
The morphism $\psi_X: \alpha \in  \numd^i(X) \mapsto \alpha \actd [X] \in  \num_{n-i}(X)$ is the unique morphism which satisfies the following properties. 
\begin{enumerate}
\item[(i)] The image of the morphism $\deg: Z_{0}(X) \to \mathbb{Z}$ seen as an element of $ Z^0(X)$ is given by $\psi_X(\deg) = [X]$.
\item[(ii)] The morphism $\psi_X$ is $\numd^i(X)$-equivariant, i.e for all $\alpha \in \numd^i(X)$ and all $\beta \in \numd^l(X)$, we have: 
\begin{equation*}
\psi_X( \alpha \cdot \beta) = \alpha \actd \psi_X(\beta).
\end{equation*} 
\item[(iii)] Suppose $q : X \to Y$ is a generically finite morphism where $Y$ is of dimension $n$, then we have the following identity:
\begin{equation*}
 q_* \circ \psi_X \circ  q^* = \deg(q) \times \psi_Y .
\end{equation*} 
\end{enumerate}
\end{thm}

\begin{proof}

Recall that $\deg$ is the unit in $\numd^\bullet(X)$, hence $\psi_X(\deg) = [X]$ and $(ii)$ follows directly from the definition and Proposition \ref{prop_num_numd_module}. 
 
Assertion $(iii)$ is then a consequence of the projection formula (see Theorem \ref{thm_num_sum_up}.$(iv)$) and the fact that $q_*[X] = \deg(q) [Y]$.
\smallskip 

Let us prove that $\psi_X$ is unique. Suppose that $\varphi : \numd^i(X) \to \numd_{n-i}(X) $ satisfies the hypothesis of the theorem. 
Since $\varphi(\deg) = [X]$ and since $\deg$ is the unit element of the ring $\numd^\bullet(X)$, we have that for any $\alpha \in \numd^i(X)$, $\alpha = \alpha \cdot \deg$. By $(ii)$, 
\begin{equation*}
\varphi( \alpha ) = \varphi (\alpha \cdot \deg) = \alpha \actd \varphi(\deg) = \alpha \actd [X] = \psi_X(\alpha),
\end{equation*}
as required.

\end{proof}

Now we prove some properties of $\psi_X$ in some particular cases.
\begin{thm} \label{thm_cano_mor_prop} 
The following properties are satisfied.
\begin{enumerate}
\item[(i)] If $X $ is smooth, then for all integers $0 \leqslant i \leqslant n$,  the induced morphism $\psi_X : \numd^i(X)_\mathbb{Q} \to \num_{n-i}(X)_\mathbb{Q}$ is an isomorphism.

\item[(ii)] If $X$ is smooth and $q: X \to Y$ is a surjective generically finite morphism where $Y$ is a normal projective variety. Then we have for all integer $i$:
\begin{equation}\label{eq_im_psi_orth}
q^* ( \psi_Y(\numd^{n-i}(Y)_\mathbb{Q})^\bot) =  q^*( \numd^i(Y)_\mathbb{Q}) \cap \Ker( q_* \circ \psi_X: \numd^i(X)_\mathbb{Q} \to \num_{n - i}(Y)_\mathbb{Q} ). 
\end{equation}

\end{enumerate}
\end{thm}

\begin{proof} $(i)$ Let us show that $\psi_X$ is surjective.
 By the Grothendieck-Riemann-Roch's theorem (Theorem \ref{thm_grr}), the Chern character induces an isomorphism:
\begin{equation*}
  \ch  \actd [X] : E \in K^0(X) \otimes \mathbb{Q} \rightarrow \ch(E) \actd [X] \in A_\bullet(X) \otimes \mathbb{Q}.
\end{equation*}
This implies that the morphism $\psi_X: \numd^i(X)_\mathbb{Q} \to \num_{n-i}(X)_\mathbb{Q} $ is surjective because any Chern class is the image of a product of Cartier divisors by a flat map (see Remark \ref{rem_segre_image_diviseur}). 
\smallskip

We now prove that $\psi_X : \numd^i(X)_\mathbb{Q} \to \num_{n-i}(X)_\mathbb{Q}$ is injective. 
  Take $\alpha_1\in \numd^i(X)_\mathbb{Q}$ such that $\psi_X(\alpha_1) = 0$. By Proposition \ref{prop_representabilite_numd}, the class $\alpha_1$ is induced by $\gamma_1 \in \IC^{e_1 + i}(X_1)_\mathbb{Q} $ where $p_1 : X_1 \to X$ is a flat morphism of relative dimension $e_1$.
 The condition $\psi_X(\alpha_1) =0$ is equivalent to the equality ${p_1}_* \gamma_1 = 0 \in \numd_{n-i}(X)$. 
 We need to show that $(\gamma_1 \cdot p_1^*z)= 0$ for any cycle $z \in Z_i(X)$.
As $X$ is smooth, we may compute intersection products inside the Chow group $A_\bullet(X)$ directly by Remark \ref{rem_segre_image_diviseur} and we get:
\begin{equation*}
(\gamma_1 \cdot p_1^* z) = ( {p_1}_* (\gamma_1 \cdot p_1^* z) ) = ({p_1}_* \gamma_1 \cdot z ) = 0
\end{equation*}
as the class $z \in \num_i(X)$ is the image of an element of $ \numd^{n-i}(X)_\mathbb{Q}$ by surjectivity of $\psi_X$.
\bigskip

\bigskip

$(ii)$ We have the following series of equivalence:
\begin{equation*}
\begin{array}{llll}\label{equiv_1}
 \beta \in \psi_Y ( \numd^{n- i }(Y)_\mathbb{Q})^\bot &\Leftrightarrow & \forall \alpha \in \numd^{n-i}(Y)_\mathbb{Q}, (\beta \actd \psi_Y(\alpha) ) = 0 \\
 & \Leftrightarrow & \forall \alpha \in \numd^{n-i}(Y)_\mathbb{Q} ,( \beta \actd (q_* \psi_X q^* \alpha)) = 0  \\
 & \Leftrightarrow & \forall \alpha \in \numd^{n-i}(Y)_\mathbb{Q} , (q^* \beta \cdot q^* \alpha) = 0 \\
& \Leftrightarrow & \forall \alpha \in \numd^{n-i}(Y)_\mathbb{Q} , (\alpha \actd q_* \psi_X q^* \beta ) = 0 \\
& \Leftrightarrow &  q^* \beta \in \Ker(q_* \circ \psi_X :  \numd^i(X)_\mathbb{Q} \to \num_{n - i}(Y)_\mathbb{Q}, 
\end{array}
\end{equation*}
where the second equivalence follows from Theorem \ref{thm_cano_mor_prop}.(iii), the third and the fourth equivalence from the projection formula, and the last equivalence is a consequence of the fact that $\psi_X$ is self-adjoint :
\begin{equation*}
( \beta \actd \psi_Y(\alpha)  ) = (\beta\actd (\alpha \actd [Y]))=(\alpha \actd( \beta \actd [Y ])) = (\alpha \actd \psi_Y( \beta)),
\end{equation*}
where $\alpha \in \numd^{i}(Y)$ and $\beta \in \numd^{n-i}(Y) $.

\end{proof}

\begin{rem} \label{rem_numd_smooth} The proof of Theorem \ref{thm_cano_mor_prop}.$(i)$ shows that when $X$ is smooth, $\num_i(X)_\mathbb{Q}$ is the quotient of $Z_i(X)_\mathbb{Q}$ by cycles $z \in Z_i(X)_\mathbb{Q} $ such that for any cycle $z' \in Z_{n-i}(X)_\mathbb{Q}$, one has $(z \cdot z') = 0$.
\end{rem}

\begin{rem} When $X$ is smooth and when $\C = \mathbb{C}$, denote by $\Alg^i(X)$ the subgroup of the de Rham cohomology $H^{2i}(X, \mathbb{C})$ generated by algebraic cycles of dimension $i$ in $X$. Then there is a surjective morphism $\Alg^i(X) \to \numd^i(X)_\mathbb{Q}$ 
\end{rem}

\subsection{Numerical spaces are finite dimensional}

\begin{thm} 
Both $\mathbb{Q}$-vector spaces $\num_i(X)_\mathbb{Q}$ and $\numd^i(X)_\mathbb{Q}$ are finite dimensional.
\end{thm}

\begin{proof} 
If $X$ is smooth, then using Remark \ref{rem_numd_smooth}, $\numd_i(X)_\mathbb{Q}$ is the quotient of $Z_i(X)_\mathbb{Q}$ by the equivalence relation which identifies cycles $\alpha $ and $\beta$ in $Z_i(X)_\mathbb{Q}$ if for any cycle $z \in Z_{n-i}(X)_\mathbb{Q}$, $(z \cdot \alpha) = ( z \cdot \beta)$. 
In particular, the vector-space $\num_i(X)_\mathbb{Q}$ is finitely generated (see \cite[Theorem 23.6]{milneLEC} for a reference), and so is $\numd^i(X)_\mathbb{Q}$ using Theorem \ref{thm_cano_mor_prop}.$(i)$. 

\medskip 

If $X$ is not smooth, by DeJong's alteration theorem (cf \cite[Theorem 4.1]{de_jong}), there exists a smooth projective variety $X'$ and a generically finite surjective morphism $q: X' \to X$. We only need to show that the pushforward $q_* : \num_i(X')_\mathbb{Q} \to \num_i(X)_\mathbb{Q}$ is surjective. 
Indeed this first implies that $\num_i(X)_\mathbb{Q}$ is finite dimensional. Since the natural pairing $\numd^i(X)_\mathbb{Q} \times \num_{i}(X)_\mathbb{Q} \to \mathbb{Q}$ is non degenerate we get an injection of $\numd^i(X)_\mathbb{Q}$ onto $\Hom_{\mathbb{Q}}(\num_i(X)_\mathbb{Q}, \mathbb{Q})$ which is also finite dimensional.
\smallskip

We take $V$ an irreducible subvariety of codimension $i $ in $X$. If $\dim q^{-1}(V) = \dim V $, then the class $q_* [q^{-1}(V)]$ in $\num_{\dim V}(X)_\mathbb{Q}$ is represented by a cycle of dimension $\dim V$ which is included in $V$. As $V$ is irreducible, we have $q_*[q^{-1}(V)] \equiv \lambda [V]$ for some $\lambda \in \mathbb{N}^*$. 

 If the dimension of $q^{-1}(V)$ is strictly greater than $V$, we take $W$ an irreducible component of $q^{-1}(V)$ such that its image by $q_{|W} : W \to V$ is dominant. 
 We write the dimension of $W$ as $\dim V + r$ where $r>0$ is an integer. 
Fix an ample divisor $H_X$ on $X$. 
 The class $H_X^r \actd [W] \in \num_{\dim V}(X')_\mathbb{Q}$ is represented by a cycle of dimension $\dim V$ in $W$. So the image of the class $q_*(H_X^r \actd [W]) \in \num_{\dim V}(X)_\mathbb{Q}$ is a multiple of $[V]$ which implies the surjectivity of $q_*$.

\end{proof}

\begin{cor} 
For any integer $0 \leqslant i \leqslant n$, the pairing $\numd^i(X)_\mathbb{R} \times \num_i(X)_\mathbb{R} \to \mathbb{R}$ is perfect (i.e the canonical morphism from $\numd^i(X)_\mathbb{R} $ to $\Hom_\mathbb{R}( \num_i(X)_\mathbb{R}, \mathbb{R})$ is an isomorphism).

\end{cor}

\begin{cor} Suppose that the dimension of $X$ is $2n$, then the morphism $\psi_X: \numd^n(X)_\mathbb{Q} \to \num_n(X)_\mathbb{Q}$ is an isomorphism.
\end{cor}

\begin{proof}
We apply \eqref{eq_im_psi_orth} to an alteration $X'$ of $X$ where $q : X' \to X$ is a proper surjective morphism and where $X'$ is a smooth projective surface. This proves that $\psi_X : \numd^{n}(X)_\mathbb{Q} \to \num_n(X)_\mathbb{Q}$ is surjective. By duality, this gives that $\psi_X : \numd^n(X)_\mathbb{Q} \to \num_{n}(X)_\mathbb{Q}$ is injective. As a consequence, we have that $\psi_X:\numd^n(X)_\mathbb{Q} \to \num_n(X)_\mathbb{Q} $ is an isomorphism.

\end{proof}

\begin{cor}  Let $X$ be a complex normal projective variety with at most rational singularities. We suppose that $X$ is numerically $\mathbb{Q}$-factorial in the sense of \cite{boucksom_fernex_favre_urbinati}. Then the morphisms $\psi_X : \numd^1(X)_\mathbb{Q} \to \num_{n-1}(X)_\mathbb{Q}$ and $\psi_X:\numd^{n-1}(X)_\mathbb{Q} \to \num_1(X)_\mathbb{Q}$ are isomorphisms.
\end{cor}

\begin{proof} Using \cite[Theorem 5.11]{boucksom_fernex_favre_urbinati}, then any Weil divisor which is numerically $\mathbb{Q}$-Cartier is $\mathbb{Q}$-Cartier. In particular, $\psi_X : \numd^1(X)_\mathbb{Q } \to \num_{n-1}(X)_\mathbb{Q}$ is surjective. Using \eqref{eq_im_psi_orth} to an alteration of $X'$ applied to $i=1$, we have that $\psi_X: \numd^1(X)_\mathbb{Q} \to \num_{n-1}(X)_\mathbb{Q}$ is injective. Hence $\numd^1(X)_\mathbb{Q}$ and $\num_{n-1}(X)_\mathbb{Q}$ are isomorphic and by duality $\numd^{n-1}(X)_\mathbb{Q}$ and $\num_1(X)_\mathbb{Q}$ are also isomorphic.
\end{proof}

\begin{ex} When $X = X(\Delta)$ be a toric variety associated to a complete fan $\Delta$. 
The map $\psi_X: \numd^1(X)_\mathbb{Q} \to \num_{n-1}(X)_\mathbb{Q}$ is an isomorphism if and only if $\Delta$ is a simplicial fan. Indeed, denote by $N$ the lattice containing $\Delta$ and $M = \Hom(N, \mathbb{Z})$ its dual. 
For any cone $\sigma \in N$, we denote by $M(\sigma)$ the vector space defined by $ M(\sigma) = \{ l \in M \ | \langle l , v \rangle = 0, \  \forall  v \in \sigma \}$. 
The proposition in \cite[\S 5.1]{fulton_toric} implies that any class in $\num_{n-1}(X)_\mathbb{Q}$ is represented by a torus-invariant Weil $\mathbb{Q}$-divisor $D = \sum a_i [V_i]$ in $X(\Delta)$. Since every maximal cone $\sigma$ in the fan $\Delta \subset N$ is full-dimensional, one has $M(\sigma) = \{ 0\}$ and there exists an element $u(\sigma) \in M / M(\sigma) = M $ such that for any $1$-dimensional ray $v_i \in \sigma$, one has:
\begin{equation*}
\langle u(\sigma) , v_i \rangle = -a_i.
\end{equation*}
The element $u(\sigma)$ is uniquely determined if and only if the family of rays $ v_i \in \sigma$ are linearly independent (i.e $\Delta$ is simplicial).
\end{ex}

\section{Positivity} \label{section_positivity}

The notion of positivity is relatively well understood for cycles of codimension $1$ and of dimension $1$. 
For cycles of intermediate dimension this situation is however more subtle and was only recently seriously considered (see \cite{debarre_ein_lazarsfeld_voisin}, \cite{chen_coskun}, \cite{coskun_lesieutre_ottem} and the recent series of papers by Fulger and Lehmann (\cite{fulger_lehmann_kernel}, \cite{fulger_lehmann}).

For our purpose, we will first review the notions of pseudo-effectivity and numerically effective classes. Then we generalize the construction of the basepoint free cone introduced by \cite{fulger_lehmann} to normal projective varieties. This cone is suitable for stating generalized Siu's inequalities (see Section \ref{sec_siu}).

\subsection{Pseudo-effective and numerically effective cones}

As in the previous section, $X$ is a normal projective variety of dimension $n$. 
To ease notation we shall also write $\numd^i(X)$ and $\num_i(X)$ for the real vector spaces $\numd^i(X)_\mathbb{R}$ and $\num_i(X)_\mathbb{R}$.

\begin{defi} 
A class $\alpha \in \num_i(X)$ is pseudo-effective if it is in the closure of the cone generated by effective classes. This cone is denoted $\psef_i(X)$.  
\end{defi}

When $i=1$, $\psef_1(X)$ is the Mori cone (see e.g \cite[Definition 1.17]{kollar_mori}), and when $i = n-1$, $\psef_{n-1}(X)$ is the classical cone of pseudo-effective divisors, its interior being the big cone.

\begin{defi} A class $\beta \in \numd^i(X)$ is numerically effective (or nef) if for any class $\alpha \in \psef_{n-i}(X)$, $ (\beta \actd \alpha)  \geqslant 0$. We denote this cone by $\nefd^i(X)$. 
\end{defi}

When $i= 1$, the cone $\nefd^1(X)$ is the cone of numerically effective divisors, its interior is the ample cone.
\smallskip

We can define a notion of effectivity in the dual $\numd^i(X)$. 
\begin{defi} 
A class $\alpha \in \numd^i(X)$ is pseudo-effective if $\psi_X(\alpha) \in \psef_{n-i}(X)$. We will write this cone as $\psefd^i(X)$. 
\end{defi}

\begin{defi} 
A class $z \in \num_i(X)$ is numerically effective if for any class $\alpha \in \psefd^i(X)$, one has $( \alpha \actd z) \geqslant 0$. This cone is denoted $\nef_i(X)$.
\end{defi}

By convention, we will write $\alpha \leqslant \beta$ (resp. $\alpha \leqslant \beta $) for any $\alpha, \beta \in \num_i(X)$ (resp. $\alpha, \beta \in \numd^i(X)$ ) if $\beta  - \alpha  \in \psef_i(X)$ (resp. $\beta -\alpha \in \psefd^i(X)$). 

When $X$ is smooth, the morphism $\psi_X$ induces an isomorphism between $\numd^i(X)$ and $\num_{n-i}(X)$, and we can identify these cones:
\begin{equation*}
\begin{array}{l}
 \nefd^i(X) = \nef_{n-i}(X) ,\\
 \psefd^i(X) = \psef_{n-i}(X).
\end{array}
\end{equation*}

\subsection{Pliant classes}

We recall the definition of pliant classes introduced in \cite[Definition 3.1]{fulger_lehmann} and their main properties. Their definition involve Schur classes which were introduced in Section \ref{section_characteristic}.

\begin{defi}
The pliant cone $\PL^\bullet(X)$ is defined as the convex cone generated by product of Schur classes of globally generated vector bundle.
\end{defi}
We denote by $\PL^i(X)$ the set of pliant classes of codimension $i$ in $X$.

\begin{thm} \label{thm_fulger_lehmann_pliant}(see \cite[Theorem 1.3]{fulger_lehmann}) 
The pliant cone $\PL^i(X)$ satisfies the following properties.
\begin{enumerate}
\item[(i)] The cone $\PL^i(X)$ is a closed convex salient cone with non-empty interior in $\numd^i(X)_\mathbb{R}$.
\item[(ii)] The cone $\PL^i(X)$ contains product of ample Cartier divisors in its interior.
\item[(iii)] For all integer $i, l$, we have $\PL^i(X) \cdot \PL^l(X) \subset \PL^{i+l}(X)$.
\item[(iv)] For any (proper) morphism $q: X \to Y$, one has that $q^* \PL^i(Y) \subset \PL^i(X)$.
\end{enumerate}
\end{thm}

We recall another proposition which we will reuse in our proofs.
\begin{prop}(cf \cite[Example 3.13]{fulger_lehmann}) Let $\mathbb{G}$ be a Grassmannian variety. Then $\PL^i( \mathbb{G}) = \psefd^i(\mathbb{G})$.
\end{prop}

\subsection{Basepoint free cone on normal projective varieties}

In this section, we define a cone $\Cc^i(X)$ and prove in Corollary \ref{cor_bpf_smooth}  that this cone is equal to the basepoint free cone defined by Fulger-Lehmann when $X$ is smooth. 
This generalizes \cite[Theorem 1.7]{fulger_lehmann} to normal projective varieties and our proof follows closely Fulger-Lehmann's approach.

\medskip
  
Recall that a complete intersection $\gamma \in \IC^{i+e}(X')$ on $X'$ where $p:X' \to X$  is a flat morphism of relative dimension $e$ and where $X' $ is an equidimensional projective scheme induces naturally (see Definition \ref{defi_numd}) an element  $[\gamma] \in \numd^i(X)_\mathbb{R} = \Hom_\mathbb{R}(\num_i(X)_\mathbb{R}, \mathbb{R})$ by intersecting the class $\gamma$ with the pullback by $p$ of a $i$-dimensional cycle in $X$. We also refer to Proposition \ref{prop_num_numd_module} for the definition of the product $\numd^i(X)_\mathbb{R} \times \numd^l(X)_\mathbb{R} \to \numd^{i+l}(X)_\mathbb{R}$.  
\begin{defi}
 The cone $\Cc^i(X)$ is the closure of the convex cone in $\numd^i(X)_\mathbb{R}$ generated by products of the form $[\gamma_1] \cdot \ldots \cdot [\gamma_l]$ where each $\gamma_j $ is a product of $e_j + i_j$ ample Cartier divisors on an equidimensional projective scheme $X_j$ which is flat over $X$ of relative  dimension $e_j$ and where $i_j$ are integers satisfying $i_1 + \ldots + i_l = i$. 
\end{defi}

\begin{rem} By definition, the cone $\Cc^i(X)$ contains the products of ample Cartier divisors and Segré classes of anti-ample vector bundles.
\end{rem}

Recall also that if $q : X \to Y$ is a flat morphism of relative dimension $e$ between projective schemes, then the pushforward is well-defined on numerical cycles 
$q_* : \numd^i(X)_\mathbb{R} \to \numd^{i-e}(Y)_\mathbb{R}$ (see Corollary \ref{cor_push_num}).

\begin{thm} \label{thm_classes_pliantes} The cone $\Cc^i(X)$ is  satisfies the following properties.
\begin{enumerate}
\item[(i)] The cone $\Cc^i(X)$ is a salient, closed, convex cone with non-empty interior in $\numd^i(X)_\mathbb{R}$.
\item[(ii)] The cone $\Cc^i(X)$ contains products of ample Cartier divisors in its interior.
\item[(iii)] For all integer $i$ and $l$, we have  $\Cc^i(X) \cdot \Cc^l(X) \subset \Cc^{i+l}(X)$. 
\item[(iv)] For any (proper) morphism $q: X \to Y$, we have $q^* \Cc^i(Y) \subset \Cc^i(X)$. 
\item[(v)] For any integer $i$, we have $\Cc^i(X) \subset \nefd^i(X) \cap \psefd^i(X)$. 
\item[(vi)] In codimension $1$, one has $\Cc^1(X) = \nefd^1(X)$.
\item[(vii)] For any flat morphism $q: X \to Y$ between equidimensional projective schemes of relative dimension $e$ and any integer $i \geqslant e$, we have $q_* \Cc^{i}(X) \subset \Cc^{i - e}(Y)$.
\end{enumerate}

Moreover, $\BPF(X)$ is the smallest cone satisfying properties $(iii),(vi)$ and $ (vii)$.
\end{thm}

\begin{proof}
We prove successively the items $(iii)$, $(vii)$, $(v)$, $(vi)$, $(iv)$, $(ii)$ and $(i)$. 

$(iii)$, $(vii)$ This follows from the definition of $\Cc^i(X)$.

\bigskip

$(v)$ It is sufficient to prove that for any effective cycle $z \in Z_{n-l}(X)$ and any basepoint free class $\alpha \in \Cc^i(X)$, then $\alpha \actd z \in \psef_{n-i-l}(X)$. Indeed, apply this successively to $z =[X]$ and $z \in \psef_i(X)$ give the inclusions $\Cc^i(X) \subset \psefd^{i}(X)$ and $\Cc^i(X)\subset  \nefd^i(X)$. 
 By definition of basepoint free classes and by linearity, we can suppose that $\alpha $ is equal to a product $[\gamma_1] \cdot \ldots \cdot [\gamma_p]$ where $\gamma_i \in \IC^{e_j + i_j}(X_j)_\mathbb{R}$ are products of ample Cartier divisors on $X_j$ where $ p_j : X_i \to X$ is a flat proper morphism of relative dimension $e_j$ and where $i_j$ are integers such that $i_1 + \ldots + i_p = i$.
 By definition, one has 
$ [\gamma_1 ] \actd z = {p_1}_* (\gamma_1 \cdot  p_1^* z)$. 
Because the cycle $z$ is pseudo-effective, the cycle $p_1^* z $ remains pseudo-effective as $p_1$ is a flat morphism. As $\gamma_1$ is a positive combination of products of ample Cartier divisors, we deduce that the cycle $\gamma_1 \cdot p_1^* z$ is pseudo-effective. Hence, $[\gamma_1] \actd z \in \psef_{n-l_1-l}(X)$. Iterating the same argument, we get that $\alpha \actd z \in \psef_{n-i-l}(X)$ as required.
\bigskip

$(vi)$ The interior of $\nefd^1(X)$ is equal to the ample cone of $X$ so by definition: 
\begin{equation*}
\Int ({\nefd^1(X)})  \subset \Cc^1(X).
\end{equation*} 
As the closure of the ample cone is the nef cone by \cite[Theorem 1.4.21.(i)]{lazarsfeld_positivity_1}, one gets $\nefd^1(X)  \subset \Cc^1(X)$. Conversely, the cone $\Cc^1(X) $ is included in the cone $\nefd^1(X)$, so we get $\Cc^1(X) = \nefd^1(X)$. 

\bigskip

 $(iv)$ By linearity and stability by products, we are reduced to treat the case of a class $[D] $ induced by an ample Cartier divisor on $Y_1$ where  $p_1 : Y_1 \to Y$ is a flat proper morphism, and prove that $q^*[D]$ is a limit of ample Cartier divisors on a flat variety over $X$. 
Let $X_1$ be the fibre product of $Y_1$ and $X$ and let $q'$ be the natural projection from $X_1$ to $Y_1$, observe that $q^*[D]$ is induced by $q'^* D$ which remains nef on $X_1$ as $q'$ is proper.
In particular, it is the limit of ample divisors on $\numd^1(X_1)$.

\bigskip

$(i)$ 
Take $\alpha \in \Cc^i(X)$ such that $-\alpha \in \Cc^i(X)$. Then for all $z \in \psef_i(X)$, one has that $(\alpha \actd z)= 0$ as $\alpha$ is nef by (v). 
Since effective classes of dimension $i$ generate $Z_i(X)$, it follows that  $(\alpha \actd z) = 0$ for any $z \in \num_i(X)_\mathbb{R}$ which implies by definition that $\alpha  = 0$.  
This shows $\Cc^i(X)$ is salient.

\bigskip 

$(ii)$ We show now that $\Cc^i(X)$ contains product of ample divisors in its interior.
To do so we prove that $\PL^i(X) \subset \Cc^i(X)$ for any integer $i \geqslant 1$.

\smallskip
For $i = 1$, $\Cc^1(X) = \nefd^1(X)$, and by definition, the divisor $h$ is ample so it is in the interior of the nef cone and we are done.
Take a globally generated vector bundle $E$ of rank $r$ on $X$ and consider the induced morphism $\phi$ given by:
\begin{equation*}
\phi : X \to \mathbb{G}= G(r, \mathbb{P}(H^0(X, E)^*)).
\end{equation*}
 Since $\PL^i(X) \supset \phi^* \PL^i( \mathbb{G})$ and since these cones are preserved by pullbacks, we are then reduced to proving that
 $\PL^i(\mathbb{G}) \subset \Cc^i(\mathbb{G})$.
 Denote by $G = \PGL(H^0(X,E)^*)$ the projective special orthogonal group of the vector space $H^0(X,E)^*$ and consider a class $\alpha \in \numd^i(\mathbb{G})_\mathbb{R}$. 
Since $\mathbb{G}$ is smooth, $\psi_{\mathbb{G}} :\numd^i(\mathbb{G})_\mathbb{R} \to \numd_{n-i}(\mathbb{G})_\mathbb{R}$ is an isomorphism by Theorem \ref{thm_cano_mor_prop} and $\alpha$ is represented by an effective cycle $z \in Z_{n-i}(\mathbb{G})_\mathbb{R}$.
 
Consider $W$ the Zariski closure in $ G \times \mathbb{G}$ given by:
\begin{equation*}
W =  \overline{\{ (g, g \cdot x)\}}_{g\in G, x\in z} \subset G\times \mathbb{G}.
\end{equation*}
By construction, $W$ is a quasi-projective scheme and the projection $p : W \to \mathbb{G}$ onto $\mathbb{G}$ is a flat morphism. 
Denote by $q: W \to G$ the projection onto $G$.  
Fix $H$ a very ample divisor on $G$ and denote by $M$ the dimension of the group $\PGL(H^0(X,E)^*)$. Then there exists an open embedding $j : W \to \Pg^M_{\mathbb{G}}$ such that one has the following diagram:
\begin{equation*}
\xymatrix{ G \ar[r] & \Pg^M  & \mathbb{P}^M_\mathbb{G}  \ar[l]_{h} \ar[rd]^{\pi} & \\
&W \ar[ru]^{j} \ar[lu]^{q} \ar[rr]^p & & \mathbb{G} }
\end{equation*}
where $\pi : \mathbb{P}^M_\mathbb{G} \to \mathbb{G}$ is the projection onto $\mathbb{G}$ and $h: \mathbb{P}^M_\mathbb{G} \to \mathbb{P}^M$ is the projection onto $\Pg^M$
  
By construction the general fiber of $q$ over an element $g \in G $ is numerically equivalent to $\alpha$ and since we can choose $H$ to be a hyperplane of $\Pg^M$, we have:
\begin{equation*}
\dfrac{1}{(H^M)} p_* q^* H^{M} =  \alpha \in \numd_{n-i}(\mathbb{G})_\mathbb{R}.
\end{equation*}
Moreover \cite[Proposition 1.7]{fulton} implies that $ p_* j^*  = \pi_*$ in $Z_{n-i}(\Pg^M_\mathbb{G})$, hence:
\begin{equation*}
p_* q^* H^M = \pi_* h^* H^M = (H^M) \alpha \in \num_{n-i}(\mathbb{G})_\mathbb{R}.
\end{equation*}
Since $H$ is ample, $h^* H $ is nef and the class $h^* H^M$ belongs to $\BPF_{n-i}(\Pg^M_\mathbb{G})$. 
Assertion $(vii)$ thus implies that the class $\pi_* h^* H^M / (H^M) = \alpha $ belongs to $\BPF_{n-i}(\mathbb{G})_\mathbb{R}$, as required. 

Since $\PL^i(X)$ has non-empty interior in $\numd^i(X)_\mathbb{R}$ by Theorem \ref{thm_fulger_lehmann_pliant}.(ii), we have proved $(ii)$.
\bigskip

Let us prove that the cone $\Cc^i(X)$ is the smallest cone satisfying properties $(iii),(vi)$ and $ (vii)$.
Denote by $\Cc'$ the minimal cone satisfying these conditions. We have that $\Cc'^i(X) \subset \Cc^i(X)$ by minimality. 
Take $q: X_1 \to X$ a flat morphism of relative dimension $e$ where $X_1$ is an equidimensional projective scheme and consider $\alpha \in \IC^{i+e}(X_1)$ a product of ample Cartier divisors on $X_1$. 
Since $q_* : \numd^i(X_1)_\mathbb{R} \to \numd^{i-e}(X)_\mathbb{R}$ and since $\alpha \in \Cc'^{i+e}(X_1)$, we have that $q_* \alpha \in \Cc'^i(X)$ by $(vii)$, hence $\Cc^i(X) \subset \Cc'^i(X)$ as required.

\end{proof}

We recall Fulger-Lehmann's construction of the basepoint free cone. 
A class $\alpha \in \numd_{n-i}(X)_\mathbb{R}$ is strongly basepoint free if there is: 
\begin{enumerate}
\item[$\bullet$] an equidimensional quasi-projective scheme $U$ of finite type over $\C$,
\item[$\bullet$] a flat proper morphism $s: U \to X$ 
\item[$\bullet$] and a proper morphism $p : U \to W$ of relative dimension $n-i$ to a quasi-projective scheme $W $ such that each component of $U$ surjects onto $W$
\end{enumerate}
such that 
$${s_{| F_p}}_* ([F_p]) = \alpha,$$ 
where $[F_p]$ is the fundamental class of a general fiber of $p$. 
We denote by $\BPF'^i(X)$ the closure of the convex cone generated by strongly basepoint free classes in this sense. 
The cone $\BPF'(X)$ as above was defined by Fulger-Lehmann and they proved that this cone satisfies Theorem \ref{thm_classes_pliantes} when $X$ is smooth (\cite[Theorem 1.7]{fulger_lehmann}). The following result proves that the cones $\BPF'(X)$ and $\BPF(X)$ are equal in this case. 

\begin{cor} \label{cor_bpf_smooth} Suppose $X$ is smooth, then the cone $\Cc^i(X)$ is equal to the basepoint free cone $\BPF'^i(X)$. 
\end{cor}

\begin{rem} Our construction of the cone $\BPF(X)$ allows us to generalize Fulger-Lehmann's result for normal varieties. This improvement is due to the fact that we are able to pushforward dual numerical classes by flat morphism. 
\end{rem}

\begin{proof}
By \cite[Theorem 1.7]{fulger_lehmann}, the cone $\BPF'(X)$ satisfies the conditions of Theorem \ref{thm_classes_pliantes}, hence $\Cc^i(X) \subset \BPF'^i(X)$. 
Let us prove the reverse inclusion $\BPF'^i(X) \subset \Cc^i(X) $. 
Take $p : U \to W$ a projective morphism onto an equidimensional quasi-projective variety $W$ where $U$ is a quasi-projective scheme and a flat map $s: U \to X $ such that $s_* [F_p] = \alpha$ where $F_p$ is a general fiber of $p$. 
Take $H_W$ an ample divisor on $W$, then the class $\alpha$ satisfies:
\begin{equation*}
\alpha = s_* p^* H_W^{i+e} \in \numd_{n-i}(X)_\mathbb{R}.
\end{equation*}
Choose an ample divisor $H$ on  $U$, since the class $p^* H_W$ is nef, for any $\epsilon >0$, the divisor $p^* H_W + \epsilon H$ is ample.

Since the morphism $s: U \to X$ is also quasi-projective and there exists an integer $l$ (which depends on $\epsilon$) such that the following diagram is commutative
\begin{equation*}
\xymatrix{ & \mathbb{P}^l_X \ar[rd]^{\pi} & \\
U \ar[ru]^{f_\epsilon} \ar[rr]^s & & X  }
\end{equation*}
where $f_\epsilon : U \to \mathbb{P}^l_X$ is an immersion induced by $p^* H_W + \epsilon H$ and $\pi: \mathbb{P}^l_X \to X$ is the flat projection onto $X$. 

Let $\xi$ be the relative class $c_1(\mathcal{O}_{\mathbb{P}^l_X}(1))$ on $\mathbb{P}^l_X$, then one has that for any cycle $z \in Z_{i}(X)_\mathbb{R}$:
\begin{equation*}
((p^* H_W + \epsilon H)^{i+e} \cdot s^*z) = ( \xi^{i+e} \cdot \pi^* z), 
\end{equation*}
since $f_\epsilon^* \xi = p^* H_W + \epsilon H$. 
Hence, we obtain:
\begin{equation*}
( s_* (p^* H_W + \epsilon H)^{i+e} \cdot z) = (\pi_* \xi^{i+e} \cdot z).
\end{equation*}
Since the class $\xi^{i+e}$ is nef and since these cones are stable by flat pushforward, we have $\pi_*(\xi^{i+e}) \subset  \Cc^i(X)$. 
Taking the limit as $\epsilon \rightarrow 0$, we have that $s_* (p^* H_W + \epsilon H)^{i+e} \rightarrow \alpha = s_* p^* H_W^{i+e}$, hence $\alpha \in \Cc^i(X)$ since each class $s_* (p^* H_W + \epsilon H)^{i+e}) \in \numd^i(X)_\mathbb{R}$ belongs to $\Cc^i(X)$.

\end{proof}

\medskip
We give here a detailed proof of the fact that the pseudo-effective cone is salient (see also \cite[Corollary 3.17]{fulger_lehmann}). The proof uses a useful proposition that we will use later on.  

\begin{prop}
 
\label{prop_nullite_psef}
Let $\alpha \in \psef_{n-i}(X)$ be a pseudo-effective class on $X$ and $\gamma \in \Cc^{n-i}(X)$ be class lying in the interior of the basepoint free cone. 
Then we have $( \gamma \actd \alpha) = 0$ if and only if $\alpha = 0$.  
\end{prop}

\begin{proof}

Let us fix two basepoint free classes $\beta$ and $\gamma$ in $ \numd^{n-i}(X)$, and a norm $|| \cdot ||$ on $\numd^{n-i}(X)_\mathbb{R}$. As $\gamma$ is in the interior of $\Cc^{n-i}(X)$ by Theorem \ref{thm_classes_pliantes}.(ii), there exists a positive constant $C > 0$ such that for any $\beta \in \Cc^{n-i}(X)$, one has:
\begin{equation*}
C || \beta ||_{\numd^{n-i}(X)_\mathbb{R}} \gamma  -\beta \in \Cc^{n-i}(X). 
\end{equation*}  
Intersecting with $\alpha \in \psef_{n-i}(X)$ and using Theorem \ref{thm_classes_pliantes}.(v), we have that $(\beta \cdot \alpha )= 0$.
 Since the basepoint free cone $\Cc^{n-i}(X)$ generates all $\numd^{n-i}(X)_\mathbb{R}$ by Theorem \ref{thm_classes_pliantes}.(i), we have proved that $(\beta' \actd\alpha) = 0$ for any $\beta' \in \numd^{n-i}(X)$, hence $\alpha= 0$ as required.
\end{proof}

\begin{cor}  
The pseudo-effective cone $\psef_{n-i}(X)$ is a closed, convex, full dimensional salient cone in $\num_{n-i}(X)_\mathbb{R}$.
\end{cor}

\begin{proof}
We take $u \in \psef_{n-i}(X)$ such that $-u \in \psef_{n-i}(X)$, then for any ample Cartier divisor $H_X$ on $X$, the products $(H_X^{n-i} \cdot u) $ and $(-u \cdot H_X^{n-i})$ are non-negative hence $(u \cdot H_X^{n-i}) = 0$. This implies that $u = 0$ by Proposition \ref{prop_nullite_psef}.

\end{proof}

\subsection{Siu's inequality in arbitrary codimension} \label{sec_siu}

We recall Siu's inequality:

\begin{prop} \label{ineg_fonda} (\cite[Theorem 2.2.13]{lazarsfeld_positivity_1}) Let $V$ be a closed subscheme of dimension $r$ in $X$ and
let $A,B$ be two $\Q$-divisors nef on $X$ such that $A_{|V}$ is big, then we have in $\num_{i-1}(X)$,
\begin{equation*}
B \actd [V] \leqslant \dfrac{r ( (A^{r-1}\cdot B )\actd [V] )}{(A^r \actd [V])} \  A \actd [V].
\end{equation*}
\end{prop}

\begin{rem} The case $V= X$ is a consequence of the bigness criterion given in \cite[Theorem 2.2.13]{lazarsfeld_positivity_1}, however we will need the result for possibly non-reduced subschemes of $X$. 
\end{rem}

\begin{rem} The proof of the previous proposition implies that $B_{|V} \leqslant r (A^{r-1} \cdot B \actd [V])/ (A^r \actd [V]) \times A_{|V}$ in the Chow group $A^1(V)$. However, since we want to work in the numerical group, we compare these classes in $X$ (we look at their pushforward by the inclusion of $V$ in $X$).
\end{rem}

\begin{proof} The proof is the same as in \cite[Theorem 2.2.13]{lazarsfeld_positivity_1}, that is to find a section of the line bundle $\mathcal{O}_V(m(A -B))$. 
Up to some small pertubations of $A$ and $B$ of the form  $A + \epsilon H$ and $B + \epsilon H$ of $A$ and $B$ where $\epsilon \rightarrow 0$, we can suppose that $A$ and $B$ are ample. 
Moreover, by taking a high multiple of $A$ and $B$, we can suppose that they are also both very ample.  
Since $B$ is very ample, we choose $m$ general elements $E_j$ of the linear system $|B|$ and consider the exact sequence:
\begin{equation*}
\xymatrix{0 \ar[r] & \mathcal{O}_V(mA - mB) \ar[r]& \mathcal{O}_V(mA) \ar[r]&  \mathcal{O}_{\cup E_j}(mA)\ar[r] & 0 . }
\end{equation*}
Taking long exact sequence associated, one obtains the minoration:
\begin{equation*}
h^0(V, \mathcal{O}_V(mA-mB) \geqslant h^0(V , \mathcal{O}_V(mA)) -  h^0( \cup_{j=1}^m E_j, \mathcal{O}_{\cup_{j=1}^m E_j}(mA)). 
\end{equation*}
Observe that $[\cup E_j] = \sum_{j=1}^m [E_j] = m B \actd [V]$.
 Applying \cite[Corollary 3.6.3]{gil_gubler_kunneman_lazarsfeld_2016} to the nef divisor $A$, we get $h^0(V, \mathcal{O}_V(mA))= m^{r}/(r!) (A^r \actd [V]) + o(m^{r})$ and 
 \begin{equation*}
 h^0(\cup E_j , \mathcal{O}_{\cup_{j=1}^m E_j}(mA))=  \sum_{j=1}^m \dfrac{m^{r-1}}{(r-1)!} A^{r-1} \cdot B \actd [V] + o (m^{r}).
\end{equation*}  
 Hence,
 \begin{equation*}
 h^0(V, \mathcal{O}_V(mA-mB))\geqslant \dfrac{m^r}{r!} (A^r - r A^{r-1}\cdot B) \actd [V] + o(m^r).
 \end{equation*}
 In particular, this implies the required inequality.
\end{proof}

The next result is a key for our approach to controlling degrees of dominant rational maps.

\begin{thm} \label{thm_siu_gen}  
Let $i$ be an integer and $V$ be a closed subscheme of dimension $r$ in $X$. For any Cartier divisors $\alpha_1, \ldots, \alpha_i$ and $\beta$ which are big and nef on $V$, then there exists a constant $C>0$ depending only on $r$ and $i$ such that:
\begin{equation*}
 (\alpha_1 \cdot \ldots \cdot \alpha_i) \actd [V] \leqslant (r-i+1)^i\dfrac{ (\alpha_1 \cdot \ldots \cdot \alpha_i \cdot \beta^{r-i} \actd [V])}{(\beta^r \actd [V])}  \times   \beta^i \actd [V] \in \num_{r-i}(X).
\end{equation*}
\end{thm}
\begin{rem}Observe that $(\beta^n) >0$ since $\beta$ is big.
\end{rem}

\begin{proof}
By continuity, we can suppose that $\alpha_i$ and $\beta$ are ample Cartier divisors. 
We apply successively Siu's inequality by restriction to subschemes representing the classes $\alpha_2 \cdot \ldots \cdot \alpha_i \actd [V]$ , $\beta \cdot \alpha_3 \cdot \ldots \cdot \alpha_i \actd [V]$, $\ldots$, $\beta^{i-1} \cdot \alpha_i \actd [V]$:  

\begin{equation*}
\begin{array}{llll}
 \alpha_1 \cdot \alpha_2 \cdot \ldots \cdot \alpha_i \actd [V]&\leqslant& (r-i+1) \dfrac{(\alpha_1 \cdot \ldots \cdot \alpha_i \cdot \beta^{r-i} \actd [V])}{ (\beta^{r-i+1} \cdot \alpha_2 \cdot \ldots \cdot \alpha_i \actd [V]) } \times \beta \cdot \alpha_2 \cdot \ldots \cdot \alpha_i \actd [V], \\
\beta \cdot \alpha_2 \cdot \ldots \cdot \alpha_i \actd [V]&\leqslant & (r-i+1) \dfrac{(\beta^{r-i+1}\cdot \alpha_2 \cdot \ldots \cdot \alpha_i\actd [V]) }{(\beta^{r-i+2} \cdot \alpha_3 \cdot \ldots \cdot \alpha_i\actd [V]) } \times \beta^2 \cdot\alpha_3 \cdot \ldots  \cdot \alpha_i\actd [V] , \\
\ldots & & \ldots \\
\beta^{i-1} \cdot \alpha_i\actd [V] &\leqslant & (r-i+1) \dfrac{(\beta^{r-1} \cdot \alpha_i\actd [V])}{(\beta^r\actd [V])} \times \beta^i\actd [V].
\end{array}
\end{equation*}
This gives the required inequality:
\begin{equation*}
\alpha_1 \cdot \ldots \cdot \alpha_i\actd [V] \leqslant (n-i+1)^i \dfrac{(\alpha_1 \cdot \ldots \cdot \alpha_i \cdot \beta^{r-i}\actd [V])}{ (\beta^r\actd [V])} \times \beta^i\actd [V].
\end{equation*}
\end{proof}

\begin{cor} \label{cor_siu_gen} Let $i$ be an integer, then for any $a \in \Cc^i(X)$ and any big nef Cartier divisor $\beta$ on $X$, one has:
\[ a \leqslant (n-i+1)^i \dfrac{( a \cdot \beta^{n-i})}{(\beta^n)} \times \beta^i.\]
  
\end{cor}

\begin{proof} By linearity and stability by product, we just need to prove the inequality for $ a = D_1 \cdot \ldots \cdot D_{e_1 + i}  \in \IC^{e_1 + i}(X_1)$, where $D_i$ are ample Cartier divisors $X_1$, where $p_1: X_1 \to X$ is a flat proper morphism of relative dimension $e_1$. We apply Theorem \ref{thm_siu_gen} to $ a'= D_{e_1 + 1} \cdot \ldots \cdot D_{e_1 + i} \cdot Z$ and $\beta' = p_1^* \beta_{|Z}$ where $Z = D_1 \cdot \ldots \cdot D_{e_1}$. We obtain:
\begin{equation*}
 a \leqslant  (n-i+1)^i \dfrac{ (a \cdot p_1^*\beta^{n-i} )}{( p_1^* \beta^n \cdot Z)} \times  p_1^* b^{i} \cdot Z. 
\end{equation*}
As the restriction of $p_1$ on $Z$ is generically finite, by the projection formula, we get:
\begin{equation*}
a \leqslant (n-i + 1)^i \dfrac{(a\cdot  \beta^{n-i})}{(\beta^n)} \times \beta^i.
\end{equation*} 
\end{proof}

The previous inequality can be applied when we have positivity hypothesis on a birational model as follows. 

\begin{cor} \label{cor_comp_ccb}Let $X,Y$ be two normal projective varieties of dimension $n$. Let  $\beta $ be a class in $\Cc^i(Y)$, we suppose there exists a birational morphism $q: X \to Y$ and an ample Cartier divisor $A$ on $X$ such that $A^i  \leqslant q^* \beta$. Then there exists a class $\beta^* \in \num_i(X)_\mathbb{R} \cap \psef_{i}(X)$ such that for any class $\alpha \in \Cc^i(X)$, we have:
\[\alpha \leqslant ( \alpha \actd \beta^* ) \times \beta.\]   
\end{cor}

\begin{proof} We just have to set $\beta^* = \dfrac{(n-i+1)^i }{ (A^n)} q_* \psi_X( A^{n-i}) $. 

\end{proof}

\begin{rem} We conjecture that for any basepoint free class $a \in \Cc^i(X)$ and any big nef divisor $b$, one has 
\begin{equation}
a \leqslant \left ( \begin{array}{ll}
n \\
i
\end{array} \right )  \dfrac{  (a \cdot b^{n-i})}{(b^n)}b^i.
\end{equation}
One can show that this inequality (if true) is optimal since equality can happen when $X$ is an abelian variety.
\end{rem}

\subsection{Norms on numerical classes}

In this section, the positivity properties combined with Siu's inequality allows us to define some norms on $\num_i(X)_\mathbb{R}$ and on $\numd^i(X)_\mathbb{R}$.

\subsubsection{Norms on $\num_i(X)_\mathbb{R}$}

Let $i \leqslant n$ be an integer and let $\gamma \in \Cc^{i}(X)$ be a basepoint free class on $X$. Any cycle $z \in \num_{i}(X)_\mathbb{R}$ can be written $z =z^+ - z^-$ where $z^+ $ and $z^-$ are pseudo-effective. We define :

\begin{equation} \label{formule_norme}
F_\gamma(z) := \inf_{ \substack{ z = z^+ - z^- \\
z^+, z^- \in \psef_i(X)
}} \{ (\gamma \actd z^+)  + (\gamma \actd z^-)  \} .
\end{equation}

\begin{prop} \label{prop_norm_pliant} 
For any class $\gamma \in \Cc^{i}(X)$ lying in the interior of the basepoint free cone, the function $F_\gamma$ defines a norm on $\num_i(X)_\mathbb{R}$. In particular, if we fix a norm $||\cdot ||_{\num_i(X)_\mathbb{R}}$ on $\num_i(X)_\mathbb{R}$, there exists a constant $C> 0$ such that for any pseudo-effective class $z \in \psef_i(X)$, one has:
\begin{equation}
 \dfrac{1}{C} || z||_{\num_i(X)_\mathbb{R}} \leqslant (\gamma \actd z) \leqslant C || z ||_{\num_i(X)_\mathbb{R}}.
\end{equation}
\end{prop}

\begin{proof}

\smallskip
The only point to clarify is that $F_\gamma(z) = 0$ implies $z = 0$. 
Observe that Proposition \ref{prop_nullite_psef} implies the result for $z \in \psef_i(X)$.
In general, pick any two sequences
$ (z_p^+)_{p\in \mathbb{N}}$ and $(z_p^-)_{p \in \mathbb{N}}$ in $\psef_i(X)$ such that $z = z_p^+ - z_p^-$ and such that 
$ \gamma \cdot z_p^+  + \gamma \cdot z_p^-  {\longrightarrow} 0$.
Since $z_p^+$ and $z_p^-$ are pseudo-effective and $\gamma$ is basepoint free, it follows from Theorem \ref{thm_classes_pliantes}.(v) that 
\[\lim_{p\rightarrow + \infty} (\gamma \cdot z_p^+)  = \lim_{p \rightarrow + \infty} (\gamma \cdot  z_p^-) = 0 .\]
As $\gamma$ lies in the interior of $\Cc^i(X)$, given any $\beta$ in $\Cc^i(X)$, one has that $C \gamma - \beta$ is still in $\Cc^i(X)$ for some sufficently large constant $C>0$. 
Intersecting with the pseudo-effective classes $z_p^+$ and $z_p^-$ and using Theorem \ref{thm_classes_pliantes}.(v), we have $ \lim_{p \rightarrow \infty}(\beta \actd z_p^+)=\lim_{p \rightarrow \infty}(\beta \actd z_p^-)= 0 $, thus $(\beta \actd z)= 0$. Since the basepoint free cone $\Cc^i(X)$ generates all $\numd^i(X)$ by Theorem \ref{thm_classes_pliantes}.(i), we conclude that $z = 0$ as required.
\end{proof}

\subsubsection{Norms on $\numd^i(X)_\mathbb{R}$}

\begin{defi} 
We define the subcone $\Cc_0^i(X)$ of $\Cc^i(X)$ as the classes $\alpha \in \Cc^i(X)$ such that for any birational map $q : X' \to X$, there exists an ample Cartier divisor $A$ on $X'$ such that $q^* \alpha \geqslant A^i$.   
\end{defi}

\begin{prop} \label{prop_ccb}When $X$ is smooth, the cone $\Cc_0^1(X)$ is equal to the big nef cone. In particular $\Cc_0^i$ is neither closed nor open in general.
\end{prop}

\begin{proof}
Take $\alpha \in \numd^1(X)_\mathbb{R}$ a big nef divisor. Then for any birational map $q : X' \to X$ and any ample Cartier divisor $A$, one has by Theorem \ref{thm_siu_gen} applied to $A$ and $q^*\alpha$:
\begin{equation*}
A \leqslant n \dfrac{(A \cdot q^*\alpha^{n-1})}{(\alpha^n)} q^* \alpha.
\end{equation*}
Hence, $\alpha \in \Cc^1_0(X)$. 
Conversely, take a class $\alpha \in \Cc^1_0(X)$, then there exists an ample divisor $A$ on $X$ such that $\alpha \geqslant A$. 
Since ample divisors are big, we have that $\alpha$ is big. 
Moreover, since $\Cc^1(X) = \nefd^1(X) \cap \psefd^1(X)$, we have that $\alpha$ is big and nef as required.
\end{proof}

\begin{prop} \label{prop_bpf0_interior}
The cone $\Cc_0^i(X)$ is a convex open subset of $\Cc^i(X)$ that contains the classes induced by products of big nef divisors.
\end{prop}

\begin{proof}
The cone $\Cc_0^i(X)$ contains the products of big and nef Cartier divisors. 
The fact that $\Cc_0^i(X)$ is convex is a consequence of Siu's inequality. We take two elements $\alpha$ and $\beta $ in $\Cc_0^i(X)$ and any birational map $q: X' \to X$. 
By definition, there exists some ample Cartier divisors $A$ and $B$ on $X'$ such that $q^*\alpha \geqslant A^i$ and $q^*\beta \geqslant b^i$. 
As $A$ and $B$ are ample, there is a constant $C>0$ such that $A^i \geqslant C b^i$ using the generalization of Siu's inequality (Theorem \ref{thm_siu_gen}). This proves that $q^* (t \times \alpha + (1-t)\times  \beta) \geqslant ( t C + (1-t) )\times b^i$ for any $t \in [ 0, 1]$. Hence $t\times \alpha + (1-t) \times \beta \in \Cc_0^i(X)$ and the cone $\Cc_0^i(X)$ is convex.

\black 
We prove that $\Cc_0^i(X)$ is an open subset of $\Cc^i(X)$. We take $\alpha \in \Cc_0^i(X)$. We take any ample Cartier divisor $H_X$ on $X$ such that $\alpha - t H_X^i$ is in $\Cc^i(X)$ for small $t> 0$. We just need to show that $\alpha - tH_X^i$ stays in $\Cc_0^i(X)$ when $t$ is small enough. 
Let $q: X' \to X$ be a birational map where $X'$ is projective and normal. By definition of $\alpha$, there exists an ample Cartier divisor $A$ on $X'$ such that $q^*\alpha  \geqslant  A^i$. By Siu's inequality, there exists a constant $C$ such that:
\begin{equation*}
q^* H_X^i \leqslant C \dfrac{(A^i \cdot q^* H_X^{n-i})}{(H_X^n)}\times A^i. 
\end{equation*} 
This implies the inequality:
\begin{equation} \label{eq_minoration_1}
q^* \beta - t q^* H_X^i \geqslant \left  ( 1 - t C \dfrac{A^i \cdot H_X^{n-i}}{H_X^n} \right ) \times A^i .
\end{equation}
As $A^i \leqslant q^* \alpha$, we have the following upper bound:
\begin{equation*}
(A^i \cdot H_X^{n-i} )\leqslant (q^* \alpha \cdot q^* H_X^{n-i}).
\end{equation*}
We get the following minoration which depends only on $\alpha$ and $H_X$:
\begin{equation} \label{eq_minoration_2}
1- t \dfrac{C (\alpha \cdot H_X^{n-i}) }{(H_X^n)} \leqslant  1 - t\dfrac{C (A^i \cdot H_X^{n-i})}{(H_X^n)}.
\end{equation} 
Using \eqref{eq_minoration_1} and \eqref{eq_minoration_2}, one gets that for $t <  \dfrac{(H_X^n)}{ C (\alpha \cdot H_X^{n-i})}$, the class $\alpha - t H_X^i $ is in $ \Cc_0^i(X)$.

\end{proof}

\begin{rem} The cone $\Cc_0^i(X)$ is not always equal to the cone generated by complete intersections. Following \cite[Example 9.6]{lehmann_xiao}, there exists a smooth toric threefold such that the cone generated by complete intersections in $\num_1(X)_\mathbb{R}$ is not convex, so it cannot be equal to $\Cc_0^2(X)$ using the following proposition. 
\end{rem}

Let $X$ be a normal projective variety of dimension $n$. Any class $\alpha \in \numd^i(X)_\mathbb{R}$ can be decomposed as $\alpha^+ - \alpha^-$ where $\alpha^+$ and $\alpha^-$ are basepoint free classes. For any $\gamma \in \Cc_0^{n-i}(X)$, we define the function:
\begin{equation}
G_\gamma (\alpha) :=  \inf_{ \substack{ \alpha = \alpha^+ - \alpha^- \\
\alpha^+, \alpha^- \in \Cc^i(X)
}} \{ (\gamma \cdot \alpha^+)  + (\gamma \cdot \alpha^-)  \} .
\end{equation}

\begin{prop} 
For any $\gamma \in \Cc_0^{n-i}(X)$, the function $G_\gamma$ defines a norm on $ \numd^i(X)_\mathbb{R}$. In particular, for any norm $|| \cdot ||_{\numd^i(X)_\mathbb{R}}$ on $\numd^i(X)_\mathbb{R}$, there is a constant $C>0$ such that for any class $\alpha \in \Cc^i(X)$:
\begin{equation*}
\dfrac{1}{C} || \alpha ||_{\numd^i(X)_\mathbb{R}} \leqslant (\gamma \cdot \alpha) \leqslant C || \alpha ||_{\numd^i(X)_\mathbb{R}}.
\end{equation*} 
\end{prop}

\begin{proof} 
 The only fact which is not immediate is the fact that $G_\gamma(\alpha) = 0$ implies $\alpha = 0$. We are reduced to treat the case where $\alpha \in \Cc^i(X)$.

\smallskip

Suppose first that $X$ is smooth.
Since $\gamma $ belongs to the interior of the basepoint free cone by Proposition \ref{prop_bpf0_interior}, one has that for any basepoint free class $\beta \in \BPF^{n-i}(X)$, there exists a constant $C>0$ such that:
\begin{equation*}
C || \beta || \gamma - \beta \in \BPF^{n-i}(X).
\end{equation*}
In particular, since $\alpha$ is nef, one has:
\begin{equation*}
0= G_\gamma(\alpha) = C || \beta|| (\gamma \cdot \alpha) \geqslant ( \beta \cdot \alpha ) \geqslant 0.
\end{equation*}
Hence $(\beta \cdot \alpha) = 0$ for any basepoint free class $\beta \in \BPF^{n-i}(X)$ and $\alpha = 0 \in \numd^i(X)_\mathbb{R}$ since the basepoint free cone generates all $\numd^{n-i}(X)_\mathbb{R}$ by Theorem \ref{thm_classes_pliantes}.$(i)$. 

\medskip

Suppose that $X$ is not smooth. Fix an ample Cartier divisor $H_{X}$ on $X$.
Take an alteration $\pi : X' \to X$ of $X$. Since the morphism $\pi^* : \numd^i(X)_\mathbb{R} \to \numd^i(X')_\mathbb{R}$ is injective, we are reduced to prove that $\pi^* \alpha = 0$. 
Consider $\beta \in \Cc^{n-i}(X)$, we have by the projection formula that:
\begin{equation*}
(\pi^* \gamma \cdot \pi^* \alpha) = (\alpha \cdot \gamma).
\end{equation*}
Since $\gamma$ belongs to the interior of the basepoint free cone, there exists a constant $C>0$ such that:
\begin{equation*}
H_X^{n-i} \leqslant C \gamma.
\end{equation*} 
In particular, this implies that:
\begin{equation*}
(\pi^*H_X^{n-i} \cdot \pi^*\alpha ) = ( H_X^{n-i} \cdot \alpha)=0.
\end{equation*}
Since $\pi^*H_X$ is a big nef Cartier divisor, the class $\pi^* H_X^{n-i}$ belongs to $\Cc^{n-i}_0(X')$ by Proposition \ref{prop_bpf0_interior}, hence $\pi^* \alpha = 0$ by the previous argument. 

\end{proof}

\begin{rem} In fact, the above proof gives a stronger statement:
for any generically finite morphism $q : X' \to X$ and any $\gamma \in \Cc_0^{n-i}(X)$, the function $G_{q^* \gamma}$ defines a norm on $\numd^i(X')_\mathbb{R}$. 
\end{rem}

\section{Relative numerical classes} \label{section_relative}

\subsection{Relative classes}

In this section, we fix $q:X \to Y$ a surjective proper morphism between normal projective varieties where $\dim X=n$, $\dim Y = l$ and we denote by $e = \dim X - \dim Y$ the relative dimension of $q$.

\begin{defi} 
The abelian group $\num_i(X/Y)$ is the subgroup of $\num_i(X)$ generated by classes of subvarieties $V $ of $X$ such that the image $q(V)$ is a point in $Y$. 
\end{defi} 
Observe that by definition, there is a natural injection from $\num_i(X/Y)$ into $\num_i(X)$:
\begin{equation*}
\xymatrix{ 0 \ar[r] & \num_i(X/Y)  \ar[r]& \num_i(X) .}
\end{equation*}
 
\begin{defi} 
The abelian group $\numd^i(X/Y)$ is the quotient of $Z^i(X)$ by the equivalence relation $\equiv_Y$ where $\alpha \equiv_Y 0$ if for any cycle $z \in Z_{i}(X)$ whose image by $q$ is a collection of points in $Y$, we have $(\alpha \actd z)= 0$.
\end{defi} 

 Therefore, one has the following exact sequence:
\begin{equation*}
\xymatrix{ \numd^i(X)  \ar[r]& \numd^i(X/Y) \ar[r] & 0 .}
\end{equation*}

As before, we write $\num_i(X/Y)_\mathbb{R} = \num_i(X/Y) \otimes_\mathbb{Z} \mathbb{R}$, $\numd^i(X/Y)_\mathbb{R} = \numd^i(X/Y) \otimes \mathbb{R} $, $\num_\bullet(X/Y) = \oplus \num_i(X/Y)$ and $\numd^\bullet(X/Y) = \oplus \numd^i(X/Y)$.    

\begin{prop} \label{prop_pairing_rel}
The abelian groups $\num_i(X/Y)$ and $\numd^i(X/Y)$ are torsion free and of finite type. Moreover, the pairing $\num_i(X/Y)_\mathbb{Q} \times \numd^i(X/Y)_\mathbb{Q} \to \mathbb{Q}$ induced by the pairing $\num_i(X)_\mathbb{Q} \times \numd^i(X)_\mathbb{Q} \to \mathbb{Q}$ is perfect.
\end{prop} 

\begin{proof}
Since $\num_i(X/Y)$ is a subgroup of $\num_i(X)$, it is torsion free and of finite type. The group $\numd^i(X/Y)$ is also torsion free.
Indeed pick $\alpha \in Z^i(X)$ such that $ p \alpha \equiv_Y 0$ for some integer $p$, then for any cycle $z$ whose image by $q$ is a union of points, we have $(p \alpha \actd z) = p (\alpha \actd z) = 0$ hence $\alpha \equiv_Y 0$.  
Finally, since there is a surjection from $\numd^i(X)$ to $\numd^i(X/Y)$, the group $\numd^i(X/Y)$ is also of finite type.

\smallskip
Let us show that the pairing is well defined and non degenerate.
Take a cycle $z \in Z_i(X)_\mathbb{Q}$ such that $q(z)$ is a finite number of points in $Y$, then if $\alpha \in \numd^i(X)$ such that its image is $0$ in $\numd^i(X/Y)$, then $(\alpha \actd z)= 0$ and the pairing $\numd^i(X/Y) \times \num_i(X/Y) \to \mathbb{Z}$ is well-defined.
Let us suppose that for any $\alpha \in \numd^i(X/Y)_\mathbb{Q}$, $(\alpha \actd z)=0$. This implies that for any $\beta \in \numd^i(X)$, the intersection product $(\beta \actd z) = 0$, thus $z \equiv 0$. 
Conversely, suppose that $(\alpha \actd z)=0$ for any $z \in \num_i(X/Y)$, then by definition $\alpha \equiv_Y 0$.  

\end{proof}

\begin{ex} When $Y$ is a point, we have $\num_i(X/Y) = \num_i(X)$ and $\numd^i(X/Y) = \numd^i(X)$.
  \end{ex}

\begin{ex} If the morphism $q : X \to Y$ is finite, then we have $\numd^0(X/Y)_\mathbb{Q} = \num_0(X/Y)_\mathbb{Q} = \mathbb{Q}$ and $\numd^i(X/Y) = \num_i(X/Y) = \{0 \}$ for $i\geqslant 1$ since $X$ is irreducible.
\end{ex}

\begin{ex}When $i=1$, the group $\num_1(X/Y)$ is generated by curves contracted by $q$ so that $\numd^1(X/Y) $ is the relative Neron-Severi group and its dimension is the relative Picard number (see \cite[Section 2.2, Example 2.16]{kollar_mori}). 
\end{ex}  
 
\begin{rem} When $i$ is greater than the relative dimension, the relative classes might not be trivial. For example if $q : X \to Y$ is a birational map, then $e=0$ but the space $\numd^1(X/Y)_\mathbb{R}$ is generated by classes of exceptional divisors of $q$.  
\end{rem}

\begin{prop} 
The intersection product on $\numd^\bullet(X)$ induces a structure of algebra on $\numd^\bullet(X/Y) $. Moreover, the action from $\numd^\bullet(X)$ on $\num_\bullet(X)$ induces an action from $\numd^\bullet(X/Y)$ on $\num_\bullet(X/Y)$, so that the vector space $\num_\bullet(X/Y)_\mathbb{R}$ becomes  a $\numd^\bullet(X/Y)_\mathbb{R}$-module. 
\end{prop}

\begin{proof} 
Observe that if $z \in Z_i(X)$ such that $q(z)$ is a union of points in $Y$ and $\alpha \in \numd^l(X)$, then $\alpha \actd z$ lies in $\num_{i-l}(X/Y)$. Indeed, by definition, the class $\alpha \actd z$ is represented by a cycle supported in $z$, so its image by $q$ is a collection of points in $Y$.

Let us now prove that the product is well-defined in $\numd^\bullet(X/Y)$. Take $\alpha \in \numd^i(X)$ such that $\alpha = 0$ in $\numd^i(X/Y)$ and $\beta \in \numd^l(X)$, we must prove that $\alpha \cdot \beta =0$ in $\numd^i(X/Y)$. Pick a cycle $z \in Z_{i+l}(X)$ whose image by $q$ is a collection of points, by the properties of the intersection product, $((\alpha \cdot \beta )\actd z)= ( \alpha \actd (\beta \actd z))$. As $\beta \actd z$ is in $\num_i(X/Y)$, we get that $( (\alpha \cdot \beta) \actd z )= 0$ as required.

\end{proof}

As an illustration, we give an explicit description of these groups in a particular example. 
 \begin{prop} \label{prop_calcul_rel_group}
 Suppose $q : X= \Pg(E) \to Y$ where $E$  is a vector bundle of rank $e+1$ on $Y$. Then for any integer $0  \leqslant i \leqslant e$, one has:
 \begin{equation*}
 \num_i(X/Y)_\mathbb{Q} = \mathbb{Q} \  \xi^{e-i} \actd q^* [pt] ,
\end{equation*} 
\begin{equation*}
\numd^i(X/Y)_\mathbb{Q} = \mathbb{Q} \ \xi^{i},
\end{equation*} 
where $\xi = c_1(\mathcal{O}_{\Pg(E)}(1))$.
 \end{prop}

\begin{proof} Since the pairing $\numd^i(X/Y)_\mathbb{Q} \times \numd_i(X/Y)_\mathbb{Q} \to \mathbb{Q}$ is non degenerate and since $(\xi^i \actd( \xi^{e-i} \actd q^*[pt])) = 1 $, the second equality is an immediate consequence of the first one.
We suppose first that $i > 0$. 
Pick $\alpha \in Z_i(X)$ which defines a class in $\numd_i(X/Y)_\mathbb{Q}$. 
Using \cite[Theorem 3.3.(b)]{fulton}, $\alpha $ is rationally equivalent to $\sum_{ e-i \leqslant j \leqslant e} \xi^j \actd  q^* \alpha_j $ where $\alpha_j$ is an element of the Chow group $ A_{i-e+j}(Y)_\mathbb{Q}$. 
Since the image of $\alpha $ by $q$ is a union of points in $Y$, we have that $q_* \alpha = 0$ in $A_i(Y)_\mathbb{Q}$. 
Observe that 
\begin{equation*}
q_* (\xi^e \actd q^* \alpha_e) = \alpha_e,
\end{equation*}
and that for any $j < e$, one has that
\begin{equation*}
q_* (\xi^j \actd q^* \alpha_j ) = 0 
\end{equation*}
since the support of the cycle $\alpha_i$ is of dimension $i-e+j < i$ and $q_*(\xi^e \actd q^* \alpha_j) $ belongs to $A_{i}(Y)$. 
 Hence the conditions $q_* \alpha = 0$ implies that $\alpha_{e} = 0$ in $A_{i}(Y)_\mathbb{Q}$. 
 Since $\xi^j \actd \alpha $ defines also a class in $\numd_{i-j}(X/Y)_\mathbb{Q}$, this implies also that $\alpha_{e-j} = 0$ in $A_{i-e+j}(Y)_\mathbb{Q}$ for any $j < i$. We have finally that in $\num_i(X/Y)_\mathbb{Q}$:
\begin{equation*}
\alpha = \xi^{e-i} \actd q^* \alpha_{e-i}.  
\end{equation*}
Since $\alpha_{e-i}$ belongs to $A_0(Y)_\mathbb{Q}$ and $\num_0(Y) = \mathbb{Q} [pt]$, the $\mathbb{Q}$-module $\num_i(X/Y)$ is generated by $\xi^{e-i} \actd q^*[pt]$ for $i> 0$.

\medskip

For $i = 0$, the groups $\num_0(X)_\mathbb{Q}$ and $\num_0(X/Y)_\mathbb{Q}$ are isomorphic to $\mathbb{Q}$, so we get the desired conclusion.

\end{proof}

\subsection{Pullback and pushforward}
\label{section_pullback}
In this section, we fix any two (proper) surjective morphisms $q_1 : X_1 \to Y_1$, $q_2 : X_2 \to Y_2$ between normal projective varieties.
To simplify the notation, we write $X_1/_{q_1} Y_1 \relmor{f}{g} X_2/_{q_2} Y_2$ when we have two regular maps $f : X_1 \to X_2$ and $g : Y_1 \to Y_2$ such that $q_2 \circ f = g \circ q_1$ and we shall say that $X_1/_{q_1} Y_1  \relmor{f}{g} X_2/_{q_2}Y_2$ is a morphism. 
When $f : X_1 \dashrightarrow X_2$ and $g : Y_1 \dashrightarrow Y_2$ are merely rational maps, then we write $X_1/_{q_1} Y_1 \relrat{f}{g} X_2/_{q_2} Y_2$ and we shall call it a rational map.

 \begin{prop} Let $ X_1/_{q_1} Y_1  \relmor{f}{g} X_2/_{q_2}Y_2$ be a morphism. Then the morphism of abelian groups $f_* : \num_i(X_1) \to \num_i(X_2)$ induces a morphism of abelian groups $f_* : \num_i(X_1/Y_1) \to \num_i(X_2/Y_2)$.  
 \end{prop}
 
\begin{proof} Take a cycle $z \in Z_i(X_1)$ such that $q_1(z)$ is a union of points of $Y_1$.
Then the image of the cycle $z$ by $q_2 \circ f$ is also a union of points of $Y_2$ due to the fact that $ q_2 \circ f = g \circ q_1$. Hence $f_*$ maps $\num_i(X_1/Y_1)$ to $\num_i(X_2/Y_2)$.  
\end{proof}
 
 \begin{prop} Let $ X_1/_{q_1} Y_1  \relmor{f}{g} X_2/_{q_2}Y_2$ be a morphism. Then the morphism of graded rings $f^* : \numd^\bullet(X_1) \to \numd^\bullet(X_2)$ induces a morphism of graded rings $f^* : \numd^\bullet(X_1/Y_1)_\mathbb{Q} \to \numd^\bullet(X_2/Y_2)_\mathbb{Q}$. 
 \end{prop}
 \begin{proof} This results follows immediately by duality from the previous proposition since the pairing $\numd^i(X_i/Y_i)_\mathbb{Q} \times \num_i(X_i/Y_i)_\mathbb{Q} \to \mathbb{Q}$ is non degenerate. 
 \end{proof}
\subsection{Restriction to a general fiber and relative canonical morphism}

Recall that $\dim X = n$, $\dim Y = l$ and that the relative dimension of $q: X\to Y$ is $e$. 
\begin{prop}\label{prop_class_restriction}
There exists a unique class $\alpha_{X/Y} \in \numd^l(X)_\mathbb{Q}$ satisfying the following conditions.
\begin{enumerate}
\item The image $\psi_X(\alpha_{X/Y}) $ belongs to the subspace $\num_{e}(X/Y)_\mathbb{Q}$ of $\numd_e(X)_\mathbb{Q}$.
\item For any class $\beta \in \num_l(X)_\mathbb{Q}$, $q_* \beta = (\alpha_{X/Y} \actd \beta) \  [Y] $.
\end{enumerate}
Moreover, for any open subset $V$ of $Y$ such that the restriction $q$ to $U=q^{-1}(V)$ is flat, and for all $y \in V$ and any irreducible component $F$ of the scheme-theoretic fiber $X_y$, we have:
\begin{equation*}
\psi_X(\alpha_{X/Y}) = [X_y] = r [F],
\end{equation*}
where $r$ is a rational number which only depends on $F$ and 
where $[X_y]$ (resp. $[F]$) denotes the fundamental class of $X_y$ (resp. $F$) viewed as an element of $\numd_e(X/Y)$.

More explicitly, the class $\alpha_{X/Y}$ is given by
$$  \alpha_{X/Y} = \dfrac{1}{(H_Y^l)} q^* H_Y^l \in \numd^l(X/Y)_\mathbb{Q}, $$ 
 where $H_Y$ is an ample divisor on $Y$.
\end{prop}

\begin{rem} Recall that by generic flatness (see \cite[Theorem 5.12]{FGA}), one can always find an open subset $V$ of $Y$ such that the restriction of $q$ to $q^{-1}(V)$ is flat over $V$.
\end{rem}
\begin{proof}
Fix an ample Cartier divisor $H_Y$ on $Y$, we set
\begin{equation*}
\alpha_{X/Y} := \dfrac{1}{(H_Y^l)} q^* H_Y^l \in \numd^l(X)_\mathbb{Q}.
\end{equation*}

Write the class $H_Y^l$ in $A_0(Y)$ as:
\begin{equation} \label{eq_gen_point}
H_Y^l = \sum a_j [p_j]
\end{equation}
where $p_j \in V(\C)$ are points in $V$ and $a_j$ are positive integers satisfying $\sum a_j = (H_Y^l)$.
By the projection formula (Theorem \ref{thm_num_sum_up}.(iv)), the class $\alpha_{X/Y}$ satisfies (i) and (ii)  . 
Let us show that any class satisfying (i) and (ii) is unique. 
Suppose there is another one $\alpha' \in \numd^l(X)_\mathbb{Q}$. 
Then for any class $\beta \in \num_l(X)_\mathbb{Q}$, $((\alpha_{X/Y} - \alpha') \actd \beta ) =0$ so that $\alpha = \alpha'$ since the pairing $\numd^l(X)_\mathbb{Q} \times \num_l(X)_\mathbb{Q} \to \mathbb{Q}$ is non degenerate.
\smallskip

Let us prove the last assertion. 
 By generic flatness \cite[Theorem 5.12]{FGA}, 
 Let $V$ be an open subset of $Y$ such that the restriction $q_{|q^{-1}(V)} : q^{-1}(V) \to V $ is flat and such that the dimension of every fiber is $e$. 
Since $H_Y$ is ample, we can find some hyperplanes of $H_i \subset Y$ such that $H_1 \cap  \ldots \cap H_l$ represents the class $H_Y^l$ and such that $H_1 \cap \ldots \cap H_l \subset V$. 
In particular, by \cite[Proposition 2.3.(d)]{fulton}, the pullback $q^* H_Y^l$ is represented by a cycle in the fiber of $H_1 \cap \ldots \cap H_l$. Denote by $ u : V \to Y$ and $g: U \to X$ the inclusion maps of $V$ and $U$ into $Y$ and $X$ respectively. 
The morphisms $u$ and $g$ are open embedding hence are flat. Moreover we have the following commutative diagram. 
\begin{equation*}
\xymatrix{ U \ar[d]^{q_{|U}}\ar[r]^{u} & X \ar[d]^{q} \\
 V \ar[r]^{g} & Y}
\end{equation*}
 
Using \cite[Example 2.4.2]{fulton}, one has that for any $\beta \in A_{l}(X)$: 
\begin{equation*}
(q^*  H_Y^l \actd \beta)  =   ( q_{|U}^* g^*(H_Y^l) \actd u^*\beta).
\end{equation*} 
Using \eqref{eq_gen_point}, one obtains in $A_e(X)$:
\begin{equation*}
 q_{|U}^* g^* H_Y^l =  q_{|U}^* g^*(\sum a_j [p_j]) = \sum a_j [q^{-1}(p_j)],
\end{equation*}
which is well-defined since the restriction of $q$ on $U$ is flat. 
By \cite[Theorem 10.2]{fulton}, we have that $[X_{p_j}] = [X_y] \in \num_e(X)$ for any $p_j, y \in V$. In particular, we have:
\begin{equation*}
\psi_{X}(q^* H_Y^l) = (\sum a_j)  [X_y] = (H_Y^l) \ [X_y] \in \num_e(X),
\end{equation*}
where $y $ is a point in $V$, which proves that $\psi_X(\alpha_{X/Y}) = [X_y]$ in $\num_e(X)_\mathbb{Q}$ for any point $y$ in $V$.
By the Stein factorization theorem, there exists a morphism $q' : X \to Y'$ with connected fibres and a finite morphism $f : Y' \to Y$ such that $q' = q \circ f$. 
Since $(H_Y^l)[X_y] = q^* H_Y^l = q'^*f^* H_Y^l$ and since $f^* H_Y^l \in \numd^l(Y')_\mathbb{R}$ which is canonically isomorphism to $\mathbb{R}$, we have that $f^* H_Y^l = p \cdot [y'] \in \numd^l(Y')_\mathbb{R}$ where $p$ is an integer and where $[y']$ is a general point in $f^{-1}(y)$.
 We have thus proven that:
 \begin{equation*}
 [X_y] = \dfrac{p}{(H_Y^l)} [q'^{-1}(y')] \in \numd_e(X),
 \end{equation*}
 and $q'^{-1}(y')$ is an irreducible component of $X_y$ as required.
\end{proof}

The class previously constructed allows us to define a restriction morphism.
\begin{defi} Suppose that $\dim Y = l$ and that $H_Y$ is an ample Cartier divisor on $Y$, then we define $\res_{X/Y}: \numd_\bullet(X)_\mathbb{Q} \to \num_{\bullet -l}(X/Y)_\mathbb{Q}$ by setting:
\begin{equation*}
\res_{X/Y} (\beta) := \dfrac{1}{(H_Y^l)} q^* H_Y^l \actd \beta =  \alpha_{X/Y} \actd \beta.
\end{equation*}
This morphism does not depend on the choice of $H_Y$.
\end{defi}
We shall denote by $\res_{X/Y}^* : \beta \in \numd^{\bullet}(X/Y)_\mathbb{Q} \rightarrow  \alpha_{X/Y} \cdot \beta \in \numd^{\bullet + l}(X)_\mathbb{Q}$ the dual morphism induced by $\res_{X/Y}$ with respect to the pairing $\numd^\bullet(X/Y)_\mathbb{Q} \times \numd_\bullet(X/Y)_\mathbb{Q} \to \mathbb{Q}$.

\begin{prop}\label{prop_restriction} Recall that $\dim Y=l$. The following properties are satisfied.
\begin{enumerate}
\item For any class $\alpha \in \numd^\bullet(X)_\mathbb{Q}$, one has:
\begin{equation*} 
\psi_X \circ \res_{X/Y}^* (\alpha) = \res_{X/Y} \circ  \psi_X (\alpha).
\end{equation*}
\item For any morphism $X'/_{q'}Y' \relmor{f}{g} X/_q Y$ where $\dim X' = \dim X=n$ and $\dim Y' = \dim Y = l$ such that the topological degree of $g$ is $d$, we have for any $\alpha \in \numd^{i-l}(X/Y)_\mathbb{Q}$:
\begin{equation*}
 d \times \res_{X'/Y'}^* \circ f^* \alpha = f^* \circ \res_{X/Y}^* \alpha.
\end{equation*} 
\end{enumerate}
\end{prop}

The definition of the restriction morphism gives a natural way to generalize the definition of the canonical morphism $\psi_X : \numd^i (X) \to \num_{n-i}(X)$ to the relative case.

\begin{defi} 
Recall that the relative dimension of the morphism $q: X \to Y$ is $e$. For any integer $i\geqslant 0$, we define the canonical morphism $\psi_{X/Y}$ by:
\begin{equation*}
\psi_{X/Y} := \psi_X \circ \res_{X/Y}^* :  \beta \in \numd^{i}(X/Y)_\mathbb{Q} \rightarrow   \psi_X(\alpha_{X/Y} \cdot \beta) \in\num_{e-i}(X/Y)_\mathbb{Q}. 
\end{equation*}
\end{defi}

\begin{rem} When $i> e$ by convention the map $\psi_{X/Y}$ is zero.
\end{rem}

We give here a situation where this map is an isomorphism.
\begin{prop} Suppose $q: X \to Y$ is a smooth morphism of relative dimension $e$, then for any integer $0 \leqslant i \leqslant e$, the map $\psi_{X/Y}: \numd^i(X/Y)_\mathbb{Q} \to \num_{e-i}(X/Y)_\mathbb{Q}$ is an isomorphism.
\end{prop}

\begin{proof}
Since the pairing $\numd^i(X/Y)_\mathbb{Q} \times \num_{i}(X/Y)_\mathbb{Q} \to \mathbb{Q}$ is perfect by Proposition \ref{prop_pairing_rel}, we have that the dual morphism $\psi_{X/Y}^* : \numd^{e-i}(X/Y)_\mathbb{Q} \to \num_i(X/Y)_\mathbb{Q}$ of $\psi_{X/Y}$ is surjective whenever $\psi_{X/Y}: \numd^i(X/Y)_\mathbb{Q} \to \num_{e-i}(X/Y)_\mathbb{Q}$ is injective. 
We are thus reduced to prove the injectivity of $\psi_{X/Y} : \numd^i(X/Y)_\mathbb{Q} \to \num_{e-i}(X/Y)_\mathbb{Q}$. 
Take $a \in \numd^i(X/Y)_\mathbb{Q}$ such that $\psi_{X/Y}(a) = 0$, and choose a class $\alpha \in \numd^i(X)_\mathbb{Q}$ representing $a$.  
 We fix a subvariety $V$ of dimension $i$ in a fiber $X_y$ of $q$ where $y$ is a point in $Y$. 
 We need to prove that $(\alpha \actd [V]) = 0$.
\smallskip 
 
By Proposition \ref{prop_class_restriction}, the condition $\psi_{X/Y}(\alpha) = 0$ implies that:
\begin{equation*}
\alpha \actd [X_y] = 0 \in \num_{e-i}(X)_\mathbb{Q}.
\end{equation*} 
As the morphism $q: X \to Y$ is smooth, the fiber $X_y$ over $y$ is smooth. 
By Theorem \ref{thm_cano_mor_prop}, there exists a class $\beta \in \numd^{e-i}(X_y)_\mathbb{Q}$ such that:
\begin{equation*}
\beta \actd [X_y] = [V].
\end{equation*}
In particular, we get:
\begin{equation*}
(\alpha \actd [V] )= (\alpha \actd (\beta \actd [X_y]))= (\beta \actd (\alpha \actd [X_y])) = 0
\end{equation*}
as required.
\end{proof}

\begin{ex} If $X = \Pg(E)$ where $E$ is a vector bundle on $Y$, then Proposition \ref{prop_calcul_rel_group} implies that $\psi_{X/Y} : \numd^i(X/Y)_\mathbb{Q} \to \num_{e-i}(X/Y)_\mathbb{Q}$ is an isomorphism for any integer $0 \leqslant i \leqslant e$.
\end{ex}

\begin{ex} 
If $X$ is the blow-up of $\Pg^1 \times \Pg^1$ at a point and $q  $ is the projection from $\Pg^1\times \Pg^1$ to the first component $Y=\Pg^1$ composed with the blow-down from $X$ to $\Pg^1 \times \Pg^1$. Then the morphism $\psi_{X/Y} : \numd^0(X/Y)_\mathbb{Q} \to \num_1(X/Y)_\mathbb{Q}$ is not surjective and $\psi_{X/Y}: \numd^1(X/Y)_\mathbb{Q} \to \num_0(X/Y)_\mathbb{Q}$ is not injective. 
\end{ex}

\section{Application to dynamics} \label{section_applications}

In this section, we shall consider various normal projective varieties $X_j$ and $Y_j$ respectively of dimension $n$ and $l$ and we write $e = n-l$
 Recall from Section \ref{section_pullback} that the notation $X_j /_{q_j} Y_j$ means that $q_j : X_j \to Y_j$ is a surjective morphism of relative dimension $e$ and that $X/_{q} Y  \relrat{f}{g} X'/_{q'} Y'$ means that $f: X \dashrightarrow X'$ and $g : Y \dashrightarrow Y'$ are dominant rational maps such that $q' \circ f = g \circ q $.
We shall also fix $H_{X_j}$ and $H_{Y_j}$ big and nef Cartier divisors on $X_j$ and $Y_j$ respectively.

 In this section we prove  Theorem \ref{thm_int_A} and Theorem \ref{thm_int_B}. 
 They will follow from Theorem \ref{thm_sub_multipicativity} and Theorem \ref{thm_eq_norm_deg} respectively.

\subsection{Degrees of rational maps}

\begin{defi} \label{defi_deg_rel} Let us consider a rational map 
$X_1/_{q_1} Y_1 \relrat{f}{g} X_2/_{q_2} Y_2$ and let $\Gamma_f$ (resp. $\Gamma_g$) be the  normalization of the graph of $f$ (resp. $g$) in $X_1 \times X_2$ (resp. $Y_1 \times Y_2$).
We denote by $\tilde{\Gamma_f}$ the normalization of the graph of the map induced by $ q \circ f $ from $\Gamma_f$ to $\Gamma_g$, we thus have the following diagram. 
\[\xymatrix{ & \tilde{\Gamma_f} \ar[ld]_{\pi_1} \ar[rd]^{\pi_2} \ar@/^1pc/[ddd]^{\varpi}&\\
X_1 \ar[d]^{q_1} \ar@{-->}[rr]^{f}&&X_2 \ar[d]^{q_2} \\
Y_1 \ar@{-->}[rr]^{g}& & Y_2\\
 & \Gamma_g \ar[ul]^{\pi_1'} \ar[ur]_{\pi_2'} & }\]
The  $i$-th relative degree of $f$ is defined by the formula:
\begin{equation*}
\reldeg_{i}(f) : = ( \pi_1^*(H_{X_1}^{e-i} \cdot (q_1^*H_{Y_1})^{l}) \cdot \pi_2^*(H_{X_2})^{i}).
\end{equation*}
When $Y_1 $ and $ Y_2 $ are reduced to a point, we simply write $\deg_i(f) = \reldeg_{i}(f)$.
\end{defi}
\begin{rem} If $e = 0$, then $\reldeg_{i}(f) = (q_{1}^*H_{Y_1}^l)$ if $i = 0$ and $\reldeg_{i}(f) = 0$ for $i >0$.
\end{rem}

\begin{rem} Observe that in the above diagram, the $\varpi : \tilde{\Gamma_f} \to \Gamma_g$ is a regular surjective morphism.
\end{rem}

Note that the degrees always depend on the choice of the big nef divisors, but to simplify the notations, we deliberately omit it. 

\medskip
We now explain how to associate to any rational map $X_1 /_{q_1} Y_1 \relrat{f}{g} X_2/_{q_2} Y_2 $ a pullback operator $(f,g)^{\bullet,i}$.
\begin{defi}
 Let $X_1/_{q_1} Y_1 \relrat{f}{g} X_2/_{q_2} Y_2 $ be a rational map and let $\pi_1$ and $\pi_2$ be the projections from the graph of $f$ in $X_1 \times X_2$ onto the first and the second factor respectively. 
  We define the linear morphisms $(f,g)^{\bullet,i}$ and $(f,g)_{\bullet,i}$ by the following formula:
\begin{equation*}
(f,g)^{\bullet,i} : \alpha \in \numd^{i}(X_2/Y_2)_\mathbb{R} \longrightarrow  ({\pi_1}_* \circ  \psi_{\tilde{\Gamma_f}/\Gamma_g} \circ \pi_2^*  )(\alpha) \in \num_{e-i}(X_1/Y_1)_\mathbb{R}.  
\end{equation*}

\begin{equation*}
(f,g)_{\bullet,i} : \beta \in \numd^{i}(X_1/Y_1)_\mathbb{R} \longrightarrow  ({\pi_2}_* \circ  \psi_{\tilde{\Gamma_f}/\Gamma_g} \circ \pi_1^*  )(\beta) \in \num_{e-i}(X_2/Y_2)_\mathbb{R}. 
\end{equation*}
\end{defi}
\begin{rem}
When $Y_1 $ and $ Y_2 $ are reduced to a point, then we simply write $f^{\bullet,i} (\alpha) := (f, \Id_{\{pt\}})^{\bullet,i} (\alpha)$ and $f_{\bullet,i} (\beta) := (f, \Id_{\{pt\}})_{\bullet,i} (\beta)$.
\end{rem}

\begin{rem}
Since $\numd^i(X/Y) = 0$ and $ \num_{e-i}(X) = 0$ when $i >e$, it implies that
 $(f,g)^{\bullet,i}$ and $(f,g)_{\bullet,i}$ are identically zero for any $i > e$.
\end{rem}

\subsection{Sub-multiplicativity}

\begin{thm} \label{thm_sub_multipicativity} Let us consider the composition $ X_1/_{q_1}Y_1 \relrat{f_1}{g_2} X_2/_{q_2} Y_2 \relrat{f_2}{g_2} X_3/_{q_3}Y_3$ of dominant rational maps.
Then for any integer $0 \leqslant i \leqslant e$, there exists a constant $C>0$ which depends only on the choice of $H_{X_2} $, $H_{Y_2}$, $i$, $l$ and $e$ such that:
\begin{equation*}
\reldeg_{i}(f_2 \circ f_1 ) \leqslant C \reldeg_{i}(f_1) \reldeg_{i} (f_2).
\end{equation*}  
More precisely, $C= (e-i+1)^i/ (H_{X_2}^e \cdot q_2^* H_{Y_2}^l)$.
\end{thm}

\begin{proof} We denote by $\tilde{\Gamma_{f_1}}$ (resp. $\tilde{\Gamma_{f_2}}$, $\Gamma_{g_1}, \Gamma_{g_2}$) the normalization of the graph of $q_2 \circ f_1$ (resp. $q_3\circ f_2,g_1,g_2$) and $\pi_1$, $\pi_2$ (resp.$\pi_3,\pi_4$, $\pi_1'$, $\pi_2'$ and $\pi_3'$, $\pi_4'$) the projections onto the first and the second factor respectively. We set $\Gamma$ as the graph of the rational map $\pi_3^{-1} \circ f_1 \circ  \pi_1 : \tilde{\Gamma_{f_1}} \dashrightarrow \tilde{\Gamma_{f_2}}$, $u$ and $v$ the projections from $\Gamma$ onto $\tilde{\Gamma_{f_1}}$ and $\tilde{\Gamma_{f_2}}$ and $\varpi_i$ the restriction on $\tilde{\Gamma_{f_i}}$ of the projection from $X_i \times X_{i+1}$ to $Y_i \times Y_{i+1}$ for each $i=1,2$.  We have thus the following diagram.

\begin{equation} \label{big_diagram_dyn}
\xymatrix{  && \Gamma \ar[ld]_u \ar[rd]^v  & & \\
 & \tilde{\Gamma_{f_1}} \ar@/^1pc/[ddd]^{\varpi_1} \ar[ld]_{\pi_1} \ar[rd]^{\pi_2}  & & \tilde{\Gamma_{f_2}}  \ar[ld]_{\pi_3} \ar[rd]^{\pi_4} \ar@/^1pc/[ddd]^{\varpi_2}& \\
 X_1 \ar[d]_{q_1} \ar@{-->}[rr]_{f_1} & &  X_2\ar[d]_{q_2} \ar@{-->}[rr]_{f_2} & & X_3 \ar[d]_{q_3}\\
 Y_1\ar@{-->}[rr]^{g_1} & & Y_2\ar@{-->}[rr]^{g_2} & & Y_3 \\
 &\Gamma_{g_1} \ar[ul]^{\pi'_1} \ar[ru]_{\pi'_2}& & \Gamma_{g_2}  \ar[ul]^{\pi'_3} \ar[ru]_{\pi'_4}& \\
}
\end{equation}
By Proposition \ref{prop_class_restriction} applied to $ q_2 \circ \pi_2 \circ u : \Gamma \to Y_2$, the class $\psi_\Gamma( u^* \pi_2^* q_2^* H_{Y_2}^l)$ is represented by the fundamental class $[V]$ where $V$ is a subscheme of dimension $e$ in $\Gamma$ which is a general fiber of $q_2 \circ \pi_2 \circ u$. 
We apply Theorem \ref{thm_siu_gen} by restriction to $V$ to the class $a = v^* \pi_4^* H_{X_3}^i  \actd  [V] $ and $b=  u^* \pi_2^* H_{X_2} \actd [V] $. We obtain: 

\begin{equation} \label{eq_submul_1}
  v^* \pi_4^* {H_{X_3}^i} \actd [V] \leqslant  (e-i+1)^i \dfrac{( {v^* \pi_4^* H_{X_3}^i} \cdot {u^* \pi_2^* H_{X_2}^{e-i}} \actd [V]  )}{( {u^* \pi_2^* H_{X_2}^{e}} \actd [V] )}  \  {u^* \pi_2^* H_{X_2}^{i}} \actd [V] \in \num_{e-i}(\Gamma) .
 \end{equation}
 Let us simplify the right hand side of inequality \eqref{eq_submul_1}.
Since $ \pi_2 \circ u = \pi_3 \circ v$, $\psi_\Gamma( u^* \pi_2^* q_2^* H_{Y_2}^l) = [V] \in \num_{e}(\Gamma)$ and since the morphism $v$ is generically finite, one has that:
\begin{equation} \label{eq_num_1}
( {v^* \pi_4^* H_{X_3}^i} \cdot {u^* \pi_2^* H_{X_2}^{e-i}} \actd [V]  ) = (v^* (\pi_4^* H_{X_3}^i \cdot  \pi_3^* H_{X_2}^{e-i} \cdot \pi_3^* q_2^* H_{Y_2}^l)  )= d \times \reldeg_{i}(f_2),
\end{equation} 
where $d$ is the topological degree of $v$. 
The same argument gives:
\begin{equation} \label{eq_den_1}
(u^* \pi_2^* H_{X_2}^e \actd [V]) = d \times (H_{X_2}^e \cdot q_2^* H_{Y_2}^l).
\end{equation}
Using \eqref{eq_num_1}, \eqref{eq_den_1}, inequality \eqref{eq_submul_1} can be rewritten as:
 \begin{equation*}
u^* \pi_2^* q_2^*H_{Y_2}^{l} \cdot v^* \pi_4^* H_{X_3}^i \leqslant C \reldeg_{i}( f_2) \  u^* \pi_2^* H_{X_2}^{i} \cdot u^* \pi_2^* q_2^*H_{Y_2}^{l}  \in \numd^{l+i}(\Gamma),
\end{equation*}
where $C = (e-i+1)^i   /  (H_{X_2}^e \cdot q_2^* H_{Y_2}^l)$. Since the class $u^*\pi_1^* H_{X_1}^{e-i} \in \numd^{e-i} (\Gamma)$ is nef, we can intersect this class in the previous inequality to obtain:
\begin{equation} \label{eq_submul_2}
(u^*(\pi_1^* H_{X_1}^{n-l-i} \cdot \pi_2^* q_2^*H_{Y_2}^{l}) \cdot v^* \pi_4^* H_{X_3}^{i}) \leqslant C' \reldeg_{i}(f_2)   (u^* \pi_2^* H_{X_2}^{i} \cdot u^* \pi_2^* q_2^*H_{Y_2}^{l}  \cdot u^*\pi_1^* H_{X_1}^{n-l-i}).
\end{equation}
Let us simplify the expressions in inequality \eqref{eq_submul_2}. 
Because $\pi_2^* q_2^* H_{Y_2}^l = \varpi_1^* \pi_2'^* H_{Y_2}^l$ and  
$\deg_l(g_1) = (\pi_2'^* H_{Y_2}^l)$, we deduce that: 
\begin{equation} \label{eq_simp_1}
 \pi_2^* q_2^* H_{Y_2}^l = \dfrac{\deg_l(g_1)}{(H_{Y_1}^l)}  \varpi_1^* {\pi'_1}^* H_{Y_1}^l =\dfrac{\deg_l(g_1)}{(H_{Y_1}^l)} \pi_1^* q_1^* H_{Y_1}^l .
\end{equation}
Applying \eqref{eq_simp_1}, the inequality \eqref{eq_submul_2} can be translated as:
\begin{equation*}
\dfrac{\deg_l(g_1)}{(H_{Y_1}^l)}  (u^*\pi_1^*( H_{X_1}^{n-l-i} \cdot  q_1^*H_{Y_1}^{l}) \cdot v^* \pi_4^* H_{X_3}^{i}) \leqslant C \dfrac{\deg_l(g_1)}{(H_{Y_1}^l)}  \reldeg_{i}( f_2)  ( u^*( \pi_2^* H_{X_2}^{i} \cdot  \pi_1^* q_1^*H_{Y_1}^{l}  \cdot \pi_1^* H_{X_1}^{n-l-i})) . 
\end{equation*}
We obtain thus:
\[\dfrac{\deg_l(g_1)}{(H_{Y_1}^l)}  \reldeg_{i}(f_2 \circ f_1) \leqslant C \dfrac{\deg_l(g_1)}{(H_{Y_1}^l)} \reldeg_{i}(f_1) \reldeg_{i}(f_2) . \]
This concludes the proof of the inequality after dividing by $\deg_l(g_1)/ (H_{Y_1}^l)$.
\end{proof}

\black

\black

\subsection{Norms of operators associated to rational maps}

The proof of Theorem \ref{thm_int_B} relies on an easy but crucial lemma which is as follows.
\begin{lem} \label{lem_decomp} Let us consider $(V,|| \cdot ||)$ a finite dimensional normed $\mathbb{R}$-vector space and let $\mathcal{C}$ be a closed convex cone with non-empty interior in $V$. Then there exists a constant $C> 0$ such that any vector $u \in V$ can be decomposed as $v = v^+ - v^-$ where $u^+$ and $u^-$ are in $\mathcal{C}$ such that:
\begin{equation*}
|| v^{+/-}|| \leqslant C ||v||.
\end{equation*}
\end{lem} 

\begin{proof} 
 Let us define the map $f : V \to \mathbb{R}^+$ given by:
\begin{equation*}
f (v) = \inf \{  ||v'|| +|| v' - v|| \ | \ v' \in \mathcal{C}\  , \ v' - v \in \mathcal{C} \}.
\end{equation*}
We check easily that $f$ defines a norm on $V$ which is similar to the proof of Proposition \ref{prop_norm_pliant}. 
Since $V$ is finite dimensional, there exists a constant $C$ such that for any $v \in V$, one has:
\begin{equation*}
f(v) \leqslant C ||v||,
\end{equation*}  
Hence $||v^+|| \leqslant C ||v ||$ and $||v^-|| \leqslant C || v ||$.
\end{proof}

\begin{thm} \label{thm_eq_norm_deg}
 Let $ X/_q Y \relrat{f}{g} X/_qY$ be a rational map.
We fix an integer $i \leqslant e$, some norms on $\numd^i(X/Y)_\mathbb{R}$, on $\num_{e-i}(X/Y)_\mathbb{R}$.
Then there is a constant $C >0$ such that for any rational map $ X/_q Y \relrat{f}{g} X/_q Y$, we have:
\begin{equation*}
\dfrac{1}{C} \leqslant \dfrac{|| (f, g)^{\bullet,i} ||}{\reldeg_{i}(f)} \leqslant C.
\end{equation*}
 In particular, the $i$-th relative dynamical degree of $f$ satisfies the following equality:
\begin{equation*}
\lambda_i( f,X/Y) = \lim_{p\rightarrow + \infty} || (f^p, g^p)^{\bullet,i} ||^{1/p}.
\end{equation*}
Moreover, when $ Y$ is reduced to a point, we obtain:

\begin{equation*}
\lambda_i(f)=\lim_{p \rightarrow +\infty } ||(f^p)^{\bullet,i} ||^{1/p}.
\end{equation*}
\end{thm}

\begin{rem} The proof of Theorem \ref{thm_int_B} follows directly from Theorem \ref{thm_eq_norm_deg} since $\numd^i(X/Y) = \numd^i(X)$ and $\numd_{e-i}(X/Y) = \num_{e-i}(X)$ when $Y$ is reduced to a point.
\end{rem}

\begin{proof} We denote by $\pi_1$ and $\pi_2$ the projections from the normalization of the graph $\tilde{\Gamma_f}$ of $q \circ f$  onto the first and the second component respectively as in Definition \ref{defi_deg_rel}. 
Since we want to control the norm of $f^{\bullet,i}$ by the $i$-th relative degree of $f$, we first find an appropriate norm to relate the norm on $\num_{e-i}(X)_\mathbb{R}$ with an intersection product.
 As $\num_{e-i}(X/Y)_\mathbb{R}$ is a subspace of $\num_{e-i}(X)_\mathbb{R}$, we can extend the norm $|| \cdot ||_{\num_{e-i}(X/Y)_\mathbb{R}}$ into a norm on $\num_{e-i}(X)_\mathbb{R}$.
As $\num_{e-i}(X)_\mathbb{R}$ is a finite dimensional vector space and since $H_X^{e-i}$ is a class in the interior of the basepoint free cone $\Cc^{e-i}(X)$, we can suppose by equivalence of norms that the norm on $\num_{e-i}(X)_\mathbb{R}$  given by 
\begin{equation*}
|| z|| = \inf_{ \substack{ z = z^+ - z^- \\
z^+, z^- \in \psef_{e-i}(X)
}} \{ (H_X^{e-i} \actd z^+)  + (H_X^{e-i} \actd z^-)  \} 
\end{equation*}
 as in Proposition \ref{prop_norm_pliant}. 
\smallskip 
 
 Let us prove that the lower bound of $|| (f,g)^{\bullet,i}|| / \reldeg_{i}(f)$ is $1$. 
  We denote by $\varphi : \numd^i(X) \to \numd^i(X/Y)$ the canonical surjection. Since $H_X^i$ is basepoint free, it implies that the class $(f,g)^\bullet( \varphi(H_X^i)) \in \num_{e-i}(X/Y)_\mathbb{R} \subset \num_{e-i}(X)_\mathbb{R}$ is pseudo-effective. In particular, this implies that its norm is exactly $\reldeg_{i}(f)$. We have thus by definition: 
\begin{equation*}
\dfrac{ || (f,g )^{\bullet,i} ||}{ \reldeg_{i}(f)} =  \left (
\dfrac{ || (f,g )^{\bullet,i} ||}{ || (f,g)^{\bullet,i}  \varphi(H_X^i)||}\right  )  \geqslant 1 , 
\end{equation*}
as required. 

\medskip
Let us find an upper bound for $|| (f,g )^{\bullet,i} ||/|| (f,g)^{\bullet,i} \varphi(H_X^i)||$. 
First we fix a morphism $s: \numd^i(X/Y)_\mathbb{R} \to \numd^i(X)_\mathbb{R}$ such that $\varphi \circ s = \Id$. 
Take $\alpha \in \numd^i(X/Y)_\mathbb{R}$ of norm $1$, then the class $u = s(\alpha) \in \numd^i(X)_\mathbb{R}$ is a representant of $\alpha$. By construction, the norm of $u$ is bounded by $|| u ||_{\numd^i(X)_\mathbb{R}} \leqslant C_1 || \alpha||_{\numd^i(X/Y)_\mathbb{R}} = C_1$ where $C_1$ is the norm of the operator $s$. Since by Proposition \ref{prop_restriction}.(ii), $\res_{\Gamma_f/\Gamma_g}^* \circ \pi_2^* = (1/\deg_l(g)) \times  \pi_2^* \circ \res_{X/Y}^*$, we have therefore:
\begin{equation*}
(f,g)^{\bullet,i} \alpha = \dfrac{1}{\deg_l(g)} \times  {\pi_1}_* \circ \psi_{\Gamma_f} \circ \pi_2^* \circ  \res_{X/Y}^* (\alpha) = \res_{X/Y} f^{\bullet,i} u. 
\end{equation*} 
By Theorem \ref{thm_classes_pliantes}, the pliant cone $\Cc^i(X)$ has a non-empty interior in $\numd^i(X)_\mathbb{R}$ and we can apply Lemma \ref{lem_decomp}.
There exists a constant $C_2>0$ which depends only on $\Cc^i(X)$ and the choice of the norm on $\numd^i(X)_\mathbb{R}$ such that the class $u$ can be decomposed as $u = u_1 - u_2$ where $u_i \in \Cc^i(X)$ such that $||u_i||_{\numd^i(X)_\mathbb{R}} \leqslant C_2 || u||_{\numd^i(X)_\mathbb{R}}$ for $i=1,2$. We set $\alpha_i = \varphi(u_i)$ for all $i \in \{ 1,2\}$.
By the triangular inequality, we have:
\begin{equation*}
\dfrac{||(f,g)^{\bullet,i} \alpha ||_{\num_{e-i}(X/Y)}}{ || (f,g)^{\bullet,i} \varphi(H_X^i)||}  \leqslant 
\dfrac{|| (f,g)^{\bullet,i} \alpha_1 ||_{\num_{e-i}(X)_\mathbb{R}}}{|| (f,g)^{\bullet,i}\varphi(H_X^i)||} + 
\dfrac{|| (f,g)^{\bullet,i} \alpha_2 ||_{\num_{e-i}(X)_\mathbb{R}}}{|| (f,g)^{\bullet,i}\varphi(H_X^i)||} . 
\end{equation*}
We have to find an upper bound of $|| (f,g)^{\bullet,i} \alpha_i||_{\num_{e-i}(X/Y)_\mathbb{R}}$ for each $i=1,2$.
Applying Siu's inequality (Corollary \ref{cor_siu_gen}) to $a = \pi_2^* u_i$ and $b = \pi_2^* H_X$ and then composing with $\res_{X/Y}\circ {\pi_1}_* \circ \psi_{\Gamma_f}$ gives
\[ \res_{X/Y} (f^{\bullet,i}(u_i)) \leqslant C_3 \dfrac{|| u_i||_{\numd^i(X)_\mathbb{R}}}{(H_X^n)} \times \res_{X/Y}(f^{\bullet,i}(H_X^i)), \]
 where $C_3$ is a positive constant which depends only on the choice of big nef divisors.
This implies by intersecting with $H_X^{e-i}$ the inequality:
\begin{equation*} \label{ineg_secondaire}
 || ( (f,g)^{\bullet,i} (\alpha_i)||_{\num_{e-i}(X/Y)_\mathbb{R}} \leqslant C_3 \dfrac{ || u_i||_{\numd^i(X)_\mathbb{R}}}{(H_X^n)} || (f,g)^{\bullet,i}( \varphi(H_X^i))||_{\num_{e-i}(X/Y)_\mathbb{R}} . 
\end{equation*}
In particular we have shown that:
\begin{equation*}
1 \leqslant \dfrac{ ||(f,g)^{\bullet,i} \alpha ||_{\num_{e-i}(X/Y)_\mathbb{R}}}{|| (f,g)^{\bullet,i} \varphi(H_X^i)||_{\num_{e-i}(X/Y)_\mathbb{R}}} \leqslant \dfrac{ 2 C_1 C_2 C_3}{(H_X^n)},
\end{equation*}
which concludes the proof.

\end{proof}

\section{Semi-conjugation by dominant rational maps} \label{section_semi_conj_rational}

In this section, we consider a more general situation than in the previous section. We still suppose that the varieties $X_i$ and $Y_i$ are of dimension $n$ and $l$ respectively such that the relative dimension is $e= n-l$, but we suppose the maps $q_i : X_i \dashrightarrow Y_i$ merely \emph{rational} and dominant: they may exhibit indeterminacy points. 
Recall also that $H_{X_i}$ and $H_{Y_i}$ are again big and nef Cartier divisors on $X_i$ and $Y_i$ respectively.

\begin{defi} Let $f : X_1 \dashrightarrow X_2$, $g : Y_1 \dashrightarrow Y_2$, $q_1: X_1 \dashrightarrow Y_1$ and $q_2 : X_2 \dashrightarrow Y_2$ be four dominant rational maps such that $ q_2\circ f = g \circ q_1$. 
We define the $i$-th relative dynamical degree of $f$ (still denoted $\reldeg_i(f)$) as the relative degree $\reldeg_i(\tilde{f})$ with respect to the rational map $ \Gamma_{q_1}/ Y_1 \relrat{\tilde{f}}{g} \Gamma_{q_2} / Y_2$ where $\Gamma_{q_i}$ are the  normalization of the graphs of $q_i$ in $X_i \times Y_i$ for each integer $i\in \{ 1,2\}$ respectively and $\tilde{f}: \Gamma_{q_1} \dashrightarrow \Gamma_{q_2}$ is the rational map induced by $f$.  
\end{defi}

\begin{thm} \label{thm_general_relatif}
\begin{enumerate}
\item[(i)] Consider now the following commutative diagram:
\begin{equation*}
\xymatrix{ X_1 \ar@{-->}[r]^{f_1} \ar@{-->}[d]^{q_1} & X_2 \ar@{-->}[r]^{f_2} \ar@{-->}[d]^{q_2} & X_3  \ar@{-->}[d]^{q_3}\\
Y_1 \ar@{-->}[r]^{g_1} & Y_2 \ar@{-->}[r]^{g_2}& Y_3 }
\end{equation*}
 where $f_i : X_i \dashrightarrow X_{i+1}$, $g_i : Y_i \dashrightarrow Y_{i+1}$, $q_1 : X_1 \dashrightarrow Y_1 $, $q_2 : X_2 \dashrightarrow Y_2$, $q_3 : X_3 \dashrightarrow Y_3$ are dominant rational maps for any integer $j\in \{1,2,3\}$ such that $q_{j+1} \circ f_j = g_j \circ q_j$ for any integer $j\in \{1,2 \}$. Then there exists a constant $C >0$ which depends only $e,i$ and the choice of big nef Cartier divisors such that:
\begin{equation*}
\reldeg_i(f_2 \circ f_1) \leqslant C \reldeg_i(f_2) \reldeg_i(f_1).
\end{equation*} 
\item[(ii)] Consider now the following commutative diagram:
\begin{equation*}
\xymatrix{ && X_1' \ar@{-->}[lld]_{\varphi_1} \ar@{-->}[rrr]^{\tilde{f}}\ar@{-->}[dd]& && X_2' \ar@{-->}[lld]_{\varphi_2} \ar@{-->}[dd]\\
X_1 \ar@{-->}[rrr]^{f} \ar@{-->}[dd]^{q_1} &&  & X_2  \ar@{-->}[dd]_<<<<{q_2} &&\\
&& Y_1' \ar@{-->}[rrr]^{\tilde{g}} \ar@{-->}[lld]_{\phi_1}&&& Y_2' \ar@{-->}[lld]^{\phi_2}\\
 Y_1 \ar@{-->}[rrr]^{g}&& & Y_2&&},
\end{equation*}
where $ f:X_1 \dashrightarrow X_2$, $g : Y_1 \dashrightarrow Y_2$, $q_1: X_1 \dashrightarrow Y_1$, $q_2 : X_2 \dashrightarrow Y_2$ are four dominant rational maps such that $ q_2\circ f= g \circ q_1$. 
We consider some birational maps $\varphi_i : X_i' \dashrightarrow X_i$ and $\phi_i : Y_i' \dashrightarrow Y_i$ for $i = 1,2$ such that  $\tilde{f}  = \varphi_2^{-1} \circ f \circ \varphi_1$ and $\tilde{g}  = \phi_2^{-1}\circ g \circ \phi_1 $. 
Then for any integer $0 \leqslant i \leqslant e$, there exists a constant $C> 0$ which depends on $e, i$, on the choice of big nef Cartier divisors and on the rational maps $\varphi_1$ and $\varphi_2$ such that:
\begin{equation}
\dfrac{1}{C}\reldeg_i(f) \leqslant \reldeg_i(\tilde{f}) \leqslant C \reldeg_i(f).
\end{equation} 
\end{enumerate}

\end{thm}

\begin{proof}
(i) 
Recall that the normalization of the graph of $q_j$ in $X_j \times Y_j$ is birational to $X_j$ for $j\in \{ 1,2\}$, hence one can define $\tilde{f}_j : \Gamma_{q_j} \dashrightarrow \Gamma_{q_{j+1}} $  the rational maps induced by $f_j$ on the graph $\Gamma_{q_j}$ of $q_j$ for $j \in \{ 1,2\}$ respectively. Then $(i)$ results directly from Theorem \ref{thm_sub_multipicativity} applied to the composition $\Gamma_{q_1}/Y_1 \relrat{\tilde{f_1}}{g_1} \Gamma_{q_2}/Y_2 \relrat{\tilde{f_2}}{g_2} \Gamma_{q_3}/Y_3$.

\smallskip
 (ii) Let us suppose first that the maps $q_j : X_j \to Y_j$ and $q_j': X_j' \to Y_j'$ are all regular for $j=1,2$. 
Let us apply successively Theorem \ref{thm_sub_multipicativity} to the composition $ X_1'/_{q_1'} Y_1' \relrat{\varphi_1}{\phi_1} X_1/_{q_1} Y_1 \relrat{f}{g} X_2/_{q_2} Y_2 \relrat{\varphi_2^{-1}}{\phi_2^{-1}} X_2' /_{q_2'} Y_2'$.
We obtain : 
\begin{equation}
\reldeg_{i} ( \varphi_2^{-1}\circ f \circ \varphi_1) \leqslant C_2 \reldeg_i(f \circ \varphi_1) \reldeg_i(\varphi_2^{-1}) \leqslant C_1 C_2 \reldeg_i(f) \reldeg_i(\varphi_1) \reldeg_i(\varphi_2^{-1}),
\end{equation}
where $C_1 = (e-i+1)^i / (H_{X_1}^e \cdot q_1^*H_{Y_1}^l)$ and $C_2 = (e-i+1)^i / (H_{X_2}^e \cdot q_2^* H_{Y_2}^l)$.
This proves that:
\begin{equation*}
\reldeg_i(\varphi_2^{-1} \circ f \circ \varphi_1) \leqslant C \reldeg_i(f),
\end{equation*} 
where 
$$C = \dfrac{(e-i+1)^{2i} \reldeg_i(\varphi_1) \reldeg_i(\varphi_2^{-1}) } {(H_{X_1}^e\cdot q_1^* H_{Y_1}^l) (H_{X_2}^e \cdot q_2^* H_{Y_2}^l)}.$$
The proof follows easily from the regular case since the maps $\Gamma_{q_1'} \dashrightarrow \Gamma_{q_1}$ and $\Gamma_{q_2'} \dashrightarrow \Gamma_{q_2}$ are birational where $\Gamma_{q_i'}$ are the graphs of $q_i'$ in $X_i' \times Y_i'$ for $i=1,2$. 
\end{proof}

\textit{Proof of Theorem \ref{thm_int_A}}: 
$(i)$ We apply Theorem \ref{thm_sub_multipicativity} to $Y_1 = Y_2 = Y_3 = \Spec( \C)$, $X_1 = X_2 = X_3 = X$ and $H_{X_1}= H_{X_2} = H_{X_3} = H_X$, we get thus the desired conclusion:
\begin{equation*}
\deg_{i} (g \circ f) \leqslant \dfrac{(n-i+1)^i}{(H_X^n)} \deg_i(f) \deg_i(g).\qedhere
\end{equation*}
(ii) Applying Theorem \ref{thm_general_relatif}.(ii) to the varieties $X_1'= X_2'=X_1 = X_2 = X$,  $Y_1'= Y_2' =Y_1 = Y_2 = \Spec(\C)$, to the choice of big nef divisors $H_{X_1'} = H_{X_2'} = H_X'$, $H_{Y_1'} = H_{Y_2'} = H_Y'$, $H_{X_1} = H_{X_2} = H_X$ and to the rational maps $\varphi_1 = \varphi_2 = \Id_X$, $\phi_1= \phi_2 = g = \Id_{\Spec(\C)}$, $f : X\dashrightarrow X$ yields the desired result. 

\section{Mixed degree formula}

Let us consider three dominant rational maps $f : X \dashrightarrow X$, $q : X \dashrightarrow Y$, $g : Y \dashrightarrow Y$ such that $q \circ f = g \circ q$. Theorem \ref{thm_general_relatif}.(i) implies that for any integer $i\leqslant e$ the sequence $\reldeg_i(f^n)$ is submultiplicative. Define $i$-th relative dynamical degree as follows. 
\begin{equation*}
\lambda_i(f, X/Y) := \lim_{p \rightarrow + \infty} (\reldeg_i(f^p))^{1/p}.
\end{equation*}
When $Y$ is reduced to a point, then we simply write $\lambda_i(f) := \lambda_i( f , X/ \{pt\})$.

\begin{rem} Since $\reldeg_i(f^p) \in \mathbb{N}$ is an integer, one has that $\lambda_i(f, X/Y) \geqslant 1$.
\end{rem}

\begin{rem} Theorem \ref{thm_general_relatif}.(ii) implies that $\lambda_i(f, X/Y)$ is invariant by birational conjugacy, i.e $\lambda_i(f,X/Y)$ does not depend on the choice of big nef Cartier divisors and on any choice of varieties $X'$ and $Y'$ which are birational to $X$ and $Y$ respectively.
\end{rem}
Our aim in this section is to prove Theorem \ref{thm_int_C}.
To that end, we follow the approach from \cite{dinh_nguyen_truong}. The main ingredient (Corollary \ref{cor_ku_pullback}) is an inequality relating basepoint free classes which generalizes to arbitrary fields (see \cite[Proposition 2.3]{dinh_nguyen_2011} and \cite[Proposition 2.5]{dinh_nguyen_truong}).
This inequality is a direct consequence of Theorem \ref{thm_ku_pullback} which estimates the positivity of the diagonal in a quite general setting.
After this, we prove in Theorem \ref{thm_fonda} the submultiplicativity formula for the mixed degrees.
Once the submultiplicativity of these mixed degrees holds, the proof follows from a linear algebra argument.

\subsection{Positivity estimate of the diagonal}

In this section, we prove the following theorem.

\begin{thm} \label{thm_ku_pullback}Let $q: X \to Y$ be a surjective morphism such that $\dim Y = l$ and such that $q$ is of relative dimension $e$. There exists a constant $C>0$ such that for any surjective generically finite morphism $\pi : X' \to X$ and any class $\gamma \in \Cc^{l+e}(X'\times X')$:
 \begin{equation} \label{eq_diag_weak}
 (\gamma \actd [\Delta_{X'}]) \leqslant C \times ( \gamma \cdot (\pi\times \pi)^* (H_X^e \cdot H_Y^l)) ,
 \end{equation}
 where $p_1$ and $p_2$ are the projections from $X\times X$ to the first and the second factor respectively, $H_X = p_1^* H_X + p_2^* H_X$ and $H_Y = p_1^* q^* H_Y + p_2^* q^* H_Y$, and where $\Delta_{X'}$ (resp. $\Delta_X$) is the diagonal of $X'$ (resp. of $X$) in $X'\times X'$ (resp. in $X\times X$).
\end{thm}

\begin{rem} The fact that the constant $C>0$ does not depend on $\pi$ but only on $H_X$, $H_Y$ is crucial in the applications. 
Theorem \ref{thm_ku_pullback} implies that the difference $ (\pi \times \pi)^* (H_X^e \cdot H_Y^l) - [\Delta_X'] $ belongs to the dual cone of the cone $\BPF_{e+l}(X'\times X')_\mathbb{R}$ with respect to the intersection product, however we conjecture that this class should be pseudo-effective:
\begin{equation} \label{eq_diag_strong}
[\Delta_{X'}] \leqslant C \psi_{X'\times X'}((\pi\times \pi)^* (H_X^e \cdot H_Y^l)) \in \num_{l+e}(X'\times X')_\mathbb{R}.
\end{equation}

\end{rem}

We shall use several times the following lemma which is proved at the end of this section.
\begin{lem} \label{lem_technical_siu} 
Let $  X_1/_{q_1} Y_1 \relrat{f}{g} X_2/_{q_2} Y_2 $ be two dominant rational maps where $\dim Y_1 = \dim Y_2 =l$ and $\dim X_1 = \dim X_2 = e+l$ and where $q_1,q_2$ are regular surjective morphisms.
We denote by $\Gamma_f$ and $\Gamma_g$ the normalizations of the graph of $f$ and $g$ in $X_1\times X_2$ and $Y_1\times Y_2$ respectively, $\pi_1,\pi_2, \pi_1', \pi_2'$ are the projections from $\Gamma_f$ and $\Gamma_g$ on the first and the second factor respectively.
Then there exists a constant $C>0$ such that for any surjective generically finite morphism $\pi: X' \to  \Gamma_f$, any integer $0 \leqslant j \leqslant l$ and any class $\beta \in \Cc^{e+l-j}(X')$, one has:
\begin{equation*}
( \beta \cdot \pi ^*\pi_2^* q_2^* H_{Y_2}^j)  \leqslant C \dfrac{\deg_j(g)}{(H_{Y_1}^l)}  \times (\beta \cdot  \pi^* \pi_1^* q_1^* H_{Y_1}^j),
\end{equation*}
where $\deg_j(g)$ is the $j$-th degree of the rational map $g$ with respect to the divisors $H_{Y_1}$ and $H_{Y_2}$.

\end{lem}

\textit{Proof of Theorem \ref{thm_ku_pullback}}.
By Siu's inequality, we can suppose that both the classes $H_X$ and $H_Y$ are ample in $X$ and $Y$ respectively.
We proceed in three steps. Fix $\pi : X' \to X$.

\medskip
\textbf{Step 1}: We suppose first that $X = \Pg^l \times \Pg^e$, $Y = \Pg^l$ and $q$ is the projection onto the first factor.
Since $X\times X$ is smooth, the pullback $(\pi\times \pi)^*$ is well-defined in $\num_{l+e}(X\times X)_\mathbb{R}$ because the morphism $\psi_{X\times X} : \numd^{l+e}(X\times X)_\mathbb{R} \to \num_{l+e}(X\times X)_\mathbb{R}$ is an isomorphism. 
Our objective is to prove that there exists a constant $C_1>0$ such that 
\begin{equation*}
[\Delta_{X'}] \leqslant C_1  \psi_{X'\times X'}((\pi\times \pi)^* (H_X^e \cdot H_Y^l)) \in \num_{l+e}(X'\times X')_\mathbb{R}.
\end{equation*}
 As $X\times X$ is homogeneous, we apply the following lemma analogous to \cite[Lemma 4.4]{tuyen2} which we prove at the end of the section.
 \begin{lem} \label{lem_homogen} Let $X$ be a homogeneous projective variety of dimension $n$ and let $\pi : X' \to X$ be a surjective generically finite morphism. Then one has that :
  \begin{equation*}
[\Delta_{X'}] \leqslant (\pi\times \pi)^* [\Delta_X] \in \num_{n}(X'\times X')_\mathbb{R}.
\end{equation*}
 \end{lem}

We denote by $p_1', p_2'$ (resp. $p_1'', p_2''$) the projections from $Y\times Y$ (resp. from $X\times X$) onto the first and the second factor respectively. 
Since the basepoint free cone has a non-empty interior by Theorem \ref{thm_classes_pliantes}.(i) and since the class $p_1'^*H_Y+ p_2'^* H_Y$ is ample on $Y \times Y$, there exists a constant $C_2>0$ such that the class $ - [\Delta_Y] + C_2 (p_1'^* H_Y + p_2'^* H_Y)^{l} \in \numd^{l}(Y\times Y)_\mathbb{R} $ is basepoint free.
Since $\Delta_X = \Delta_Y \times \Delta_{\Pg^e}$ and by intersection and by pullback, we have that the class:
\begin{equation*}
-[\Delta_{X}] + C_2   H_Y^l \cdot p^*[\Delta_{\Pg^e}] \in \numd^{e+l}(X\times X)
\end{equation*}
is basepoint free where $p$ denotes the projection from $X\times X$ to $\Pg^e\times \Pg^e$.
By the same argument, there exists a constant $C_3>0$ such that the class $-p^*[\Delta_{\Pg^e}] + C_3 H_X^e \in \numd^e(X\times X)_\mathbb{R}$ is basepoint free. 
We have proved that the class:
\begin{equation*}
- [\Delta_X] + C_2 C_3 H_Y^l \cdot H_X^e \in \numd^{e+l}(X\times X)_\mathbb{R}
\end{equation*}
is basepoint free.
Since the basepoint free cone is stable by pullback, we have thus:
\begin{equation*}
[\Delta_{X'}] \leqslant (\pi\times \pi)^* [\Delta_X] \leqslant C_1   \psi_{X'\times X'}( (\pi\times \pi)^* (H_Y^l \cdot H_X^e)) \in \num_{l+e}(X'\times X')_\mathbb{R},
\end{equation*}
where $C_1 = C_2 \times C_3$ as required.
\medskip

\textbf{Step 2}: We now suppose that $X = Y \times \Pg^e$.
Since $Y$ is projective, there exists a dominant rational map $\phi : Y \dashrightarrow \Pg^l$ ($\phi$ can be chosen as the composition of an embedding in $\Pg^N$ with a linear projection on a linear hypersurface). 
Let $Y'$ be the normalization of the graph of $\phi$ in $X \times \Pg^e \times \Pg^l$ and we denote by $\phi_1$ and $\varphi_1$ the projections from $Y'$ onto the first and the second factor respectively. 
Let $\varphi_2: Y' \times \Pg^e \to \Pg^l \times \Pg^e$ (resp. $\phi_2 : Y'\times \Pg^e \to X$) the map induced by $\varphi_1$ (resp. $\phi_1$).
Let $X''$ be the fibred product of $X'$ with $Y'\times\Pg^e$ so that $\phi_3$, $\pi'$ are the projections from $X''$ onto $X'$ and $Y'\times \Pg^e$ respectively.
We obtain the following commutative diagram:
\begin{equation*}
\xymatrix{ X' \ar[d]^\pi& \ar[l]^{\phi_3} X'' \ar[d]^{\pi'} \\
Y \times \Pg^e \ar[dd]^q& \ar[l]^{\phi_2} Y' \times \Pg^e \ar[rd]^{\varphi_2} \ar[dd]^{p_{Y'}}& \\
&& \Pg^l \times \Pg^e \ar[dd]^{p_{\Pg^l}}\\
Y & \ar[l]^{\phi_1} Y' \ar[rd]^{\varphi_1}\\
& & \Pg^l }
\end{equation*} 
where $p_{Y'}$ and $p_{\Pg^l}$ are the projections from $Y' \times \Pg^e$ and $\Pg^l \times \Pg^e$ onto $Y'$ and $\Pg^l$ respectively and where the horizontal arrows are birational maps.
Let us prove that there exists a constant $C_4>0$ which does not depend on the morphism $\pi: X' \to X$ such that for any basepoint free class $\gamma' \in \Cc^{e+l}(X''\times X'')$, one has:
\begin{equation*}
(\gamma' \actd [\Delta_{X''}] ) \leqslant C_4 (\gamma' \cdot (\phi_3 \times \phi_3)^* (\pi\times \pi)^* (H_X^e \cdot H_Y^l)).
\end{equation*} 
Fix a class $\gamma' \in \Cc^{e+l}(X''\times X'')$.
We apply the conclusion of the first step to the surjective generically finite morphism $\pi'':=\varphi_2 \circ \pi': X'' \to \Pg^l \times \Pg^e$. 
There exists a constant $C_1>0$ such that 
\begin{equation} \label{eq_step_2_diag}
[\Delta_{X''}] \leqslant C_1  \psi_{X''\times X''}((\pi'' \times \pi'')^* (H_{\Pg^l}^l \cdot H_{\Pg^l \times \Pg^e}^e)) \in \num_{l+e}(X''\times X'')_\mathbb{R},
\end{equation}
where $H_{\Pg^l\times \Pg^e}$ is an ample Cartier divisor in $(\Pg^l \times \Pg^e )^2$ and $H_{\Pg^l}$ is the pullback by $p_{\Pg^l} \times p_{\Pg^l}$ of an ample Cartier divisor in $\Pg^l \times \Pg^l$.
Let us apply Theorem \ref{thm_siu_gen} to the class $(\pi''\times \pi'')^* H_{\Pg^l \times \Pg^e}^e$ and to the class $ (\pi'\times \pi')^* (\phi_2\times \phi_2)^* H_X$, there exists a constant $C_5>0$ such that:
\begin{multline*}
( \pi'' \times \pi'')^* H_{\Pg^l \times \Pg^e}^e \leqslant C_5 \dfrac{ ((\pi'\times \pi')^* ((\phi_2\times \phi_2)^* H_X^{2l+e} \cdot (\varphi_2\times\varphi_2)^* H_{\Pg^l \times \Pg^e}^e ))}{((\pi'\times\pi')^*(\phi_2\times \phi_2)^* H_X^{2(l+e)})}   \\
\times (\pi'\times \pi')^* (\phi_2\times \phi_2)^* H_X^e \in \numd^e(X''\times X'')_\mathbb{R}.
\end{multline*}
Since $((\pi'\times\pi')^* \alpha) = \deg( \pi') (\alpha)$ for any class $\alpha \in \numd^{2l+2e}((Y'\times\Pg^e)^2)_\mathbb{R}$, we have thus:
\begin{equation} \label{eq_ku_siu_1}
(\pi'' \times \pi'') H_{\Pg^l \times \Pg^e}^e \leqslant C_6 (\pi'\times \pi')^* (\phi_2\times\phi_2)^* H_X^e \in \numd^e(X''\times X'')_\mathbb{R},
\end{equation}
where $C_6 = C_5 ((\phi_2\times \phi_2)^* H_X^{2l+e} \cdot (\varphi_2\times\varphi_2)^* H_{\Pg^l\times\Pg^e}^e )/ ((\phi_2\times \phi_2)^* H_X^{2l+2e})>0$ does not depend on $\pi : X' \to X$.
Using \eqref{eq_ku_siu_1} and \eqref{eq_step_2_diag}, we obtain:
\begin{equation}
[\Delta_{X''}] \leqslant C_7 \psi_{X''\times X''}((\pi''\times \pi'')^* H_{\Pg^l}^l \cdot  (\phi_3\times \phi_3)^*(\pi\times \pi)^* H_X^e) \in \num_{l+e}(X''\times X'')_\mathbb{R},
\end{equation}
where $C_7 = C_6\times C_1$.
Since the basepoint free cone is contained in the nef cone by Theorem \ref{thm_classes_pliantes}.(v), we have thus:
\begin{equation} \label{eq_step_2_siu_first}
(\gamma' \actd [\Delta_{X''}]) \leqslant C_7 (\gamma' \cdot  (\phi_3 \times \phi_3)^* (\pi\times \pi)^*H_X^e \cdot (\pi''\times \pi'')^* H_{\Pg^l}^l ).
\end{equation}
Let us denote by $X_1 = (Y \times \Pg^e)^2$, $X_2 = (\Pg^l \times \Pg^e)^2$, $Y_1 = Y \times Y$, $Y_2 = \Pg^l \times \Pg^l$ and let $f := (\varphi_2 \circ \phi_2^{-1} \times \varphi_2\circ  \phi_2^{-1}) : X_1 \dashrightarrow X_2$ and $g:=(\varphi_1 \circ \phi_1^{-1} \times \varphi_1\circ  \phi_1^{-1}) : Y_1 \dashrightarrow Y_2$ be the corresponding dominant rational maps. 
Let us apply Lemma \ref{lem_technical_siu} to the class $(\pi' \times \pi')^* (\varphi_2\times \varphi_2)^* H_{\Pg^l}^l$ and to the class $(\pi'\times \pi')^* (\phi_2 \times \phi_2)^* H_Y^l$, there exists a constant $C_8>0$ which is independent of the morphism $\pi'\times \pi' : X''\times X'' \to (Y'\times \Pg^e)^2$ such that for any class $\beta \in \Cc^{2e+l}(X''\times X'')$:
\begin{equation} \label{eq_apply_siu_technical}
(\beta \cdot  (\pi' \times \pi')^* (\varphi_2\times \varphi_2)^* H_{\Pg^l}^l) \leqslant C_8 \dfrac{\deg_l(g)}{(H_Y^{2l})} (\beta \cdot (\phi_3\times \phi_3)^*(\pi \times \pi)^*  H_Y^l). 
\end{equation} 
Using \eqref{eq_step_2_siu_first} and \eqref{eq_apply_siu_technical} to the class $\beta = \gamma' \cdot (\phi_3\times \phi_3)^* (\pi\times \pi)^*H_X^e \in \Cc^{l+2e}(X''\times X'')$, we obtain:
\begin{equation*}
(\gamma' \actd [\Delta_{X''}]) \leqslant C_4 (\gamma' \cdot (\phi_3\times \phi_3)^*(\pi\times \pi)^* (H_X^e \cdot H_Y^l )),
\end{equation*}
where $C_4= C_4 \times C_8 (\deg_l(g))/ (H_Y^{2l})>0$ does not depend on $\pi$.  
The conclusion of the theorem follows from the projection formula and from the fact that $\phi_3 \times \phi_3$ is a birational map.
Indeed, we apply the previous inequality to $\gamma' = (\phi_3 \times \phi_3)^* \gamma$ where $\gamma \in \Cc^{l+e}(X'\times X')$, we obtain
\begin{equation*}
(\gamma \actd [\Delta_{X'}]) = ( (\phi_3 \times \phi_3)^* \gamma \actd [\Delta_{X''}]) \leqslant C_4 ( \gamma \cdot (\pi \times \pi)^* (H_X^e \cdot H_Y^l))
\end{equation*} 
as required.
\medskip

\textbf{Step 3}: We prove the theorem in the general case.
Suppose $q : X \to Y$ is a surjective morphism of relative dimension $e$ and fix a class $\beta \in \Cc^{l+e}(X'\times X')$.
Since $X$ is projective over $Y$, there exists a closed immersion 
$i:X \to Y\times \Pg^N $ such that $q = p'_{Y} \circ i$ where $p'_Y$ is the projection of $Y \times \Pg^N$ onto $Y$.
 Let us choose a projection $Y \times \Pg^N \dashrightarrow Y \times \Pg^e $ so that the composition with $i$ gives a dominant rational map $f : X \dashrightarrow Y \times \Pg^e$. 
Let us denote by $\Gamma_f$ the normalization of the graph of $f$ in $X \times Y \times \Pg^e$ and $\pi_1, \pi_2$ the projections of $\Gamma_f$ onto the first and the second factor respectively.
We set $X''$ the fibred product of $X'$ with $\Gamma_f$ and we denote by $\pi' $ and $\phi$ the projection of $X''$ to $\Gamma_f$ and $X'$ respectively. 
We get the following commutative diagram:
\begin{equation*}
\xymatrix{ 
X' \ar[d]^{\pi}& \ar[l]^{\phi } \ar[d]^{\pi'} X'' \\
X \ar[d]^{q} \ar@{-->}[rd]^{f}& \Gamma_f \ar[d]^{\pi_2} \ar[l]_{\pi_1} \\
Y & \ar[l]^{p_Y} Y \times \Pg^e 
}
\end{equation*}
where $p_Y$ is the projection of $Y\times \Pg^e$ onto $Y$.
We apply the result of Step 2 to the class $(\phi \times \phi)^* \beta \in \Cc^{l+e}(X'' \times X'')$ and to the diagonal of $X''$. 
There exists a constant $C_4>0$ which does not depend on $\pi$ such that:
\begin{equation} \label{eq_step_3_apply_step_2}
( (\phi\times \phi)^* \beta \actd [\Delta_{X''}]) \leqslant C_4 ((\phi\times \phi)^*(\beta \cdot (\pi\times \pi)^* H_Y^l) \cdot ((\pi_2 \circ \pi')\times (\pi_2 \circ \pi'))^* H_{Y\times \Pg^e}^e). 
\end{equation} 
Let us apply Theorem \ref{thm_siu_gen} to the class $((\pi_2 \circ \pi')\times (\pi_2 \circ \pi'))^* H_{Y\times \Pg^e}^e$ and to the class $(\phi \times \phi)^* (\pi \times \pi) H_X$.
There exists a constant $C_9>0$ such that:
\begin{multline*}
((\pi_2 \circ \pi')\times (\pi_2 \circ \pi'))^* H_{Y\times \Pg^e}^e \leqslant C_9 \dfrac{( (\pi'\times\pi')^* (  (\pi_2\times \pi_2)^* H_{Y\times \Pg^e}^e \cdot (\pi_1\times\pi_1)^* H_X^{2l+e} ))  }{((\pi'\times \pi')^*(\pi_1\times \pi_1)^* H_X^{2l+2e})} \\
(\phi \times \phi)^* (\pi \times \pi) H_X^e \in \numd^{e}(X''\times X'')_\mathbb{R}.
\end{multline*}
Since $( (\pi'\times\pi')^* (  (\pi_2\times \pi_2)^* H_{Y\times \Pg^e}^e \cdot (\pi_1\times\pi_1)^* H_X^{2l+e} ))  /((\pi'\times \pi')^*(\pi_1\times \pi_1)^* H_X^{2l+2e}) = \deg_{e}(f\times f)/ (H_X^{2l+2e})$and using \eqref{eq_step_3_apply_step_2}, we obtain:
\begin{equation*}
((\phi\times \phi)^* \beta \actd [\Delta_{X''}] )\leqslant C ( (\phi\times \phi)^* ( \beta \cdot (\pi\times \pi)^* (H_X^e \cdot H_Y^l)) ),
\end{equation*}
where $C=C_4 C_9 \deg_e(f\times f) / (H_X^{2l+2e}) $. 
Since the morphism $\pi_1: \Gamma_f \to X$ is birational, the map $\phi: X'' \to X'$ is also birational and we conclude using the projection formula and since $(\phi\times \phi)_* [\Delta_{X''}] = [\Delta_{X'}]$:
\begin{equation*}
(\beta \actd [\Delta_{X'}] ) \leqslant C (\beta \cdot (\pi\times \pi)^* (H_X^e \cdot H_Y^l)).
\end{equation*}
\qed

Recall that $X,Y$ are normal projective varieties and $H_X,H_Y$ are ample divisors on $X$ and $Y$ respectively.
\begin{cor} \label{cor_ku_pullback} Let $q: X \to Y$ be a surjective morphism of relative dimension $e$ where $\dim Y = l$. Then there exists a constant $C>0$ such that for any surjective generically finite morphism $\pi : X' \to X$ such that for any class $\alpha \in \Cc^i(X')$ and any class $\beta \in \Cc^{l+e-i}(X')$, one has:
\begin{equation} \label{eq_ku_pull}
(\beta \cdot \alpha)  \leqslant  C \sum_{\max(0, i-l) \leqslant j \leqslant \min(i,e)} U_j(\pi_* \psi_{X'}(\alpha))\times  (\beta\cdot \pi^* (q^* H_Y^{i-j} \cdot H_X^j)),
\end{equation} 
where $U_j(\pi_* \psi_{X'}(\alpha)) = ( H_X^{e-j} \cdot q^* H_Y^{l-i+j} \actd \pi_* \psi_{X'}(\alpha))$.
\end{cor}

\begin{rem} Note that when $i\leqslant e$, then the inequality is already a consequence of Siu's inequality (Theorem \ref{thm_siu_gen}). 
Indeed, the term on the right hand side of \eqref{eq_ku_pull} with $j = i$ corresponds exactly to the term $C (\pi^*H_X^{n-i} \cdot  \alpha) \times \pi^* H_X^i$. 
\end{rem}

\begin{rem}Equation \eqref{eq_ku_pull} proves that the class
\begin{equation*}
-\psi_{X'}(\alpha) + C  \sum_{\max(0, i-l) \leqslant j \leqslant \min(i,e)} U_j(\pi_* \psi_{X'}(\alpha))  \psi_{X'}(\pi^* (q^* H_Y^{i-j} \cdot H_X^j)) \in \num_{n-i}(X')_\mathbb{R}
\end{equation*}
is in the dual of the basepoint free cone $\Cc^{n-i}(X')$. 
Moreover, if \eqref{eq_diag_strong} is satisfied, then this class is pseudo-effective.
\end{rem}

\begin{proof}
We apply Theorem \ref{thm_ku_pullback} to the class $\gamma = p_1^* \beta \cdot p_2^* \alpha \in \Cc^n(X'\times X')$.
There exists a constant $C_1>0$ such that for any surjective generically finite morphism $\pi: X'\to X$ and any class $\gamma \in \Cc^n(X'\times X')$, one has:
\begin{equation*}
(\gamma \actd [\Delta_{X'}]) \leqslant C_1 (\gamma \cdot (\pi\times\pi)^* (H_X^e \cdot H_Y^l)). 
\end{equation*}
We denote by $p_1$ and $p_2$ the projections of $X'\times X'$ onto the first and the second factors respectively.
Fix $\alpha \in \Cc^i(X')$ and $\beta \in \Cc^{n-i}(X')$. 
Let us apply the previous inequality to $\gamma = p_1^* \beta \cdot p_2^* \alpha \in \Cc^n(X'\times X')$. 
We obtain:
\begin{equation*}
(\beta\cdot \alpha) = (p_1^* \beta \cdot p_2^* \alpha \actd [\Delta_{X'}]) \leqslant C_1( p_1^*\beta \cdot p_2^* \alpha \cdot (\pi\times \pi)^* (H_X^e \cdot H_Y^l)).
\end{equation*}
Since $ ( p_1^*\pi^*( H_X^m \cdot q^*H_Y^j) \cdot p_2^*(\pi^*(q^* H_Y^{l-m} \cdot H_X^{e-j}) \cdot \gamma)) = 0 $ when $m + j \neq i$, we obtain :
\begin{equation*}
(\beta \cdot\alpha ) \leqslant   C \sum_{ \max( 0 ,i-l) \leqslant j \leqslant \min(e,i)} (\pi^*(q^* H_Y^{l-i+j} \cdot H_X^{e-j} ) \cdot \alpha)  (  \pi^*(q^* H_Y^{i-j} \cdot H_X^{j}) \cdot \beta ) .
\end{equation*}
where $C = C_1 \left ( 1 + \max \left ( \left ( \begin{array}{l}
   e \\
j
\end{array}    \right ) \left ( \begin{array}{l}
l \\
i-j
\end{array} \right ) \right ) \right )$.
Hence by the projection formula, we have proved the required inequality:
\begin{equation*}
( \beta \cdot\alpha) \leqslant C \sum_{ \max( 0 ,i-l) \leqslant j \leqslant \min(e,i)} U_j(\pi_* \psi_{X'} (\alpha) ) \times ( \beta \cdot \pi^* (q^* H_Y^{i-j} \cdot H_X^j))) .
\end{equation*}
\end{proof}

\textit{Proof of Lemma \ref{lem_homogen}}: (see \cite[Lemma 4.4]{tuyen2})
Since $X$ is homogeneous, it is smooth.
Let $G$ be the automorphism group of $X \times X$, we denote by $\cdot $ the (transitive) action of $G$ on $X\times X$. 
By generic flatness (see \cite[Theorem 5.12]{FGA}), there exists a non empty open subset $V \subset X \times X$ such that the restriction of $\pi \times \pi$ to $U := (\pi \times \pi)^{-1}(V)$ is flat over $V$.
Recall that two subvarieties $F \subset X\times X$ and $W \subset X\times X$ intersect properly in $X \times X$ if $\dim (F \cap W) = \dim F + \dim W - 2n$. 
Since $G$ acts transitively on $X \times X$, there exists by \cite[Lemma B.9.2]{fulton} a Zariski dense open subset $O \subset G$ such that for any point $g \in O$,
 the cycle $g \cdot [\Delta_X] $ intersects properly every component of $X\times X \setminus V$. 
 In particular, there exists a one parameter subgroup $ \tau : \mathbb{G}_m \to G$ such that $\tau(1) = \Id \in G $ and such that $\tau$ maps the generic point of $\mathbb{G}_m$ to a point in $O$.
 Let $S$ be the closure in $X'\times X'\times \Pg^1$ of the set $ \{ (x',t) \in U \times \mathbb{G}_m \ | \ (\pi\times \pi)(x') \in \tau(t) \cdot \Delta_X \} $.
Let $p: X'\times X' \times \Pg^1 \to X'\times X'$ be the projection onto $X'\times X'$ and let $f : S \to \Pg^1$ be the morphism induced by the projection of  $X'\times X' \times \Pg^1$ onto $\Pg^1$.
 As in \cite[Section 1.6]{fulton}, we denote by $S_t:= p_*[f^{-1}(t)] \in Z_{n}(X'\times X')$ for any $t \in \mathbb{G}_m$. 
 By construction the cycle $S_1 \in Z_{n}(X'\times X')$ is effective and its support contains the diagonal $\Delta_{X'}$ in $X'\times X'$, hence:
 \begin{equation*}
 [\Delta_{X'}] \leqslant S_1 \in \num_{n}(X'\times X')_\mathbb{R}.
 \end{equation*}
 Let $t \in \mathbb{G}_m$ such that $\tau(t) \in O $.
 Since $S_1 = S_t \in A_{n}(X'\times X')$ for any $t \in \Pg^1$, we have thus:
 \begin{equation*}
 [\Delta_{X'}] \leqslant S_t \in \num_{n}(X'\times X')_\mathbb{R}.
 \end{equation*}
Since the cycle $\tau(t) \cdot [\Delta_X]$ intersects properly every component of $X\times X\setminus V$ and since the restriction of $\pi \times \pi$ to $U = (\pi\times \pi)^{-1}(V)$ is flat over $V$, \cite[Example 11.4.8.(b)]{fulton} asserts that the pullback of $(\pi\times \pi) ^* \tau(t) \cdot [\Delta_X]$ is rationnally equivalent to the cycle $ [ \overline{(\pi\times \pi)_{|U}^{-1} (\tau(t) \cdot \Delta_X) }]$. We have thus:
  $$S_t = [ \overline{(\pi\times \pi)_{|U}^{-1} (\tau(t) \cdot \Delta_X) }] = (\pi\times \pi)^* [\Delta_X] \in A_n(X'\times X').$$
Hence:
\begin{equation*}
[\Delta_{X'}] \leqslant (\pi\times \pi)^* [\Delta_X] \in \num_{n}(X'\times X')_\mathbb{R}.
\end{equation*}
\qed

\bigskip

\textit{Proof of Lemma \ref{lem_technical_siu}}.
Observe that one has the following commutative diagram:
\begin{equation*}
\xymatrix{
& X' \ar[d]^{\pi}&  \\
& \Gamma_f \ar[ld]_{\pi_1} \ar[rd]^{\pi_2} & 
\\
X_1 \ar@{-->}[rr]^{f} \ar[d]^{q_1}& & X_2\ar[d]^{q_2} \\
Y_1  \ar@{-->}[rr]^{g}& & Y_2\\
 & \Gamma_g \ar[ul]^{\pi_1'} \ar[ur]_{\pi_2'} & }
\end{equation*}
Fix a class $\beta \in \Cc^{e+l-j}(X')$.  
By linearity and by Proposition \ref{prop_representabilite_numd}, we can suppose that the class $\beta$ is induced by a product of nef divisors $D_1 \cdot \ldots  \cdot D_{e_1 + e+ l -j}$ where $D_i$ are nef divisors on $ X_1'$ where $p: X_1' \to X'$ is a flat morphism of relative dimension $e_1$.
The intersection $(\beta \cdot \pi^* \pi_2^* q_2^* H_{Y_2}^j)$ is thus given by the formula:
\begin{equation*}
(\beta \cdot \pi^* \pi_2^* q_2^* H_{Y_2}^j )= (D_1 \cdot \ldots \cdot D_{e_1 + e + l-j} \cdot p^*  \pi^* \pi_2^* q_2^* H_{Y_2}^j).
\end{equation*}
Take $A$ an ample Cartier divisor on $X_1'$ and set $\alpha_{\epsilon } = (D_1 + \epsilon A) \cdot \ldots (D_{e_1 + e }+ \epsilon A) \in \numd^{e_1+e}(X_1')_\mathbb{R}$ for any $\epsilon>0$. 
Since the class $\alpha_\epsilon$ is a complete intersection and since the morphisms $q_i$ are surjective, there exists a cycle $V_\epsilon \in Z_{l}(X'_1)_\mathbb{R}$ such that $\psi_{X'_1 }(\alpha_\epsilon) = \{ V_{\epsilon} \} \in \num_{l}(X'_1)_\mathbb{R}$ and such that the restrictions of the morphisms $\pi_1 \circ \pi \circ p$ and $\pi_2 \circ \pi\circ p$ to the support of $V_\epsilon$ are surjective and generically finite onto $Y_1$ and $Y_2$ respectively.
 We apply Theorem \ref{thm_siu_gen} to the class $(p^*\pi^* \pi_2^* q_2^* H_{Y_2}^j)_{|V_\epsilon}$ and to ${(p^*\pi^* \pi_1^* q_1^* H_{Y_1})}_{|V_\epsilon}$, we get:
 \begin{equation*}
p^* \pi^* \pi_2^* q_2^* H_{Y_2}^j \cdot \alpha_\epsilon \leqslant C \dfrac{ (p^*\pi^*( \pi_2^* q_2^* H_{Y_2}^j \cdot \pi_1^* H_{Y_1}^{l-j} ) \actd \{V_\epsilon\} )}{(p^* \pi^*\pi_1^* q_1^* H_{Y_1}^l \actd \{V_\epsilon\})} \times p^*\pi^* \pi_1^* q_1^* H_{Y_1}^j \cdot \alpha_\epsilon \in \numd^{j+e_1 + e}(X'_1)_\mathbb{R}.
 \end{equation*}
 By the projection formula applied to the morphism $\pi \circ p$, we have that $$(p^*\pi^*( \pi_2^* q_2^* H_{Y_2}^j \cdot \pi_1^* H_{Y_1}^{l-j} ) \actd \{V_\epsilon\} )/ ( p^*\pi^*\pi_1^* q_1^* H_{Y_1}^l \actd \{V_\epsilon\}) = \deg_j(g) / (H_{Y_1}^l),$$ hence:
 \begin{equation*}
 p^*\pi^* \pi_2^* q_2^* H_{Y_2}^j \cdot \alpha_\epsilon \leqslant C \dfrac{\deg_j(g)}{(H_{Y_1}^l)} p^*\pi^* \pi_1^* q_1^*  H_{Y_1}^j \cdot \alpha_\epsilon \in \numd^{j+e_1 + e}(X'_1)_\mathbb{R}.
 \end{equation*}
 We intersect with the class $(D_{e_1 + e + 1} \cdot \ldots \cdot D_{e_1 + e + l -j}) \in \numd^{l-j}(X_1')_\mathbb{R}$ and take the limit as $\epsilon$ tends to zero. We obtain:
 \begin{equation*}
(\beta \cdot \pi^* \pi_2^* q_2^* H_{Y_2}^j )= (D_1 \cdot \ldots \cdot D_{e_1 + e + l-j} \cdot p^*  \pi^* \pi_2^* q_2^* H_{Y_2}^j)  \leqslant C \dfrac{\deg_j(g)}{(H_{Y_1}^l)}(\beta \cdot \pi^* \pi_1^* q_1^* H_{Y_1}^j ),
 \end{equation*}
 as required.
\qed

\subsection{Submultiplicativity of mixed degrees}

\begin{defi} 
Let $ X_1/_{q_1} Y_1 \relrat{f}{g} X_2/_{q_2}Y_2$ be rational maps where $e = \dim X_i - \dim Y_i$ and $l = \dim Y_i$ for $i=1,2$. We fix some ample divisors $H_{X_i} $ and $H_{Y_i}$ on each variety respectively. We define for any integer $0 \leqslant i \leqslant n$:
$$a_{i,j} (f) :=  \left \lbrace \begin{array}{lll}
((H_{Y_1}^{l-j} \cdot H_{X_1}^{e+j-i}) \actd f^{\bullet,i} (H_{X_2}^i)) & \text{ \normalfont if }   \max(0,i-e) \leqslant j \leqslant l,\\
 0 & \text{ \normalfont otherwise}.
\end{array}  \right. $$
\end{defi}

\begin{rem} For $j = 0$, it is the $i$-th relative degree $a_{i,0} (f) = \reldeg_i(f)$ and
when $j= l $, it corresponds to the $i$-th degree of $f$, $a_{i,l}(f) = \deg_i(f)$.
\end{rem}

\begin{thm} \label{thm_fonda} Let $q_1: X_1 \to Y_1$, $q_2 : X_2 \to Y_2$, $q_3: X_3 \to Y_3$ be three surjective morphisms such that $\dim X_i = e+ l$ and $\dim Y_i = l$ for all $i \in \{ 1,2,3\}$.
Then there exists a constant $C>0$ such that for any rational maps $ X_1/_{q_1} Y_1 \relrat{f_1}{g_1} X_2/_{q_2}Y_2$, $ X_2/_{q_2} Y_2 \relrat{f_2}{g_2} X_3/_{q_3}Y_3$ and for all integers $0 \leqslant j_0 \leqslant l$:
\begin{equation*}
a_{i,j_0}(f_2 \circ f_1) \leqslant C \sum_{\max(0,i-l ) \leqslant j \leqslant \min (e , i)} \deg_{i-j}(g_1) a_{i, i-j}(f_2) a_{j, j + j_0 -i}(f_1).
\end{equation*}
\end{thm}

\begin{proof} 
Since we are in the same situation as Theorem \ref{thm_sub_multipicativity}, we can consider the diagram \eqref{big_diagram_dyn} and we keep the same notations. 
We denote by $n=e+l$ the dimension of $X_i$.

Let us denote by $d$ the topological degree of the map $f_2$.
We apply Corollary \ref{cor_ku_pullback} to the pliant class $ \alpha:=  (1/d) v^* \pi_4^* H_{X_3}^i \in \Cc^i(\Gamma)$, to the class $\beta := u^* \pi_1^* (H_{X_1}^{e-i+j_0} \cdot q_1^* H_{Y_1}^{l-j_0}) \in \Cc^{n-i}(\Gamma)$ and to the morphism $\pi = \varphi \circ\pi_3 \circ  v  $.  
There exists a constant $C_1 >0$ which depends only on the choice of divisors $H_{Y_2 \times \Pg^e}$ and $H_{Y_2}$ such that:
\begin{equation*}
a_{i,j_0}(f_2\circ f_1) \leqslant C_1 \sum_{\max(0,i-l ) \leqslant j \leqslant \min (e , i)} U_j(\pi_* \psi_{\Gamma}(\alpha)) ( \beta \cdot \pi^* (H_{X_2}^j \cdot q_2^* H_{Y_2}^{i-j})),
\end{equation*}
where $U_j( \gamma) = (H_{X_2 }^{e-j} \cdot q_2^*H_{Y_2}^{l-i+j} \actd \gamma)$ for any class $\gamma\in \num_{n-i}(X_2)_\mathbb{R}$.
We observe that $U_j(\pi_* \psi_{\Gamma}(\alpha))= a_{i,i-j}(f_2)$. 
We have thus:
\begin{equation} \label{eq_sub_1}
a_{i,j_0}(f_2\circ f_1) \leqslant C_1 \sum_{\max(0,i-l ) \leqslant j \leqslant \min (e , i)} a_{i,i-j}(f_2) (u^* (\pi_1^* (H_{X_1}^{e-i+j_0} \cdot q_1^* H_{Y_1}^{l-j_0})\cdot  \pi_2^*(H_{X_2}^j \cdot q_2^* H_{Y_2}^{i-j}))).
\end{equation}
Applying Lemma \ref{lem_technical_siu} to the class $ u^* \pi_2^* q_2^* H_{Y_2}^{i-j} \in \Cc^{i-j}(\Gamma)$ and to $\beta' = \beta\cdot u^* \pi_2^* H_{X_2}^j \in \Cc^{n-i+j}(\Gamma)$, there exists a constant $C_2 >0$ such that :
\begin{equation*}
(  \beta' \cdot u^* \pi_2^* q_2^* H_{Y_2}^{i-j}) \leqslant C_2 \deg_{i-j}(g_1) ( u^* (\pi_1^* (H_{X_1}^{e-i+j_0} \cdot q_1^*H_{Y_1}^{l-j_0 + i-j}) \cdot \pi_2^* H_{X_2}^j)).
\end{equation*}
Since the map $u : \Gamma \to \Gamma_{f_1}$ is birational, we have that:
\begin{equation} \label{eq_sub_2}
(u^* (\pi_1^* (H_{X_1}^{e-i+j_0} \cdot q_1^* H_{Y_1}^{l-j_0})\cdot  \pi_2^*(H_{X_2}^j \cdot q_2^* H_{Y_2}^{i-j}))) \leqslant C_2 \deg_{i-j}(g_1) a_{j, j_0 +j -i}(f_1).
\end{equation} 
Finally, \eqref{eq_sub_1} and \eqref{eq_sub_2} imply:
\begin{equation*} 
a_{i,j_0}(f_2\circ f_1) \leqslant C \sum_{\max(0,i-l ) \leqslant j \leqslant \min (e , i)} a_{i,i-j}(f_2) a_{j,j_0 + j-i}(f_1) \deg_{i-j}(g_1),
\end{equation*}
where $C = C_2 C_1>0$ is a constant which is independent of $f_1$ and $f_2$ as required.

\end{proof}

\subsection{Proof of Theorem \ref{thm_int_C}}

Recall that we want to prove the following formula:
\begin{equation*}
\lambda_i(f) = \max_{j\leqslant i} (\lambda_j(f, X/Y) \lambda_{i-j}(g)).
\end{equation*}
By definition of the relative degrees, we are reduced to prove the theorem when $q: X\to Y$ is a proper surjective morphism.
Recall that $\dim X = n$ and $\dim Y = l$ such that $q: X \to Y$ has relative dimension $e= n-l$.
Let us consider the following commutative diagram:
\begin{equation}
\xymatrix{ & \Gamma_f \ar[rd]^{\pi_2} \ar[ld]_{\pi_1} \ar@/^1pc/[ddd]^{\varpi} &  \\
X \ar@{-->}[rr]^{f} \ar[d]^{q}&  & X \ar[d]^{q}\\
Y\ar@{-->}[rr]^{g}& & Y \\
& \Gamma_g \ar[ru]_{\pi_2'} \ar[lu]^{\pi_1'}& }
\end{equation}
where $f: X \dashrightarrow X$, $g : Y \dashrightarrow Y$ are dominant rational maps, $\Gamma_f, \Gamma_g$ are the normalization of the graph of $f$ and $g$ respectively, $\pi_1, \pi_2, \pi_1', \pi_2'$  are the projections from $\Gamma_f$ and $\Gamma_g$ onto the first and second factor respectively and $\varpi: \Gamma_f \to \Gamma_g$ is the restriction of $q \times q$ to $\Gamma_f$.

The following lemma proves that $\max_{j \leqslant i }( \lambda_j(f,X/Y) \lambda_{i-j}(g)) \leqslant \lambda_i(f)$. 
\begin{lem} \label{lem_lower_ineq} For any integer $ \max(0, i-l) \leqslant j \leqslant \min( i , e)$, there exists a constant $C>0$ such that for any rational map $X/_q Y \relrat{f}{g} X/_q Y$,  we have  $\deg_{i-j}(g ) \reldeg_j(f) \leqslant C \deg_i(f)$.
\end{lem}

Granting the above lemma, then we obtain the lower bound on $\lambda_i(f)$ as:
\begin{equation*}
\lambda_i(f) \geqslant \lambda_j(f,X/Y) \lambda_{i-j}(g).
\end{equation*}

\begin{proof}

 It suffices to consider the product $(\pi_1^* (H_X^{e-j} \cdot q^* H_Y^{l-i+j}) \cdot  \pi_2^* (H_X^j \cdot q^* H_Y^{i-j})) $. Since $\pi_i \circ q = \varpi \circ \pi_i'$ for $i \in \{ 1,2\}$, we obtain:
 \begin{equation*}
 (\pi_1^* (H_X^{e-j} \cdot q^* H_Y^{l-i+j}) \cdot  \pi_2^* (H_X^j \cdot q^* H_Y^{i-j})) =( \varpi^* ({\pi_1'}^* H_Y^{l-i+j} \cdot \pi_2'^* H_Y^{i-j}) \cdot \pi_1^* H_X^{e-j} \cdot \pi_2^* H_X^{j} ).
 \end{equation*}
Moreover, one has that $\pi_1'^* H_Y^{l-i+j} \cdot \pi_2'^* H_Y^{i-j} = (\pi_1'^* H_Y^{l-i+j} \cdot \pi_2'^* H_Y^{i-j})\ [p_0] = \deg_{i-j}(g) \ [p_0]$ where $p_0$ is a general point in $\Gamma_g$. 
We can hence apply Proposition \ref{prop_class_restriction} to the morphism $\varpi: \Gamma_f \to \Gamma_g$ and obtain:
\begin{equation*}
(\pi_1^* (H_X^{e-j} \cdot q^* H_Y^{l-i+j}) \cdot  \pi_2^* (H_X^j \cdot q^* H_Y^{i-j})) = \deg_{i-j} (g) ( \pi_1^* H_X^{e-j} \cdot \pi_2^* H_X^{j} \actd [{\Gamma_f}_{p_0}] ).
\end{equation*}
Since $\pi_1'$ is a birational morphism, a general fiber of $\varpi$ is equal to a general fiber of $\pi_1' \circ \varpi$. 
In other words, we have that $\res_{\Gamma_f/\Gamma_g} = \res_{\Gamma_f/Y}$ and since $\pi_1^* H_X^{e-j} \cdot \pi_2^* H_X^{j} \actd [{\Gamma_f}_{p_0}]  = \res_{\Gamma_f / \Gamma_g} (\pi_1^* H_X^{e-j} \cdot \pi_2^* H_X^j)$, we obtain:
\begin{equation*}
(\pi_1^* (H_X^{e-j} \cdot q^* H_Y^{l-i+j}) \cdot  \pi_2^*( H_X^j \cdot q^* H_Y^{i-j})) = \deg_{i-j}(g) \times \reldeg_j(f).
\end{equation*}
As $H_X$ is ample, we apply Theorem \ref{thm_classes_pliantes} to the classes $\pi_2^*q^* H_Y$ and $\pi_2^*H_X$:
\begin{equation*}
\pi_2^* q^*H_Y^{i-j} \leqslant (n-i+j+1)^{i-j} \dfrac{(\pi_2^* q^* H_Y^{i-j} \cdot \pi_2^* H_X^{n-i+j})}{(\pi_2^* H_X^n)} \pi_2^*H_X^{i-j} = C_1 \pi_2^* H_X^{i-j} \in \numd^{i-j}(X)_\mathbb{R}, 
\end{equation*}
where $C_1 = (n-i+j+1)^{i-j}(q^* H_Y^{i-j} \cdot H_X^{n-i+j}) / (H_X^n)$ depends only on $n,i$ and the choice of big nef Cartier divisors. 
Intersecting with $\pi_1^* H_X^{n-i} \cdot \pi_2^* H_X^j$, one obtains:
\begin{equation*}
\deg_{i-j}(g) \cdot \reldeg_j(f) \leqslant C_1 (\pi_2^* H_X^{i} \cdot  \pi_1^* (H_X^{e-j} \cdot q^* H_Y^{l-i+j})).
\end{equation*}
By the same argument, there exists a constant $C_2>0$ which depends only on $H_Y$, $H_X$ and $i$ such that:
\begin{equation*}
\pi_1^* q^* H_Y^{l-i+j} \leqslant C_2 \pi_1^* H_X^{l-i+j}.
\end{equation*}
Hence, we obtain:
\begin{equation*}
\deg_{i-j}(g) \reldeg_j(f) \leqslant C \deg_i(f),
\end{equation*}
where $C = C_1 C_2$. 
\end{proof}

Let us prove the converse inequality.
We fix an integer $0 \leqslant i \leqslant n$. 
Let us apply Theorem \ref{thm_fonda} to $f_1=f$, $f_2=f^p$, $g_1 =  g$ and $g_2=g^p$, we can rewrite the inequality as:
\begin{equation} \label{eq_rec_formula}
a_{i,j_{0}} (f^{p+1}) \leqslant C \sum_{\max(0 , i-e) \leqslant j \leqslant \min(i,l)} \deg_j(g) a_{i-j, j_0 -j}(f) a_{i, j}(f^p).
\end{equation}
Let us denote by $U_i(f)$ the column vector given by:
\begin{equation*}
U_i(f) = ( a_{i,j}(f))_{0 \leqslant j \leqslant l} = \left ( \begin{array}{l}
a_{i,0}(f) \\
\ldots \\
a_{i, l}(f)
\end{array} \right ).
\end{equation*}
Let us also denote by $M_i(f)$ the $(l+1)\times (l+1)$ lower-triangular matrix given by:
\begin{equation*}
M_i(f) :=( \deg_{j}(g) a_{i-j,m-j}(f) \times \chi_{ [i-e, \min(i,l)]}(j)  )_{ 0 \leqslant  m \leqslant l, 0 \leqslant j \leqslant l}   ,
\end{equation*}
where $\chi_{A}$ denotes the characteristic function of the set $A$. 
Therefore, \eqref{eq_rec_formula} can be rewritten as:
\begin{equation*}
U_i(f^{p+1}) \leqslant C M_i(f) \cdot U_i(f^p),
\end{equation*}
where $\cdot$ denotes the linear action on $\mathbb{Z}^{l+1}$.
A simple induction proves:
\begin{equation*}
U_i(f^p) \leqslant C^p (M_i(f))^{p-1}\cdot U_i(f)
\end{equation*}
Since the $(l+1)$-th entry of the vector $U_i(f^p)$ corresponds to $\deg_i(f^p)$, we deduce that:
\begin{equation} \label{eq_asymp_f}
\deg_i(f^p)^{1/p} \leqslant C \left \langle e_{l} ,\left ( M_i(f)\right )^p\cdot  U_i(f)\right  \rangle^{1/p},  
\end{equation}
where $(e_0, \ldots, e_{l})$ denotes the canonical basis of $\mathbb{Z}^{l+1}$.
In particular, $\deg_i(f^p)^{1/p}$ is controlled up to a constant by the eigenvalues of the matrix $M_i(f)$ which are $\deg_j(g) \reldeg_{i-j}(f)$ for $\max(0,i-e)\leqslant j\leqslant \min(i,l)$ since $M_i(f)$ is lower-triangular. 
Applying \eqref{eq_asymp_f} to $f^{r}$, we get:

\begin{equation*}
\deg_i(f^{pr})^{1/(pr)} \leqslant C^{1/r} ||  U_i(f^r)||^{1/pr} \max_{ \max(0,i-e) \leqslant j \leqslant \min(i,l) } (\deg_{j}(g^r) \reldeg_{i-j} (f^r))^{1/r}. 
\end{equation*}
We conclude by taking the $\limsup$ as $r \rightarrow +\infty, p\rightarrow + \infty$:
\begin{equation*}
\lambda_i(f) \leqslant \max_{\max(0, i-l)\leqslant j \leqslant \min(i,e)} \lambda_{i-j}(g) \lambda_{j}(f,X/Y).
\end{equation*}

\begin{rem} Note that the previous theorem gives information only on the dynamical degrees of $f$. Lemma \ref{lem_lower_ineq} provides a lower bound on the degree of $f^p$. However, one cannot find an upper bound for $\deg_i(f^p)$ which would only depend on the relative degrees and the degree on the base. If $X = E \times E$ is a product of two elliptic curves and if $f : (z, w) \in E \times E \rightarrow (z , z + w) $ is an automorphism of $X$, then the degree growth of $f^p$ is equivalent to $p^2$ whereas the degree on the base and on any fiber are trivial.
\end{rem}

\section{K\"ahler case} \label{section_kahler}

We prove the submultiplicativity of the $k$-th degrees in the case where $(X,\omega)$ is a complex compact K\"ahler manifold. For any  closed smooth $(p,q)$-form $\alpha$ on $X$, we denote by $\{\alpha \}$ its class in the Dolbeault cohomology $H^{p,q} (X)_\mathbb{R}$. 
\begin{defi} Let $(X,\omega)$ be a compact K\"ahler manifold. A class $\alpha \in H^{1,1} (X)_\mathbb{R}$ is nef if for any $\epsilon >0$,  the class $\alpha + \epsilon \{  \omega \}$ is represented by a K\"ahler metric.

A class $\alpha$ of degree $(i,i)$ is pseudo-effective if it can be represented by a closed positive current $T$. Moreover, one says that $\alpha$ is big if there exists a constant $\delta>0$ such that $T -\delta \omega $ is a closed positive current and we write $T \geqslant \delta \omega^i$.
\end{defi}

\begin{thm}\label{thm_xiao_ineq} (cf \cite[Remark 3.1]{xiao}, \cite{popovici})Let $(X,\omega)$ be a compact K\"ahler manifold of dimension $n$. Let $k$ be an integer and $  \alpha,\beta$ be two nef classes in $H^{1,1}(X)$ such that $\alpha^i \in H^{i,i}(X)$ is  big and such that $ \int_X {\alpha^n} - \left ( \begin{array}{l}
n \\
i
\end{array} \right ) \int_X \alpha^{n-i} \wedge \beta ^i >0$. Then the class $ \alpha^i - \beta^i $ is big.  
\end{thm}

Recall that the degree of a meromorphic selfmap $f : X \dashrightarrow X$ when $(X,\omega)$ is given by:
\begin{equation*}
\deg_i(f) := \int_{\Gamma_f} \pi_1^* \omega^{n-i} \wedge \pi_2^* \omega^i, 
\end{equation*}
where $\Gamma_f$ is the desingularization of the graph of $f$ and $\pi_j$ are the projections from $\Gamma_f$ onto the first and the second factor respectively. 

\begin{rem}When $X$ is a projective variety and $\omega$ represents the class of a hyperplane section $H_X$, then the intersection of the form coincides with the cup-product in cohomology, hence $\deg_i(f) = \deg_{i,H_X}(f)$.
\end{rem} 
\begin{cor} Let $(X_1,\omega_{X_1})$, $(X_2,\omega_{X_2})$ and $(X_3,\omega_{X_3})$ be some compact K\"ahler manifolds of dimension $n$. Then there exists a constant $C> 0$ which depends only on the choice of the K\"ahler classes $\omega_{X_j}$ such that for any dominant meromorphic maps $f_1: X_1 \dashrightarrow X_2$ and $f_2 : X_2 \dashrightarrow X_3$, one has:
\[\deg_{i}(f_2 \circ f_1) \leqslant C \deg_{i }(f_1) \deg_{i}(f_2).\]
Moreover, the constant $C$ may be chosen to be equal to $ \left ( \begin{array}{l}
n\\
i
\end{array} \right ) / (\int_{X_2} \omega_{X_2}^n)$.
\end{cor}

\begin{proof} The previous theorem gives that for any big nef class $\beta^i \in H^{i,i}(X)$, for any nef class ${\alpha} \in H^{1,1}(X)$, one has:
\begin{equation} \label{eq_Xiao}
\alpha^i \leqslant \left ( \begin{array}{l}
n \\
i
\end{array} \right ) \dfrac{ \int_X \alpha^i \wedge \beta^{n-i}}{ \int_X \beta^n } \times \beta^i.
\end{equation} 
Then, the proof is formally the same as Theorem \ref{thm_sub_multipicativity}. Indeed, one only needs to consider the diagram \eqref{big_diagram_dyn} where $Y_1 = Y_2 = Y_3$ are reduced to a point and where $\Gamma_{f_1}, \Gamma_{f_2}, \Gamma$ are the desingularizations of the graph of $f_1, f_2$ and $\pi_3^{-1} \circ  f_1 \circ \pi_1$ respectively. 
We apply \eqref{eq_Xiao} to $\alpha = v^* \pi_4^* \omega_{X_3}$ and $\beta = v^* \pi_3^* \omega_{X_2}$ to obtain:
\begin{equation*}
v^* \pi_4^* \omega_{X_3}^i \leqslant \left ( \begin{array}{l}
n \\
i
\end{array} \right ) \dfrac{\deg_i(f_2)}{\int_{X_2}\omega_{X_2}^n} \times v^* \pi_3^* \omega_{X_2}^i.
\end{equation*}
By intersecting the previous inequality with the class $u^* \pi_1^* \omega_{X_1}^{n-i}$, we finally get:
\begin{equation*}
\deg_i(f_2 \circ f_1) \leqslant \left ( \begin{array}{ll
}
n \\
i
\end{array} \right ) \dfrac{\deg_i(f_2)  \deg_i(f_1)}{\int_{X_2}\omega_{X_2}^n}.
\end{equation*}
\end{proof}

\section{Comparison with Fulton's approach}
\label{appendix}

In \cite[Chapter 19]{fulton}, a cycle $z \in Z_i(X)$ on a variety $X$ is defined to be numerically trivial if $(c\actd z)$ for any product $ c =c_{i_1}(E_1) \cdot \ldots \cdot c_{i_p}(E_p) \in A^i(X)$ of Chern classes $c_{i_j}(E_j)$ where $E_j$ is a vector bundle on $X$ and $i_1 + \ldots + i_p = i$. This appendix is devoted to the proof of the following result:

\begin{thm_app} \label{thm_pull_num_class} Let $X$ be a normal projective variety of dimension $n$. For any $z \in Z_i(X)$, the following conditions are equivalent:
\begin{enumerate}
\item[(i)] For any product of Chern classes $c = c_{i_1}(E_1) \cdot \ldots \cdot c_{i_p}(E_p) \in A^i(X)_\mathbb{R}$ where $E_j$ are vector bundles on $X$ and $i_1 + \ldots + i_p = i$, we have $(c \actd z) = 0$.  
\item[(ii)] For any integer $e$, any flat morphism $p_1 : X_1 \to X$ of relative dimension $e$ where $X_1$ is a projective scheme and any Cartier divisors $D_1 , \ldots , D_{e+i}$ on $X_1$, we have $(D_1 \cdot \ldots\cdot D_{e+i} \actd p_1^*z) =0 $. 
\item[(iii)]For any integer $e$, any flat morphism $p_1 : X_1 \to X$ of relative dimension $e$ between normal projective varieties and any Cartier divisors $D_1 , \ldots , D_{e+i}$ on $X_1$, we have $(D_1 \cdot \ldots\cdot D_{e+i} \actd p_1^*z) =0 $. 
\end{enumerate}
\end{thm_app}

The implication $(ii) \Rightarrow (i)$ follows immediately from the definition of Chern classes. 
The implication $(ii) \Rightarrow (iii)$ is also straightforward.
For the converse implications $(i) \Rightarrow (ii)$ and $(i) \Rightarrow (iii)$, we rely on the following proposition. 

\begin{prop_app} Let $q : X \to Y$ be a flat morphism of relative dimension $e$ where $X$ is a projective scheme and $Y$ is a normal projective variety. For any Cartier divisors $D_1, \ldots , D_{e+i}$ be some ample Cartier divisors on $X$, there exist vector bundles $E_j$, and a homogeneous polynomial $c=P(c_{i_1}(E_1),  \ldots , c_{i_p}(E_p))$ of degree $i$ with respect to the weight $(i_1, \ldots , i_p)$, with rational coefficients such that for any cycle $z \in Z_i(X)$, $(c \cdot z ) = (D_1 \cdot \ldots \cdot D_{e+i} \cdot q^* z)$.
\end{prop_app}

\begin{proof}

We take some ample Cartier divisors $D_1 , \ldots , D_{e+i} $ on $X$. We denote by $\mathcal{L}_i$ the line bundle $\mathcal{O}_X(D_i)$.
By Grauert's Theorem (cf \cite[Corollary 12.9]{hartshorne}), the sheaves $R^i q_*(\mathcal{L}_{1}^{m_1} \otimes \ldots \otimes \mathcal{L}_{e+i}^{m_{e+i}})$ are locally free.
By \cite[Theorem 8.8]{hartshorne}, we have that $R^i q_* (\mathcal{L}_{1}^{m_1} \otimes \ldots \otimes \mathcal{L}_{e+i}^{m_{e+i}}) = 0$ for $i>0$ and $m_i$ large enough since the line bundle $\mathcal{L}_{i}$ are ample.
So the sheaf $q_* (\mathcal{L}_1^{m_1} \otimes \ldots \otimes \mathcal{L}_{e+i}^{m_{e+i}})$ is locally free and we have in $K_0(Y)$:
\begin{equation} \label{eq_K0}
q_* [\mathcal{L}_1^{m_1} \otimes \ldots \otimes {\mathcal{L}}_{e+i}^{m_{e+i}}] = \sum (-1)^{i} [R^i q_* (\mathcal{L}_1^{m_1} \otimes \ldots \otimes {\mathcal{L}}_{e+i}^{m_{e+i}})] =[ q_* (\mathcal{L}_1^{m_1} \otimes \ldots \otimes {\mathcal{L}}_{e+i}^{m_{e+i}})].
\end{equation}

\begin{lem_app} For any $j \leqslant i $:
\begin{enumerate}
\item The function $(m_1, \ldots, m_{e+i}) \rightarrow \Ch_j ( q_* (\mathcal{L}_1^{m_1} \otimes \ldots \otimes {\mathcal{L}}_{e+i}^{m_{e+i}})) \in \numd^j(Y)_\mathbb{R}$ is a polynomial of degree $e+j$ with coefficients in $\numd^j(Y)$.
\item For any cycle $z \in Z_j(Y)$, the coefficient in $m_1 \cdot \ldots \cdot m_{e+i}$ in $(\Ch_j(q_* (\mathcal{L}_1^{m_1} \otimes \ldots \otimes {\mathcal{L}}_{e+i}^{m_{e+i}})) \actd  z)$ is $ ((D_1 \cdot \ldots \cdot  D_{e+i}) \actd q^* z)$.
\end{enumerate}
\end{lem_app}

\begin{proof} Let us set $\mathcal{F} = \mathcal{L}_1^{m_1} \otimes \ldots \otimes {\mathcal{L}}_{e+i}^{m_{e+i}}$. We prove the result by induction on $0 \leqslant j \leqslant i$. 
\smallskip

For $j = 0$, choosing a point $y \in Y(\C)$, the number $\Ch_0(q_* (\mathcal{F}))$ is equal to $h^0( X_y , \mathcal{F}_{|X_y} )$. 
By asymptotic Riemann-Roch, for $m_1, \ldots, m_{e+i}$ large enough, it is a polynomial of degree $\dim X_y = e$. Moreover, Snapper's theorem (see \cite[Definition 1.7]{debarre_higher}) states that the coefficient in $m_1 \cdot \ldots \cdot m_{e+i} $ is the number $(D_1 \cdot \ldots \cdot D_{e+i} \actd [X_y])$.

\medskip
We suppose by induction that $\Ch_i( q_* (\mathcal{F}))$ is a polynomial of degree $e +i$ for any $i \leqslant j$ where $j \leqslant i-1$. For any subvariety $V$ of dimension $j+1$ in $Y$, we denote by $W$ its scheme-theoretic preimage by $q$.

\smallskip 
For any scheme $V$, let us denote by $\tau_V$ the morphisms:
\begin{equation*}
\tau_V : K_0(V) \otimes \mathbb{Q} \to A_\bullet (V) \otimes \mathbb{Q}.
\end{equation*} 
We refer to \cite[Theorem 18.3]{fulton} for the construction of this morphism and its properties. 
We apply Grothendieck-Riemann-Roch's theorem for singular varieties (see \cite[Theorem 18.3.(1)]{fulton}) and using \eqref{eq_K0}, we get in $A_\bullet(Y)_\mathbb{Q}$: 
\begin{equation} \label{eq_GRR} 
\Ch( q_* (\mathcal{L}_1^{m_1} \otimes \ldots \otimes {\mathcal{L}}_{e+i}^{m_{e+i}}) ) \actd \tau_V( \mathcal{O}_V ) =  q_* ( \Ch (\mathcal{L}_1^{m_1} \otimes \ldots \otimes {\mathcal{L}}_{e+i}^{m_{e+i}}) \actd \tau_W(\mathcal{O}_W)). 
\end{equation}
The term in $A_0(Y)_\mathbb{Q}$ in the left handside of the previous equation is equal to:
\begin{equation*}
\Ch_{j+1} ( q_* (\mathcal{F})) \actd [V] + \sum_{i \leqslant j} \Ch_i ( q_*(\mathcal{F})) \actd \tau_{V,i} (\mathcal{O}_V), 
\end{equation*}
where $\tau_{V,i} (\mathcal{O}_V)$ is the term in $A_i(Y)$ of $\tau_V(\mathcal{O}_V)$.
By the induction hypothesis, every $\Ch_i(q_* \mathcal{F})$ is a polynomial of degree $e+i$, and the right hand side of equation \eqref{eq_GRR} is a polynomial of degree $e+j + 1$, so $\Ch_{j+1}(q_*( \mathcal{L}_1^{m_1} \otimes \ldots \otimes {\mathcal{L}}_{e+i}^{m_{e+i}}))$ is also a polynomial of degree $e + j + 1$. Now we identify the coefficients in $m_1 \cdot \ldots \cdot m_{e+i}$ of the term in $\num_0(Y)$ in equation \eqref{eq_GRR}. 
It follows from \cite[example 18.3.11]{fulton} that $\tau_W(\mathcal{O}_W) = [W] + R_W$ where  $R_W$ is a linear combination of cycles of dimension $< e+i $. Therefore, the coefficient in $m_1 \cdot \ldots \cdot m_{e+i}$ of the right hand side of equation \eqref{eq_GRR} in $\num_0(Y)$ is $((D_1 \cdot \ldots \cdot D_{e+i}) \actd [W])$ if $j+1 = i$ or $0$ otherwise. 
 
We have proved that the coefficient of $\Ch_{j+1}(q_* ( \mathcal{L}_1^{m_1} \otimes \ldots \otimes {\mathcal{L}}_{e+i}^{m_{e+i}})) \actd [V]$ is $((D_1 \cdot \ldots \cdot D_{e+i}) \actd [W])$ if $\dim V = i$ or $0$ otherwise. Extending it by linearity, one gets the desired result.
 
\end{proof}

We have that $\Ch_i( q_* (\mathcal{L}_1^{m_1} \otimes \ldots \otimes \mathcal{L}_{e+i}^{m_{e+i}}))$ is by definition a polynomial in Chern classes of vector bundles on $Y$. Using the previous lemma, the coefficient $U(D_1, \ldots , D_{e+i})$ in $m_1 \cdot \ldots \cdot m_{e+i}$ of $\Ch_i( q_* (\mathcal{L}_1^{m_1} \otimes \ldots \otimes \mathcal{L}_{e+i}^{m_{e+i}}))$ is equal to $P(c_{i_1}(E_1) , \ldots , c_{i_p}(E_p))$ where $P$ is a homogeneous polynomial with rational coefficients of degree $i$ with respect to the weight $(i_1, \ldots, i_p)$ and $E_i$ are vector bundles on $Y$. We have proven that for any cycle $z \in Z_i(Y)$: 
\begin{equation*}
(P(c_{i_1}(E_1) , \ldots ,c_{i_p}(E_p) ) \actd z ) = ((D_1 \cdot \ldots \cdot D_{e+i} )\actd q^*z).
\end{equation*}
As any Cartier divisor can be written as a difference of ample Cartier divisors. The proposition provides a proof for the implication $(i) \Rightarrow (ii)$ of Theorem \ref{thm_pull_num_class}. 

\end{proof}

\begin{rem_app} In codimension $1$, the intersection product $(D_1 \cdot \ldots \cdot D_{e+1} \actd q^*z) $ is represented by Deligne's product $I_X(\mathcal{O}_X(D_1), \ldots , , \mathcal{O}_X(D_{e+1})) \in \numd^1(X)_\mathbb{R}$ (see \cite{garcia} for a reference). Indeed, one has by \cite[Section 6]{garcia} that for any cycle $z \in \num_1(X)$: 
\begin{equation*}
c_1( I_X( \mathcal{O}_X(D_1), \ldots , , \mathcal{O}_X(D_{e+1})) ) \actd z =  D_1 \cdot \ldots \cdot D_{e+1} \actd q^*z.
\end{equation*}
 
\end{rem_app}

This gives an answer to the question of numerical pullback formulated in \cite[section 1.2]{fulger_lehmann}.
\begin{cor_app} \label{cor_push_num} Let $q: X \to Y$ be a flat morphism of relative dimension $e$ between normal projective varieties. Then the morphism $q^* : A_{\bullet}(Y)_\mathbb{Q} \to A_{e+ \bullet}(X)_\mathbb{Q}$ induces a morphism of abelian groups $q^* : \numd_\bullet(Y)_\mathbb{Q} \to \num_{e+\bullet}(X)_\mathbb{Q}$. 
By duality, the morphism $q_* : A^{\bullet}(X)_\mathbb{Q} \to A^{\bullet - e}(Y)_\mathbb{Q} $ induces a morphism of abelian groups $q_*: \numd^{\bullet}(X)_\mathbb{Q} \to \numd^{\bullet - e}(Y)_\mathbb{Q}$.
\end{cor_app}


\bibliographystyle{amsalpha}
\bibliography{ref_degre_dynamique}

\newcommand{\etalchar}[1]{$^{#1}$}
\providecommand{\bysame}{\leavevmode\hbox to3em{\hrulefill}\thinspace}
\providecommand{\MR}{\relax\ifhmode\unskip\space\fi MR }
\providecommand{\MRhref}[2]{%
  \href{http://www.ams.org/mathscinet-getitem?mr=#1}{#2}
}
\providecommand{\href}[2]{#2}
\begin{thebibliography}{BdFFU15}

\bibitem[BdFFU15]{boucksom_fernex_favre_urbinati}
Sebastien Boucksom, Tomasso de~Fernex, Charles Favre, and Stefano Urbinati,
  \emph{Valuation spaces and multiplier ideals on singular varieties}, Recent
  advances in algebraic geometry, London Math. Soc. Lecture Note Ser., vol.
  417, Cambridge Univ. Press, Cambridge, 2015, pp.~29--51. \MR{3380442}

\bibitem[BFJ08]{boucksom_favre_jonsson_deggrowth}
S{\'e}bastien Boucksom, Charles Favre, and Mattias Jonsson, \emph{Degree growth
  of meromorphic surface maps}, Duke Math. J. \textbf{141} (2008), no.~3,
  519--538. \MR{2387430}

\bibitem[CC15]{chen_coskun}
Dawei Chen and Izzet Coskun, \emph{Extremal higher codimension cycles on moduli
  spaces of curves}, Proc. Lond. Math. Soc. (3) \textbf{111} (2015), no.~1,
  181--204. \MR{3404780}

\bibitem[CLO16]{coskun_lesieutre_ottem}
Izzet Coskun, John Lesieutre, and John~Christian Ottem, \emph{Effective cones
  of cycles on blow-ups of projective space}, arXiv preprint 1603.04808 (2016).

\bibitem[Cut15]{cutkosky_teissier_pb}
Steven~Dale Cutkosky, \emph{Teissier's problem on inequalities of nef
  divisors}, J. Algebra Appl. \textbf{14} (2015), no.~9, 1540002, 37.
  \MR{3368254}

\bibitem[Deb01]{debarre_higher}
Olivier Debarre, \emph{Higher-dimensional algebraic geometry}, Universitext,
  Springer-Verlag, New York, 2001. \MR{1841091}

\bibitem[DELV11]{debarre_ein_lazarsfeld_voisin}
Olivier Debarre, Lawrence Ein, Robert Lazarsfeld, and Claire Voisin,
  \emph{Pseudoeffective and nef classes on abelian varieties}, Compos. Math.
  \textbf{147} (2011), no.~6, 1793--1818. \MR{2862063}

\bibitem[DN11]{dinh_nguyen_2011}
Tien-Cuong Dinh and Vi{\^e}t-Anh Nguy{\^e}n, \emph{Comparison of dynamical
  degrees for semi-conjugate meromorphic maps}, Comment. Math. Helv.
  \textbf{86} (2011), no.~4, 817--840. \MR{2851870}

\bibitem[DNT12]{dinh_nguyen_truong}
Tien-Cuong Dinh, Vi{\^e}t-Anh Nguy{\^e}n, and Tuyen~Trung Truong, \emph{On the
  dynamical degrees of meromorphic maps preserving a fibration}, Communications
  in Contemporary Mathematics \textbf{14} (2012), no.~06, 1250042.

\bibitem[DS04]{dinh_sibony_reg_currents_entropy}
Tien-Cuong Dinh and Nessim Sibony, \emph{Regularization of currents and
  entropy}, Ann. Sci. \'Ecole Norm. Sup. (4) \textbf{37} (2004), no.~6,
  959--971. \MR{2119243}

\bibitem[DS05]{dinh_sibony_une_borne_sup}
\bysame, \emph{Une borne sup\'erieure pour l'entropie topologique d'une
  application rationnelle}, Ann. of Math. (2) \textbf{161} (2005), no.~3,
  1637--1644. \MR{2180409}

\bibitem[Eis95]{eisenbud_commutative}
David Eisenbud, \emph{Commutative algebra}, Graduate Texts in Mathematics, vol.
  150, Springer-Verlag, New York, 1995, With a view toward algebraic geometry.
  \MR{1322960}

\bibitem[FGI{\etalchar{+}}05]{FGA}
B.~Fantechi, L.~Göttsche, L.~Illusie, S.~Kleiman, N.~Nitsure, and A.~Vistoli,
  \emph{Fundamental algebraic geometry: Grothendieck’s fga explained},
  Mathematical Surveys and Monographs, 2005.

\bibitem[FL14a]{fulger_lehmann_kernel}
Mihai Fulger and Brian Lehmann, \emph{Kernels of numerical pushforwards}, arXiv
  preprint 1407.6455 (2014).

\bibitem[FL14b]{fulger_lehmann}
\bysame, \emph{Positive cones of dual cycle classes}, arXiv preprint 1408.5154
  (2014).

\bibitem[Ful93]{fulton_toric}
William Fulton, \emph{Introduction to toric varieties}, Annals of Mathematics
  Studies, vol. 131, Princeton University Press, Princeton, NJ, 1993, The
  William H. Roever Lectures in Geometry. \MR{1234037}

\bibitem[Ful98]{fulton}
\bysame, \emph{Intersection theory}, second ed., Ergebnisse der Mathematik und
  ihrer Grenzgebiete. 3. Folge. A Series of Modern Surveys in Mathematics
  [Results in Mathematics and Related Areas. 3rd Series. A Series of Modern
  Surveys in Mathematics], vol.~2, Springer-Verlag, Berlin, 1998. \MR{1644323}

\bibitem[Gar00]{garcia}
Elvira~Mu{\~n}{\'o}z Garc{\'{\i}}a, \emph{Fibr\'es d'intersection}, Compositio
  Math. \textbf{124} (2000), no.~3, 219--252. \MR{1809336}

\bibitem[GGJ{\etalchar{+}}16]{gil_gubler_kunneman_lazarsfeld_2016}
Jos{\'e} Ignacio~Burgos Gil, Walter Gubler, Philipp Jell, Klaus Kuennemann, and
  Florent Martin, \emph{Differentiability of non-archimedean volumes and
  non-archimedean monge-amp$\backslash$ere equations (with an appendix by
  robert lazarsfeld)}, arXiv preprint 1608.01919 (2016).

\bibitem[Har77]{hartshorne}
Robin Hartshorne, \emph{Algebraic geometry}, vol.~52, Springer Science \&
  Business Media, 1977.

\bibitem[Jon96]{de_jong}
Aise Johan~De Jong, \emph{Smoothness, semi-stability and alterations},
  Publications Mathématiques de l'IHÉS (1996).

\bibitem[KM98]{kollar_mori}
J{\'a}nos Koll{\'a}r and Shigefumi Mori, \emph{Birational geometry of algebraic
  varieties}, Cambridge Tracts in Mathematics, vol. 134, Cambridge University
  Press, Cambridge, 1998, With the collaboration of C. H. Clemens and A. Corti,
  Translated from the 1998 Japanese original. \MR{1658959}

\bibitem[Laz04]{lazarsfeld_positivity_1}
Robert Lazarsfeld, \emph{Positivity in algebraic geometry. {I}}, Ergebnisse der
  Mathematik und ihrer Grenzgebiete. 3. Folge. A Series of Modern Surveys in
  Mathematics [Results in Mathematics and Related Areas. 3rd Series. A Series
  of Modern Surveys in Mathematics], vol.~48, Springer-Verlag, Berlin, 2004,
  Classical setting: line bundles and linear series. \MR{2095471}

\bibitem[Lef53]{lefschetz_algebraic_geometry}
Solomon Lefschetz, \emph{Algebraic geometry}, Princeton University Press,
  Princeton, N. J., 1953. \MR{0056950}

\bibitem[LX15]{lehmann_xiao}
Brian Lehmann and Jian Xiao, \emph{Convexity and zariski decomposition
  structure}.

\bibitem[Mil13]{milneLEC}
James~S. Milne, \emph{Lectures on etale cohomology (v2.21)}, 2013, Available at
  www.jmilne.org/math/.

\bibitem[Pop16]{popovici}
Dan Popovici, \emph{Sufficient bigness criterion for differences of two nef
  classes}, Math. Ann. \textbf{364} (2016), no.~1-2, 649--655. \MR{3451400}

\bibitem[RS97]{russakovskii_shiffman}
Alexander Russakovskii and Bernard Shiffman, \emph{Value distribution for
  sequences of rational mappings and complex dynamics}, Indiana Univ. Math. J.
  \textbf{46} (1997), no.~3, 897--932. \MR{1488341}

\bibitem[Tra95]{trapani}
Stefano Trapani, \emph{Numerical criteria for the positivity of the difference
  of ample divisors}, Math. Z. \textbf{219} (1995), no.~3, 387--401.
  \MR{1339712}

\bibitem[Tru15]{tuyen1}
Tuyen~Trung Truong, \emph{(relative) dynamical degrees of rational maps over an
  algebraic closed field}, arXiv preprint 1501.01523 (2015).

\bibitem[Tru16]{tuyen2}
\bysame, \emph{Relative dynamical degrees of correspondance over a field of
  arbitrary characteristic}, arXiv (2016).

\bibitem[Xia15]{xiao}
Jian Xiao, \emph{Weak transcendental holomorphic {M}orse inequalities on
  compact {K}\"ahler manifolds}, Ann. Inst. Fourier (Grenoble) \textbf{65}
  (2015), no.~3, 1367--1379. \MR{3449182}

\end{thebibliography}


\affiliationone{
   Nguyen-Bac Dang\\
   CMLS, \'Ecole polytechnique, CNRS, Universit\'e Paris-Saclay, 91128 Palaiseau Cedex,  \\
   France
   \email{nguyen-bac.dang@polytechnique.edu}}

\end{document}